\documentclass{article}

\usepackage[margin=1in]{geometry}
\usepackage{parskip}

\usepackage[T1]{fontenc}
\usepackage{mathtools} 
\usepackage{amssymb,amsthm}

\usepackage{aliascnt}
\usepackage{hyperref}
\usepackage{orcidlink}
\usepackage{graphicx}
\usepackage{caption}
\usepackage[inline]{enumitem}
\setlist[enumerate]{label=\textit{(\roman*)}}
\usepackage{xcolor}
\hypersetup{
	colorlinks=true,
	linkcolor=blue,
	citecolor=blue,
	urlcolor=blue
}

\usepackage{authblk}

\usepackage{tikz}
\usetikzlibrary{positioning,math,patterns.meta}

\theoremstyle{plain}
\newtheorem{theorem}{Theorem}[section]
\newaliascnt{prop}{theorem}
\newtheorem{prop}[prop]{Proposition}
\aliascntresetthe{prop}
\newaliascnt{lemma}{theorem}
\newtheorem{lemma}[lemma]{Lemma}
\aliascntresetthe{lemma}
\newaliascnt{corollary}{theorem}
\newtheorem{corollary}[corollary]{Corollary}
\aliascntresetthe{corollary}

\newtheorem{maintheorem}{Theorem} 

\newaliascnt{maincorollary}{maintheorem}
\newtheorem{maincorollary}[maincorollary]{Corollary} 
\aliascntresetthe{maincorollary}

\theoremstyle{definition}
\newaliascnt{defn}{theorem}
\newtheorem{defn}[defn]{Definition}
\aliascntresetthe{defn}
\newaliascnt{example}{theorem}
\newtheorem{example}[example]{Example}
\aliascntresetthe{example}

\theoremstyle{remark}
\newaliascnt{remark}{theorem}
\newtheorem{remark}[remark]{Remark}
\aliascntresetthe{remark}
\newtheorem*{remark*}{Remark} 

\usepackage[nameinlink,noabbrev,capitalise]{cleveref}
\Crefname{prop}{Proposition}{Propositions}
\Crefname{defn}{Definition}{Definitions}
\Crefname{maintheorem}{Theorem}{Theorems}
\Crefname{maincorollary}{Corollary}{Corollaries}
\Crefname{subsection}{Subsection}{Subsections}

\newcommand{\bbE}{\mathbb{E}}
\newcommand{\bbP}{\mathbb{P}}
\newcommand{\cD}{\mathcal{D}}
\newcommand{\cF}{\mathcal{F}}
\newcommand{\cM}{\mathcal{M}}

\DeclareMathOperator{\Kn}{\mathbf{K}}
\DeclareMathOperator{\dist}{dist}
\DeclareMathOperator{\med}{med}
\DeclareMathOperator{\gap}{gap}
\DeclareMathOperator{\hyp}{hyp}
\DeclareMathOperator{\diam}{diam}
\DeclareMathOperator{\supp}{supp}
\DeclareMathOperator{\blocks}{\mathcal{B}}
\DeclareMathOperator{\cut}{cut}
\DeclareMathOperator{\bctree}{BC}
\DeclareMathOperator{\dis}{dis}
\DeclareMathOperator{\vr}{VR}
\DeclareMathOperator{\summ}{sum}

\newcommand{\R}{\mathbb{R}}
\newcommand{\Z}{\mathbb{Z}}
\newcommand{\Sp}{\mathbb{S}}

\newcommand{\W}{\mathsf{W}}
\newcommand{\GH}{\mathsf{GH}}

\newcommand{\Dvr}[1]{\mathbf{D}_{#1}^{\mathrm{VR}}}
\newcommand{\Uvr}[1]{\mathbf{U}_{#1}^{\mathrm{VR}}}
\newcommand{\dgmVR}{\operatorname{dgm}^{\mathrm{VR}}}

\newcommand{\empirical}[3]{\hat{#1}_{#2,#3}}
\newcommand{\setcolon}{:}

\newcommand{\gp}[3]{(#1 \cdot #2)_{#3}}

\title{Distributions of metric $n$-point invariants\\ and applications to molecular evolution}
\date{}

\author[1]{Michael Bleher\,\orcidlink{0000-0002-7796-1665}}
\author[2]{Mario G{\'o}mez\,\orcidlink{0000-0001-6143-8995}}
\author[3]{Facundo M{\'e}moli\,\orcidlink{0000-0001-8409-0549}}
\author[1,4]{Maximilian Neumann\,\orcidlink{0009-0004-8028-1160}}
\author[1]{Andreas Ott\,\orcidlink{0000-0002-2342-6907}}

\affil[1]{Institute for Mathematics, Heidelberg University, Germany}
\affil[2]{Department of Mathematics, Florida State University, USA}
\affil[3]{Department of Mathematics, Rutgers University, USA}
\affil[4]{Institute for Biological Interfaces 5, Karlsruhe Institute of Technology, Germany}
\affil[ ]{\footnotesize
	\href{mailto:mbleher@mathi.uni-heidelberg.de}{mbleher@mathi.uni-heidelberg.de},
	\href{mailto:mrgomez304@gmail.com}{mrgomez304@gmail.com},
	\href{mailto:facundo.memoli@rutgers.edu}{facundo.memoli@rutgers.edu},
	\href{mailto:mneumann@mathi.uni-heidelberg.de}{mneumann@mathi.uni-heidelberg.de},
	\href{mailto:aott@mathi.uni-heidelberg.de}{aott@mathi.uni-heidelberg.de}
	}

\begin{document}

\maketitle

\begin{NoHyper}
	\renewcommand{\thefootnote}{}%
\end{NoHyper}

\begin{abstract}
Datasets often arise as samples from a metric measure space whose geometry reflects the data-generating process.
We introduce metric $n$-point invariants, defined as functions of the pairwise distances of $n$-point configurations, to study the geometry of such spaces.
This provides a unified framework for existing concepts, such as the pairwise distance itself, the Gromov product, and Vietoris--Rips persistence diagrams.
The distribution of a metric $n$-point invariant $F$ over random $n$-point configurations can be used to extract geometric information from the space and is obtained as the pushforward of Gromov's $n$-th curvature measure $\mu_n$ along~$F$.
Our main result is an explicit bound for the expected $p$-Wasserstein distance between $\mu_n$ and the empirical curvature measure $\hat{\mu}_{n}$ of finitely many sampled $n$-point configurations, and carries over to the associated distributions of Lipschitz metric $n$-point invariants.
From this we derive the sample size needed to achieve a prescribed accuracy and level of confidence in the approximation of the curvature measure.
As an application, we use metric $n$-point invariants to study molecular evolution.
We describe how site-wise nucleotide frequencies shape the corresponding distributions in Jukes--Cantor and Halpern--Bruno substitution models, and analyse five million SARS\nobreakdash-CoV\nobreakdash-2 genomes, revealing increasingly common failures of the four-point condition, constrained nucleotide accessibility consistent with the biased mutation spectrum, and rare cycles in persistent homology reflecting recurrent mutations across lineages.
These findings demonstrate the potential of our sampling-based approach to the geometry of data as a scalable, phylogeny-free tool for detecting and quantifying evolutionary structure in large genomic datasets.
\end{abstract}

\section{Introduction}

Let $(X,d_X,\mu_X)$ be a bounded metric measure space (mm-space), meaning that $(X,d_X)$ is a Polish space of finite diameter $\diam(X)$ and $\mu_X$ is a Borel probability measure on $X$, not necessarily fully supported.
Many classical summaries of the geometry of $X$ are obtained by evaluating the distances between pairs, triples, or $n$-tuples of points and then taking an average or an extremum.
Instead of reducing these values to a single number, one can retain their full distribution over random configurations drawn from $\mu_X$.
This distributional viewpoint is particularly natural when $X$ is accessible only through samples, either because $X$ itself is unknown or because it is too large to compute all pairwise distances, as is increasingly the case for modern datasets.

To treat all quantities that depend only on the distances among $n$ points in a unified way, we follow Gromov \cite{Gromov2007} and assign to an $n$-tuple $\mathbf{x}=(x_1,\ldots,x_n)$ its distance matrix:
\begin{align*}
	\Psi_n \colon X^n \longrightarrow \R^{n\times n},
	\quad
	\Psi_n(\mathbf{x})_{ij} \coloneq d_X(x_i,x_j).
\end{align*}
The image of this map, $\Kn_n(X)\coloneq\Psi_n(X^n)$, is known as the $n$-th curvature set (\Cref{subsec:prelim-curvature}), and the pushforward
\begin{align*}
	\mu_n &\coloneq (\Psi_n)_\#\mu_X^{\otimes n}
\end{align*}
is called the $n$-th curvature measure.
The measure $\mu_n$ records the geometry of $X$ as seen by random $n$-point configurations in $X$.
The full family $(\mu_{n})_{n\geq 2}$ determines the mm-space up to the natural measure-preserving notion of isometry by the Gromov--Vershik reconstruction theorem \cite[Sec.~$3\frac{1}{2}.5$ \& $3\frac{1}{2}.7$]{Gromov2007}; see also \cite[Cor.~5]{Vershik2004}.

Write $\Kn_n\subset\R^{n\times n}$ for the set of all distance matrices on $n$ points, also known as the \emph{$n$-th metric cone}.
We define a \emph{metric $n$-point invariant} with values in a Polish space $Z$ , or \emph{$n$-invariant} for short (\Cref{sec:n-point-invariants}), to be a measurable function
\begin{align*}
	F\colon\Kn_n\to Z .
\end{align*}
In other words, an $n$-invariant maps the pairwise distances among $n$ points to a value in $Z$.
Examples include the pairwise distance, the area of a triangle, the four-point hyperbolicity deficit, and the persistence diagram of $n$ points, each capturing a particular geometric property of an $n$-point configuration.

The distribution of $F$ over random $n$-tuples of $X$ is the probability measure on $Z$ given by the pushforward of the curvature measure of $X$,
\begin{align*}
	\nu_F &\coloneq F_\#\mu_n.
\end{align*}
Throughout, we suppress the dependence of $\mu_n$ and $\nu_F$ on the mm-space $(X,d_X,\mu_X)$ whenever it is clear from context and record it in the subscript otherwise, as in $\mu_{n,X}$, $\nu_{F,X}$, or $\nu_{F,\mu}$.

In data analysis, the geometry of $X$ is expected to carry information about the data-generating process and is often studied through the distributions of $n$-invariants and their statistical summaries, such as the mean, median, or quantiles.
Each $\nu_F$ is simpler and often more tractable than $\mu_n$, but describes only one aspect of the geometry of $n$-point configurations.
In this article, we provide sample sizes that guarantee the approximation of $\mu_n$ itself from samples of $n$-tuples and show that they carry over to distributions of all Lipschitz $n$-invariants, including those with high- or infinite-dimensional target spaces, such as persistence diagrams.
Given independent samples $\mathbf{x}^{(1)},\ldots,\mathbf{x}^{(N)}\sim\mu_X^{\otimes n}$, denote the empirical curvature measure by
\begin{align*}
	\empirical{\mu}{n}{N}
	&\coloneq \frac{1}{N}\sum_{i=1}^N\delta_{\Psi_n(\mathbf{x}^{(i)})},
\end{align*}
where we write $\delta_y$ for the Dirac measure at a point $y$.
For every $n$-invariant $F$, the same sample then yields $\empirical{\nu}{F}{N}\coloneq F_\#\empirical{\mu}{n}{N}$, the empirical distribution of the values of $F$ on the sampled $n$-tuples.

\paragraph{Uniform approximation of the geometry of $n$-point configurations.}
Our main result is an explicit bound for the expected $p$-Wasserstein distance $\W_p$ between the empirical curvature measure $\empirical{\mu}{n}{N}$ and the true curvature measure $\mu_n$.
It is obtained by applying the Euclidean empirical-measure convergence rates of Fournier and Guillin \cite{convergence_empirical_fournier}, with the explicit constants due to Fournier \cite{Fournier2023}, to $\mu_n$ regarded as a measure on $\R^d$.
\begin{maintheorem}
	\label{thm:main-sampling}
	Fix $1\leq p<\infty$ and $n\geq2$, set $d\coloneq\binom{n}{2}$, and suppose that $p\neq d/2$.
	Let $(X,d_X,\mu_X)$ be a bounded mm-space with fully supported probability measure.
	Then the expected $p$-Wasserstein distance between the empirical curvature measure $\empirical{\mu}{n}{N}$ and the true curvature measure $\mu_n$ satisfies
	\begin{align*}
		\bbE\!\left[\W_p(\empirical{\mu}{n}{N},\mu_n)^p\right]
		\leq C(d,p)\diam(X)^pN^{-e_{d,p}},
		\quad e_{d,p}
		\coloneq
		\begin{cases}
			1/2, & p>d/2,\\
			p/d, & p<d/2,
		\end{cases}
	\end{align*}
	where $C(d,p) \coloneq \dfrac{2^{d/2-1}}{1 - 2^{d/2-p}}$ if $d < 2p$ and $C(d,p) \coloneq \dfrac{2^{p-2p/d}(1 - 2^{-d/2})^{1-2p/d}}{1-2^{p-d/2}}$ if $d > 2p$.
\end{maintheorem}

\begin{maincorollary}
	\label{cor:main-sampling-invariant}
	Under the assumptions of \Cref{thm:main-sampling}, let $F \colon \Kn_n \to Z$ be a Lipschitz metric $n$-point invariant.
	Then
	\begin{align*}
		\bbE\!\left[\W_p\bigl(\empirical{\nu}{F}{N},\nu_F\bigr)^p\right]
		&\leq C(d,p)\operatorname{Lip}(F)^p\diam(X)^pN^{-e_{d,p}}.
	\end{align*}
\end{maincorollary}

\begin{maincorollary}
	\label{cor:main-sampling-invariant-bounds}
	Choose an error margin $\epsilon>0$ and a confidence level $0<\alpha<1$.
	Under the assumptions of \Cref{thm:main-sampling}, there is an explicit $N_\ast=N_\ast(n,p,\epsilon,\alpha)$ such that, for every $N>N_\ast$, with probability at least $\alpha$,
	\begin{align*}
		\W_p(\empirical{\mu}{n}{N},\mu_n)
		&\leq\epsilon\diam(X)
	\end{align*}
	and, uniformly over all Lipschitz metric $n$-point invariants $F\colon\Kn_n\to Z$,
	\begin{align*}
		\W_p\bigl(\empirical{\nu}{F}{N},\nu_F\bigr)
		&\leq\epsilon\operatorname{Lip}(F)\diam(X).
	\end{align*}
\end{maincorollary}

These results are proved in \Cref{sec:convergence-rates} as \Cref{thm:empirical_Kn,cor:empirical_Kn_invariant,cor:empirical_Kn_samples}.
The first corollary is just the pushforward along a Lipschitz function, while the second relies on a standard bounded-differences concentration argument and McDiarmid's inequality \cite{WeedBach2019}.

Approximating the full $n$-th curvature measure guarantees that the results in \Cref{cor:main-sampling-invariant,cor:main-sampling-invariant-bounds} hold uniformly over all Lipschitz $n$-invariants $F\colon\Kn_n\to Z$.
In the analysis of a concrete dataset, this allows us to determine the number $N_\ast$ of required $n$-tuples for a prescribed accuracy and confidence in advance,\footnote{A Python package that provides $N_\ast$ and computes the associated distributions of $n$-invariants from data is available at \url{https://github.com/subthaumic/n-invariants}.} and every subsequent analysis of the geometry of $n$-point configurations of $X$ based on a single sample of this size inherits the guarantee without further sampling or corrections for multiple comparisons.
This universality comes at a cost, as the convergence rates are governed by the rapidly growing dimension of the ambient Euclidean space that contains $\Kn_n$.
For a single $n$-invariant with a low-dimensional target, applying the same techniques to $\nu_F$ directly can substantially improve the convergence rate.
For example, for a real-valued $n$-invariant $F\colon\Kn_n\to\R$, these techniques yield the rate $N^{-1/2}$ (cf.\ \Cref{prop:empirical_real_invariant}), in line with classical results, e.g.\ \cite[Thm.~7.9]{BobkovLedoux2019}.
However, if $Z$ is high- or infinite-dimensional, as for persistence diagrams, such bounds may be poor or unavailable, and passing through $\mu_n$ is the natural route to explicit rates.

The ambient dimension $d=\binom{n}{2}$ of $\Kn_n \subset \mathbb{R}^d$ is often suboptimal since the curvature set $\Kn_n(X)$ in question may have much smaller intrinsic dimension.
Convergence rates for empirical measures in terms of intrinsic dimension go back to Dudley \cite{Dudley69} and have since been further developed using covering dimensions of sets of full or almost full measure \cite{convergence_empirical_covering_dim_2014,WeedBach2019,convergence_empirical_covering_dim_2026}.
We obtain the following refinement for Ahlfors $s$-regular spaces from bounds based on an intrinsic Wasserstein-dimension of measures due to Weed and Bach \cite{WeedBach2019}.
\begin{maintheorem}
	\label{thm:main-sampling-intrinsic}
	Fix $n\geq2$. Let $(X,d_X,\mu_X)$ be an Ahlfors $s$-regular mm-space.
	Then for all $1 \leq p \leq \frac{1}{2} d_M(\Kn_n(X))$ and $t>s$ there are constants $C_1(n,p,t)$ and $C_2(X,n,p,t)$ such that for all $N \geq 1$
	\begin{align*}
		\bbE\!\left[\W_p(\empirical{\mu}{n}{N},\mu_n)^p\right]
		&\leq \diam(X)^p \left( C_1 N^{-p/{nt}} + C_2 N^{-1/2} \right),
	\end{align*}
	where $d_M$ is the Minkowski dimension.
\end{maintheorem}
This result is proved in \Cref{sec:convergence-rates} as \Cref{thm:empirical_Kn_intrinsic_WB}.
The exponent $p/nt$ can improve upon $e_{d,p}=p/d$ whenever $s<(n-1)/2$.
Unlike $C(d,p)$, however, the constant $C_2$ depends on $X$ through a parameter that is not explicit, so \Cref{thm:main-sampling-intrinsic} in its current form cannot yield improvements in sample sizes.

\paragraph{Hyperbolicity distributions.}
The distribution of hyperbolicity deficits provides an interesting example of the information that may be carried by the distribution of an $n$-invariant.
The four-point hyperbolicity deficit $\hyp\colon\Kn_4\to\R$ (\Cref{inv:hyperbolicity}) measures how far a given quadruple is from satisfying the four-point condition for points on a tree \cite{Buneman1974}.
Its supremum over all quadruples of a metric space, known as its Gromov hyperbolicity \cite{Gromov1987}, detects obstructions to the space being a tree.
For a finite metric space $(X,d_X)$, the tight span $T(X,d_X)$ is a canonical geodesic metric space containing $X$ with the structure of a finite polytopal complex \cite[Sec.~3]{lang-injective-hulls}.
Splitting the tight span at its cut-vertices, the vertices whose removal disconnects it, yields a set of blocks $\blocks(X)$ (cf.~\cite{bc-tree}).
By a classical result of Dress \cite[Thm.~8]{dress-tight-span}, the deficit vanishes on all quadruples of $X$ if and only if $T(X,d_X)$ is a tree, i.e.\ all blocks $B\in\blocks(X)$ are intervals.
For a finite mm-space $(X,d_X,\mu_X)$, we may view each block $B$ as an mm-space with the measure $\mu_B\coloneq(P_B)_\#\mu_X$, where $P_B$ maps each point outside $B$ onto the closest point of $B$.
If $\mu_X$ has full support, Dress's theorem then amounts to the statement that $\nu_{\hyp,X}=\delta_0$ if and only if all blocks are intervals, which have $\nu_{\hyp,B}=\delta_0$ as discussed above.
The following result describes how $\nu_{\hyp,X}$ is assembled from the $\nu_{\hyp,B}$ more generally.
\begin{maintheorem}
	\label{thm:main-block-decomposition}
	Let $(X,d_X,\mu_X)$ be a finite mm-space.
	Then, for every Borel set $D\subset\R$ with $0\notin D$,
		\begin{align*}
			\nu_{\hyp,X}(D) = \sum_{B\in\blocks(X)}\nu_{\hyp,B}(D).
		\end{align*}
	Moreover, $\nu_{\hyp,X}(\{0\}) = 1-\sum_{B\in\blocks(X)}\nu_{\hyp,B}(\R\setminus\{0\})$.
\end{maintheorem}
In other words, every violation of the four-point condition in $X$ comes from exactly one block that is not an interval, and the contributions of these blocks add up.
Consequently, if $\mu_X$ has full support, the Gromov hyperbolicity of $X$ is determined by a single block, while $\nu_{\hyp,X}$ retains the full hyperbolicity distribution of every block.
This result and its proof appear in \Cref{sec:blockwise-sampling} as \Cref{thm:block_decomposition_hyp}.

\paragraph{Application to molecular evolution of SARS-CoV-2.}
Molecular evolution studies how evolutionary processes, such as mutation, recombination, and natural selection, are reflected in the genome of organisms.
Genomic sequences are typically represented as aligned words over the alphabet of nucleotides $A,C,G,T$.
Together with the Hamming distance, which counts the number of differing sites, and equipped with the observed genome frequencies, the data form an mm-space (\Cref{sec:hamming-sitewise}).
We study three Lipschitz $n$-invariants on this space.
\begin{enumerate*}[label=(\roman*), itemjoin=\quad, itemjoin*=\quad]
	\item The \emph{pairwise-distance distribution}, known in population genetics as the \emph{mismatch distribution}, records genetic separation between sampled genomes, and its shape has long been used to study demographic history \cite{RogersHarpending1992}.
	\item The distribution of \emph{four-point hyperbolicity deficits}, records the frequency and magnitude of departures from the four-point condition for points on a tree \cite{Buneman1974}.
	\item The distribution of \emph{degree-one persistence diagrams of four-point samples}, known as the \emph{$(4,1)$-persistence measure} \cite{Blumberg2014,GomezMemoli2024}, records the prevalence and scales of cycles and has been studied as a signature of reticulate evolution \cite{LesnickRabadanRosenbloom2020}. 
\end{enumerate*}
Our sampling guarantees make it possible to quantify how prevalent these configurations are and at which genetic scales they occur, without reconstructing a tree or network and without even computing the full distance matrix of the data.
This is a particular advantage for the millions of severe acute respiratory syndrome coronavirus 2 (SARS\nobreakdash-CoV\nobreakdash-2) genomes sequenced during the pandemic.
At this scale, even the full distance matrix would occupy several terabytes, and the low divergence of the sequences further complicates reliable phylogenetic inference \cite{MorelEtAl2021}.

We first investigate the distributions one obtains from the uniform measure, the Jukes--Cantor model \cite{JukesCantor1969}, and the Halpern--Bruno model \cite{HalpernBruno1998} (\Cref{subsec:hamming-uniform,sec:simulated-evolution}).
In all three models, sites are independent, so the mismatch distribution is Poisson--binomial and depends only on how likely two random genomes are to differ at each site (\Cref{sec:hamming-sitewise}).
In particular, it records how variable each site is, but not whether this variability involves two or more nucleotides.
In contrast, the hyperbolicity and $(4,1)$-persistence distributions are sensitive to the frequencies of the individual nucleotides at each site.
For example, a half-integer hyperbolicity deficit can only occur if there is a site where the four genomes carry exactly three different nucleotides.

We then analyse a dataset of approximately five million SARS\nobreakdash-CoV\nobreakdash-2 genomes shared via GISAID, the global data science initiative \cite{Elbe2017}, and compute the three distributions in cumulative samples at five collection-date cutoffs, from the early outbreak to the era of Omicron recombinants (\Cref{subsec:sars-cov2}).
We find that over time the mismatch distribution shifts upward and becomes increasingly overdispersed and multimodal.
This shape is consistent with the accumulation of genetically separated lineage clusters as successive variants emerge.
The estimated fraction of quadruples violating the four-point condition rises from about $2\%$ at the early stages to $45\%$ towards the end of the pandemic.
This is in line with the many documented recurrences of mutations across lineages \cite{vanDorpEtAl2020,MartinEtAl2021}, since already a single mutation shared by genomes from two different lineages can cause such a violation.
Half-integer hyperbolicity deficits remain suppressed relative to the uniform reference, pointing to constrained triallelic accessibility and consistent with the strongly biased SARS\nobreakdash-CoV\nobreakdash-2 mutation spectrum \cite{Simmonds2020,DeMaioEtAl2021}.
Nontrivial $(4,1)$-persistence classes, which occur only in such non-tree-like quadruples, are much rarer, occurring in fewer than one in a thousand quadruples, but appear consistently at every collection-date cutoff.
Their birth times follow the dominant modes of the mismatch distribution as the lineage clusters separate, while their lifetimes remain short, with most classes having the minimal lifetime $1$.
In SARS\nobreakdash-CoV\nobreakdash-2, such short-lived $(4,1)$-persistence classes formed by single-nucleotide variants (SNVs) have been used to detect recurrent mutations within lineages and shown to correlate with adaptation \cite{BleherEtAl2021}.
The short-lived $(4,1)$-persistence classes observed here are their between-lineage counterpart, typically formed by a single SNV that arose independently in two lineage clusters.
The rarer classes with longer lifetimes record more pronounced cyclic geometry across several Hamming scales and reflect more complex reticulate configurations, including recombination, which has previously been documented to occur in SARS\nobreakdash-CoV\nobreakdash-2 \cite{TurakhiaEtAl2022}.
Taken together, these findings demonstrate the potential of our sampling-based approach to the geometry of $n$-point configurations as a scalable, phylogeny-free tool for detecting and quantifying evolutionary structure in large genomic datasets.

\paragraph{Organization.}
\Cref{sec:related-work} places $n$-invariants and their distributions in the context of the existing literature.
\Cref{sec:preliminaries} recalls standard definitions and preliminary results.
\Cref{sec:n-point-invariants} introduces $n$-invariants, their distributions, and the main examples.
\Cref{sec:geometry-Kn} establishes the geometric properties of $\Kn_n(X)$ that \Cref{sec:convergence-rates} uses to prove the main results \labelcref{thm:main-sampling}--\labelcref{thm:main-sampling-intrinsic}.
\Cref{sec:blockwise-sampling} proves \Cref{thm:main-block-decomposition} on the block decomposition of hyperbolicity distributions.
\Cref{sec:hamming} studies Hamming geometry, substitution models, and the SARS\nobreakdash-CoV\nobreakdash-2 data.

\section{Related work}
\label{sec:related-work}

Distributions of finite-configuration data have appeared under several names and in several areas.
The following non-exhaustive list illustrates some classical and recent uses of distributions of $n$-invariants in geometry, statistics, data analysis, and phylogenetics.

\paragraph{Distance distributions.}
The simplest example of an $n$-invariant is the pairwise distance $(n=2)$, whose distribution $\nu_{\dist,X}$ has long been studied in geometric probability \cite{Deltheil1926,Hammersley1950}.
Boutin and Kemper show that this distribution determines generic configurations of at least $m+2$ points in $\R^m$ up to isometry \cite{Boutin2004}.
Here genericity means that the distances lie outside the zero set of a constructed nonzero polynomial, a closed set of measure zero that contains, for instance, all configurations with a repeated distance and all configurations with a nontrivial symmetry.
In shape analysis, the same distribution is known as the $D2$ shape distribution \cite[Sec.~4.1]{OsadaEtAl2002}, and it is closely related to other classical statistics such as the $K$-function in spatial statistics \cite[Ch.~8]{Ripley1981}.
A classical way to represent $\nu_{\dist,X}$ is by its cumulative distribution function
\begin{align*}
	H_X(r) \coloneq \int_X \mu_X(B_X(x,r)) \, d\mu_X(x) = \nu_{\dist,X}([0,r]), \qquad r \geq 0.
\end{align*}
$H_X(r)$ records how often two sampled points lie within distance $r$, so it measures the average $\mu_X$-mass of an $r$-ball.
For Riemannian manifolds, the small-radius behaviour of $H_X$ is related to scalar curvature and $H_X$ similarly provides a signature of the curvature of metric measure spaces \cite[Def.~5.4]{Memoli2011}.
The third author and Needham systematically study how much geometric information is retained by global and local distance distributions beyond subsets of Euclidean space \cite{MemoliNeedham2022}.
Recently, cumulative distance distributions have found applications in parameter inference for systems with chaotic or heterogeneous output using the Correlation Integral Likelihood framework \cite{GwiazdaEtAl2025}.

\paragraph{Average hyperbolicity.}
Chatterjee and Sloman \cite{ChatterjeeSloman2021} recently proposed an average counterpart to Gromov hyperbolicity.
While the supremum of the four-point deficits $\hyp(\mathbf{x})$ (\Cref{inv:hyperbolicity}) detects whether a space looks everywhere like a tree, their average detects whether it looks `mostly' like a tree.
For a bounded measurable similarity function $s \colon S \times S \to [0,b]$ on a probability space $(S,\cF,\mu_S)$, they define
\begin{align*}
	\mathrm{Hyp}(S,\cF,\mu_S,s)
	&\coloneq \bbE\Bigl[\bigl(\min\{s(X_1,X_3),s(X_2,X_3)\}-s(X_1,X_2)\bigr)_+\Bigr],
\end{align*}
where $X_1,X_2,X_3 \sim \mu_S$ are i.i.d.\ and $(a)_+ \coloneq \max\{a,0\}$.
Thus $\mathrm{Hyp}$ is the average amount by which the inequalities $s(X_1,X_2) \geq \min\{s(X_1,X_3),s(X_2,X_3)\}$ fail.
For an mm-space $(X,d_X,\mu_X)$ and a fixed base point $w\in X$, take $s$ to be the Gromov product $s_w(x,x')\coloneq\gp{x}{x'}{w}=\frac{1}{2}(d_X(x,w)+d_X(x',w)-d_X(x,x'))$.
By symmetry, one obtains
\begin{align*}
	\mathrm{Hyp}(X,\mathcal B(X),\mu_X,s_w)
	&= \frac{1}{3}\,\bbE\bigl[\hyp(X_1,X_2,X_3,w)\bigr].
\end{align*}
The factor of $1/3$ arises because the average is over all three orderings of the Gromov products $\gp{x_1}{x_2}{w}$, $\gp{x_1}{x_3}{w}$, $\gp{x_2}{x_3}{w}$, and taking the positive part `$(\ldots)_+$' sets two of them to zero while the one non-zero case is exactly $\hyp(x_1,x_2,x_3,w)$.
Averaging over the base point as well gives $\bbE_{w\sim\mu_X}\bigl[\mathrm{Hyp}(X,\mathcal B(X),\mu_X,s_w)\bigr]=\tfrac{1}{3}\int_{\R}t\,d\nu_{\hyp,X}(t)$.
Chatterjee and Sloman show that $\mathrm{Hyp}$ carries geometrically meaningful information: just as vanishing Gromov hyperbolicity gives an isometric embedding into a tree, small average hyperbolicity guarantees an approximate one \cite[Thm.~2.4]{ChatterjeeSloman2021}.

\paragraph{Persistence sets and measures.}
Blumberg et al.\ \cite[Def.~1.1]{Blumberg2014} study the distribution $\Phi_k^n(X,d_X,\mu_X)$ of the degree-$k$ Vietoris--Rips persistence diagram over i.i.d.\ $n$-point samples of an mm-space.
They prove that these distributions are robust to adding a small proportion of arbitrary points, including outliers, because such points rarely appear in an $n$-point sample \cite[Thm.~5.5]{Blumberg2014}.
In contrast, even a few such points can change the persistence diagram of the entire space arbitrarily \cite[Props.~4.6 and~4.7]{Blumberg2014}.
The robustness of $\Phi_k^n$ makes it possible to base statistical procedures for persistent homology, such as hypothesis tests and confidence intervals, on repeated finite subsamples.
In a closely related framework, the second and third authors study the set-valued $(n,k)$-persistence set $\Dvr{n,k}(X)$ and the $(n,k)$-persistence measure $\Uvr{n,k}(X)$ \cite{GomezMemoli2024}.
Persistence sets can recover geometric information such as the curvature of a surface \cite[Thm.~5.19]{GomezMemoli2024}, while the persistence measures can distinguish spaces with the same persistence set but different frequencies of topological features \cite[Ex.~5.11, Prop.~5.12]{GomezMemoli2024}.
In our notation, the distribution of Blumberg et al.\ and this persistence measure both coincide with the pushforward under our $(n,k)$-persistence diagram invariant $\dgmVR_{n,k}$ (\Cref{inv:persistence-set}):
\begin{align*}
	\Phi_k^n(X,d_X,\mu_X)
	= \Uvr{n,k}(X)
	= \nu_{\dgmVR_{n,k},X}.
\end{align*}
Note that the distance distribution is already contained in this framework: after identifying a degree-zero diagram of the form $\{(0,r)\}$ with $r$, the measure $\Uvr{2,0}(X)$ agrees with $\nu_{\dist,X}$ \cite[Ex.~3.16]{GomezMemoli2024}.

\paragraph{Geometric distributions in phylogenetics.}
Several classical summary statistics in phylogenetics and population genetics are distributions of $n$-invariants of aligned sequences with the Hamming distance, or are closely related to them.
For a set of sequences, the \emph{mismatch distribution} \cite{SlatkinHudson1991} is the distance distribution $\nu_{\dist,X}$, whose shape carries information about demographic history \cite{RogersHarpending1992}; the \emph{nucleotide diversity} \cite{NeiLi1979} corresponds to the mean of $\nu_{\dist,X}$ divided by the sequence length.
Likewise, several measures of tree-likeness are closely related to the hyperbolicity distribution $\nu_{\hyp,X}$ (\Cref{inv:hyperbolicity}).
The \emph{$\delta$ plot} \cite{HollandEtAl2002} is the histogram, over all four-element subsets, of the $\delta$-score, a $4$-invariant obtained by normalizing the hyperbolicity deficit, and the \emph{Q-residuals} \cite{GrayBryantGreenhill2010} average $4\hyp(\mathbf{x})^2$ over quadruples after a rescaling of the distances.
An example of a distribution over quadruples that is not the distribution of an $n$-invariant is \emph{likelihood mapping} \cite{StrimmerVonHaeseler1997}, which explicitly depends on the site patterns among four sequences, not only on their distances.

\paragraph{Persistent homology and reticulate evolution.}
Reticulate evolution leaves cyclic signatures in sequence data.
The four-gamete test detects when two binary characters are incompatible with a tree on which each site mutates only once \cite{HudsonKaplan1985}.
Persistent homology of Vietoris--Rips complexes of viral sequences provides a related multiscale approach to detecting recombination and reassortment \cite{ChanCarlssonRabadan2013}.
Lesnick et al.\ develop mathematical foundations for this approach using galled-tree histories and extend their conclusions to subsamples by stability \cite{LesnickRabadanRosenbloom2020}.
In earlier joint work, the first, fourth, and fifth authors used short-lived topological cycles formed by SNVs to detect recurrent mutations in SARS\nobreakdash-CoV\nobreakdash-2 without reconstructing a phylogeny, and showed that such recurrences correlate with adaptation \cite{BleherEtAl2021}.

\section{Preliminaries}
\label{sec:preliminaries}

In this section we recall the basic definitions and preliminary results required in the rest of the paper: curvature sets and curvature measures (\Cref{subsec:prelim-curvature}), the Gromov--Hausdorff and Wasserstein distances together with empirical measures and sample sizes (\Cref{subsec:prelim-distances}), and properties of measures such as moments, enclosing radii, and intrinsic dimensions (\Cref{subsec:prelim-moments}).
The material is standard and we refer to the indicated references for further details.

Throughout, we use boldface letters $\mathbf{x}$ to represent tuples $(x_1, \dots, x_n)\in X^n$.
For a metric space $(X, d_X)$ and $p \geq 1$, the $\ell^p$ product metric on $X^n$ is defined by
\begin{equation*}
	d_{X,p}(\mathbf{x}, \mathbf{x}') \coloneq \Bigl( \sum_{i=1}^n d_X(x_i, x_i')^p \Bigr)^{1/p}.
\end{equation*}
For example, the $\ell^p$ product metric of $(\R, | \cdot |)$ is the metric induced by the $\ell^p$ norm $\| \cdot \|_p$.
Also recall that the diameter of $(X^n, d_{X,p})$ is $\diam_p(X^n) \coloneq n^{1/p} \diam(X)$.

\subsection{Curvature sets and curvature measures}
\label{subsec:prelim-curvature}

The defining property of a metric $n$-point invariant is that it depends on a tuple $\mathbf{x} \coloneq (x_1, \dots, x_n) \in X^n$ only through the pairwise distances $d_X(x_i, x_j)$.
This means it is a function of the distance matrix of $\mathbf{x}$ alone, such that the domain of an $n$-point invariant is the following space.

\begin{defn}[Metric cone]
	\label{def:metric_cone}
	The \emph{$n$-th metric cone} $\Kn_n$ is the set of matrices $M \in \R^{n \times n}$ such that
	\begin{itemize}
		\item $M_{ij} \geq 0$ and $M_{ii} = 0$,
		\item $M_{ij} = M_{ji}$, and
		\item $M_{ij} \leq M_{ik} + M_{kj}$.
	\end{itemize}
	In other words, $M \in \Kn_n$ if and only if $M$ is the distance matrix of some pseudometric space with $n$ points.
\end{defn}

\begin{remark}
The same space is called the \emph{semimetric cone} and denoted $\mathrm{MET}_n$ in polyhedral combinatorics \cite[Ch.~3]{DezaLaurent1997}.
An infinite-dimensional version, the cone of infinite distance matrices, is used by Vershik as a model for the collection of all Polish spaces \cite{Vershik2004}.
\end{remark}
\begin{remark}
	\label{rmk:Kn_smaller_embedding}
	Since distance matrices are symmetric and have $0$ diagonal, $\Kn_n$ admits an embedding into the smaller Euclidean space $\R^{\binom{n}{2}}$. We will specify the embedding dimension when the result depends on it.
\end{remark}

\begin{defn}[Curvature set {\cite[Sec.~$1.19_+$]{Gromov2007}}]
	Given a metric space $(X, d_X)$ and an integer $n \geq 2$, the \emph{$n$-th curvature set of $X$} is
	\begin{equation*}
		\Kn_n(X) \coloneq \Big\{M \in \R^{n \times n} \mid \exists \ \mathbf{x} \in X^n
		\text{ such that } M_{ij} = d_X(x_i,x_j) \Big\}.
	\end{equation*}
	$\Kn_n(X)$ comes with a natural map $\Psi_n \colon X^n \to \Kn_n(X)$ defined by $\left(\Psi_n(\mathbf{x})\right)_{ij} \coloneq d_X(x_i,x_j)$.
	We equip $\Kn_n(X)$ with the topology generated by the $\ell^p$ metric on $\R^{n \times n}$, for some $p \geq 1$.
	Note that $\Kn_n(X) \subset \Kn_n$.
\end{defn}

\begin{remark}
	\label{rmk:Kn_quotient}
	When $X$ is compact, $\Psi_n$ is a quotient map.
	Indeed, $\Psi_n$ is surjective by definition, continuous because the metric $d_X \colon X \times X \to \R$ is continuous, and closed because it maps the compact set $X^n$ to a Hausdorff space.
\end{remark}

\begin{defn}[Metric measure space]
	A \emph{metric measure space} (\emph{mm-space}), is a triple $(X,d_X, \mu_X)$ where $(X,d_X)$ is a Polish space, i.e.\ a complete separable metric space, and $\mu_X$ is a Borel probability measure on $X$, not necessarily fully supported.
\end{defn}

\begin{defn}[Curvature measure {\cite[Sec.~$3\tfrac{1}{2}{}_+$]{Gromov2007}}]
	Given an mm-space $(X,d_X, \mu_X)$, the \emph{$n$-th curvature measure} $\mu_{n,X}$, or $\mu_n$ when $X$ is clear from context, is defined as the pushforward of the product measure $\mu_X^{\otimes n}$ under the map $\Psi_n \colon X^n \to \Kn_n(X)$.
\end{defn}

\begin{remark}
	\label{rmk:mu_n_smaller_embedding}
	Similar to $\Kn_n$, $\Kn_n(X)$ can be embedded in $\R^{d}$ with either $d \coloneq n^2$ or $d \coloneq \binom{n}{2}$. Accordingly, we view the curvature measure $\mu_n$ as a measure on the corresponding $\R^d$ with support $\Kn_n(X) \subset \R^d$.
\end{remark}

\subsection{Distances and empirical measures}
\label{subsec:prelim-distances}

We recall the Gromov--Hausdorff distance, a standard metric on the space of compact metric spaces \cite{BuragoBuragoIvanov2001}, as it appears in stability results for Gromov hyperbolicity and persistent homology.
\begin{defn}[Gromov--Hausdorff distance, e.g.\ {\cite[Ch.~7]{BuragoBuragoIvanov2001}}]
	Let $(X, d_X)$ and $(Y, d_Y)$ be two metric spaces. Denote by $\pi_X:X\times Y \to X$ and $\pi_Y:X\times Y \to Y$ the canonical projections.
	A \emph{correspondence} is a relation $R \subset X \times Y$ such that $\pi_X(R) = X$ and $\pi_Y(R) = Y$, and its \emph{distortion} is
	\begin{equation*}
		\dis(R) \coloneq \sup_{(x,y), (x',y') \in R} |d_X(x,x') - d_Y(y,y')|.
	\end{equation*}
	The \emph{Gromov--Hausdorff} distance is defined as
	\begin{equation*}
		\GH(X, Y) \coloneq \frac{1}{2} \inf_R \dis(R),
	\end{equation*}
	where the infimum runs over all correspondences between $X$ and $Y$.
\end{defn}
The following comparison, which appeared previously as a step in the proof of \cite[Thm.~3.13]{GomezMemoli2024}, lets us translate the aforementioned stability results for Gromov hyperbolicity and persistent homology into a Lipschitz property of the corresponding $n$-invariant on $\Kn_n$.
\begin{lemma}
	\label{lemma:GH_vs_Kn}
	Let $(X, d_X)$ and $(Y, d_Y)$ be metric spaces with $n$ points. Choose any labelling $x_1, \dots, x_n$ of $X$ and $y_1, \dots, y_n$ of $Y$, and define $M_X \coloneq \Psi_n((x_1, \dots, x_n))$ and $M_Y \coloneq \Psi_n((y_1, \dots, y_n))$. Then
	\begin{equation*}
		\GH(X, Y) \leq \frac{1}{2}\|M_X - M_Y\|_\infty.
	\end{equation*}
\end{lemma}
\begin{proof}
	Let $R_\Delta$ be the diagonal correspondence $\{(x_1, y_1), \dots, (x_n, y_n)\}$. Then
	\begin{equation*}
		\GH(X, Y) \leq \frac{1}{2} \dis(R_\Delta) = \frac{1}{2} \max_{i,j=1, \dots, n} |d_X(x_i,x_j) - d_Y(y_i,y_j)| = \frac{1}{2} \|M_X - M_Y\|_\infty.
	\end{equation*}
\end{proof}

Our main results are formulated in terms of the Wasserstein distance, a standard metric on the space of probability measures that is widely used in optimal transport theory \cite{Villani2009}.
\begin{defn}[Wasserstein distance, e.g.\ {\cite[Ch.~6]{Villani2009}}]
	Fix $p \geq 1$ and a Polish space $(Z, d_Z)$. Let $\alpha$ and $\beta$ be Borel probability measures on $Z$ with finite $p$-th moments. Let $\Pi(\alpha,\beta)$ denote the set of couplings of $\alpha$ and $\beta$, that is, the Borel probability measures on $Z \times Z$ with marginals $\alpha$ and $\beta$.
	The \emph{$p$-Wasserstein distance} between $\alpha$ and $\beta$ is defined by
	\begin{align*}
		\W_p(\alpha,\beta)
		\coloneq
		\left(
		\inf_{\pi \in \Pi(\alpha,\beta)}
		\int_{Z \times Z} d_Z(z,z')^p \, d\pi(z,z')
		\right)^{1/p}.
	\end{align*}
\end{defn}

We will later also use the following standard result about the Wasserstein distance under Lipschitz maps, cf.\ \cite[Ch.~6]{Villani2009}, to pass from bounds on the curvature measure to bounds on the distribution of an $n$-invariant; we include a short proof for completeness.
\begin{lemma}
	\label{lemma:wasserstein_lipschitz}
	Let $(X, d_X)$ and $(Y, d_Y)$ be Polish metric spaces, and let $F \colon X \to Y$ be an $L$-Lipschitz function. For any two Borel probability measures $\alpha, \beta$ on $X$ and any $p \geq 1$,
	\begin{equation*}
		\W_p(F_\#\alpha, F_\#\beta) \leq L \cdot \W_p(\alpha, \beta).
	\end{equation*}
\end{lemma}
\begin{proof}
	Let $\pi \in \Pi(\alpha, \beta)$. Since the pushforward $\pi_{F} \coloneq (F \times F)_\# \pi$ is a coupling of $F_\#\alpha$ and $F_\#\beta$, we obtain
	\begin{align*}
		\W_p(F_\#\alpha, F_\#\beta)^p
		&\leq \int_{Y \times Y} d_Y(y, y')^p \, d\pi_{F}(y,y') \\
		&= \int_{X \times X} d_Y(F(x), F(x'))^p \, d\pi(x,x')
		\leq L^p \int_{X \times X} d_X(x, x')^p \, d\pi(x,x').
	\end{align*}
	Taking $p$-th root and infimizing over all $\pi \in \Pi(\alpha, \beta)$ yields the result.
\end{proof}

Throughout, given i.i.d.\ samples $y_1, \dots, y_N \sim \mu$ from a Borel probability measure $\mu$ on any Polish space, we write 
\begin{align}
	\hat{\mu}_N \coloneq \frac{1}{N} \sum_{i=1}^N \delta_{y_i}
\end{align}
for their \emph{empirical measure}.
For example, for the curvature measure $\mu_n$ of an mm-space $(X, d_X, \mu_X)$, we sample $N$ tuples $\mathbf{x}^{(1)}, \dots, \mathbf{x}^{(N)} \sim \mu_X^{\otimes n}$ and their distance matrices $\Psi_n(\mathbf{x}^{(i)})$ are i.i.d.\ samples from $\mu_n$, with empirical measure
\begin{equation*}
	\empirical{\mu}{n}{N} \coloneq \frac{1}{N} \sum_{i=1}^N \delta_{\Psi_n(\mathbf{x}^{(i)})}.
\end{equation*}

\Cref{thm:main-sampling,thm:main-sampling-intrinsic} provide bounds of the form $\bbE\!\left[\W_p(\hat{\mu}_N, \mu)^p\right] \leq \beta(N)$, where $\beta(N)$ is a non-increasing non-negative function of $N$.
Such bounds also determine how many samples are needed to reach a given error with a given confidence.
We use the strategy of \cite[Sec.~2.3]{WeedBach2019}: choose $N$ large enough that the expectation of $\W_p^p$ is small, then use the concentration of $\W_p^p$ around its expectation, which follows from McDiarmid's bounded differences inequality \cite[Prop.~20]{WeedBach2019}.
The following lemma records this argument in a form suitable for our purposes.
\begin{lemma}
	\label{lemma:empirical_sample_size}
	Let $1 \leq p < \infty$, and let $\mu$ be a Borel probability measure on a Polish space $(Z, d_Z)$ with compact support and $\diam(\supp \mu) \leq D$ for some $D > 0$.
	Suppose $\bbE[\W_p(\hat{\mu}_N, \mu)^p] \leq \beta(N)$ for some non-increasing function $\beta \colon (0, \infty) \to [0, \infty)$.
	Choose an error margin $\epsilon > 0$ and a confidence level $0 < \alpha < 1$, and let $N_0 > 0$ satisfy
	\begin{equation*}
		\frac{\beta(N_0)}{D^p} + \sqrt{-\tfrac{1}{2} \ln(1-\alpha)} N_0^{-1/2} \leq \epsilon^p.
	\end{equation*}
	Then for all $N > N_0$,
	\begin{align}
		\W_p(\hat{\mu}_N, \mu) < \epsilon D	
	\end{align}
	with probability at least $\alpha$.
\end{lemma}
\begin{proof}
	Write $\W_p \coloneq \W_p(\hat{\mu}_N, \mu)$ for short, and let $K \coloneq \supp \mu$.
	Almost surely, all samples lie in $K$, so $\W_p$ is the Wasserstein distance of two measures on $K$.
	Since the paper \cite{WeedBach2019} works with compact spaces of diameter at most $1$, we rescale the metric on $K$ by $1/D$, which rescales $\W_p$ to $\widetilde{\W}_p \coloneq \W_p/D$.
	By \cite[Prop.~20]{WeedBach2019}, $\bbP[\widetilde{\W}_p^p - \bbE[\widetilde{\W}_p^p] \geq t] \leq \exp\left(-2Nt^2 \right)$ or, equivalently,
	\begin{equation*}
		\bbP[\W_p^p - \bbE[\W_p^p] \geq tD^p] \leq \exp\left(-2Nt^2 \right).
	\end{equation*}
	Let $t \coloneq \sqrt{-\frac{1}{2} \ln(1-\alpha)} N^{-1/2}$, so that $\exp(-2Nt^2) = 1 - \alpha$.
	Since $\beta$ is non-increasing and $N > N_0$, with probability at least $\alpha$ we have
	\begin{equation*}
		\begin{multlined}[.7\textwidth]
			\W_p^p
			< \bbE[\W_p^p] + tD^p
			\leq \beta(N) + D^p \sqrt{-\tfrac{1}{2} \ln(1-\alpha)} N^{-1/2} \\
			< \left( \frac{\beta(N_0)}{D^p} + \sqrt{-\tfrac{1}{2} \ln(1-\alpha)} N_0^{-1/2} \right) D^p
			\leq \epsilon^p D^p.
		\end{multlined}
	\end{equation*}
\end{proof}

\subsection{Moments and dimensions of measures}
\label{subsec:prelim-moments}

We now collect the remaining ingredients required for the proofs of \Cref{thm:main-sampling,thm:main-sampling-intrinsic}.
Both bound $\bbE\!\left[\W_p(\hat{\mu}_N, \mu)^p\right]$ by functions of the form $\beta(N) = \sum_i C_i N^{-e_i}$, where the smallest $e_i$ is determined by the dimension of the set in which the convergence is described.
\Cref{thm:main-sampling} applies the bounds of \cite{convergence_empirical_fournier, Fournier2023} for arbitrary measures in Euclidean space $\R^d$ to $\mu_n$, so the relevant dimension is the ambient $d = \binom{n}{2}$.
In this case there is only a single term and the constant depends on the moments of $\mu_n$.
In contrast, \Cref{thm:main-sampling-intrinsic} relies on convergence within sets of full measure \cite{convergence_empirical_covering_dim_2014, convergence_empirical_covering_dim_2026} or of almost full measure \cite{WeedBach2019}, whose dimensions can be smaller than the ambient dimension and may thus yield a faster convergence rate.
We begin by recalling some basic facts about the moments of a measure and then turn to the relevant notions of dimension.

\begin{defn}[Moments of a measure]
	Let $X \subset \R^d$ and let $\mu_X$ be a Borel probability measure supported on $X$. Fix $1 \leq q < \infty$, $1 \leq m \leq \infty$, and $\mathbf{p}_0 \in \R^d$. The \emph{$q$-th moment of $\mu_X$ about $\mathbf{p}_0$} is
	\begin{equation*}
		\cM_q^{(m)}(\mu_X, \mathbf{p}_0) \coloneq \int_{X} \|\mathbf{x} - \mathbf{p}_0\|_m^q \ \mu_X(d\mathbf{x}).
	\end{equation*}
	We define the \emph{$q$-th translation invariant moment} as $\cM_q^{(m)}(\mu_X) \coloneq \inf_{\mathbf{p}_0 \in \R^d} \cM_q^{(m)}(\mu_X, \mathbf{p}_0)$.\\
	We also define
	\begin{align*}
		\cM_\infty^{(m)}(\mu_X, \mathbf{p}_0) &\coloneq \sup_{q \geq 1} \cM_q^{(m)}(\mu_X, \mathbf{p}_0)^{1/q} = \sup_{\mathbf{x} \in X} \|\mathbf{x} - \mathbf{p}_0\|_m\\
		\cM_\infty^{(m)}(\mu_X) &\coloneq \sup_{q \geq 1} \cM_q^{(m)}(\mu_X)^{1/q}.
	\end{align*}
\end{defn}

The translation invariant moments are defined by an infimum over the base point $\mathbf{p}_0$ and we will need to know how the moments may vary with $\mathbf{p}_0$, which is the content of the following lemma.
\begin{lemma}
	\label{lemma:moments_lipschitz}
	The functions $\mathbf{p}_0 \mapsto \cM_q^{(m)}(\mu_X, \mathbf{p}_0)^{1/q}$ with $1 \leq q < \infty$ and $\mathbf{p}_0 \mapsto \cM_\infty^{(m)}(\mu_X, \mathbf{p}_0)$ are \mbox{$1$-Lipschitz}.
\end{lemma}
\begin{proof}
	By Minkowski's inequality,
	\begin{align*}
		\cM_q^{(m)}(\mu_X, \mathbf{p}_0)^{1/q}
		&\leq \left(\int_{X} ( \|\mathbf{x} - \mathbf{p}_0'\|_m + \|\mathbf{p}_0' - \mathbf{p}_0\|_m)^q \ \mu_X(d\mathbf{x}) \right)^{1/q} \\
		&\leq \left(\int_{X} \|\mathbf{x} - \mathbf{p}_0'\|_m^q \ \mu_X(d\mathbf{x}) \right)^{1/q} + \left(\int_{X} \|\mathbf{p}_0' - \mathbf{p}_0\|_m^q \ \mu_X(d\mathbf{x}) \right)^{1/q} \\
		&= \cM_q^{(m)}(\mu_X, \mathbf{p}_0')^{1/q} + \|\mathbf{p}_0' - \mathbf{p}_0\|_m.
	\end{align*}
	Swapping the roles of $\mathbf{p}_0$ and $\mathbf{p}_0'$ yields $|\cM_q^{(m)}(\mu_X, \mathbf{p}_0')^{1/q} - \cM_q^{(m)}(\mu_X, \mathbf{p}_0)^{1/q}| \leq \|\mathbf{p}_0' - \mathbf{p}_0\|_m$, as claimed. The proof for $\cM_\infty^{(m)}(\mu_X, \mathbf{p}_0)$ is analogous.
\end{proof}

The moments of a measure are controlled by the size of its support, more precisely its minimum enclosing radius.
This will allow us to replace the moment factor in the empirical convergence rates of \cite{Fournier2023} by the diameter of $X$, which is how we obtain an explicit constant in \Cref{thm:main-sampling}.
\begin{defn}[Minimum enclosing radius]
	Let $X \subset \R^d$ and fix $1 \leq p \leq \infty$. The \emph{minimum enclosing radius} of $X$ in the $\ell^p$ metric is
	\begin{equation*}
		R_p(X)
		\coloneq \inf \{R > 0 \mid \exists \ \mathbf{p}_0 \in \R^d \text{ such that } \|\mathbf{x} - \mathbf{p}_0\|_p \leq R \text{ for all } \mathbf{x} \in X\}.
	\end{equation*}
	Alternatively, $R_p(X) = \inf_{\mathbf{p}_0 \in \R^d} \sup_{\mathbf{x} \in X} \|\mathbf{x} - \mathbf{p}_0\|_p$.
\end{defn}
\begin{prop}
	\label{prop:moment_radius}
	Let $X \subset \R^d$ and let $\mu_X$ be a Borel probability measure supported on $X$. Fix $1 \leq m \leq \infty$. Then $\displaystyle \cM_\infty^{(m)}(\mu_X) = \sup_{q \geq 1} \cM_{q}^{(m)}(\mu_X)^{1/q} = R_m(X)$.
\end{prop}
\begin{proof}
	Given $\mathbf{p}_0 \in \R^d$, define $f_{m,\mathbf{p}_0} \colon X \to \R$ by $f_{m,\mathbf{p}_0}(\mathbf{x}) \coloneq \|\mathbf{x} - \mathbf{p}_0\|_m$. Since $\mu_X$ is a probability measure and $\cM_q^{(m)}(\mu_X, \mathbf{p}_0)^{1/q} = \|f_{m,\mathbf{p}_0}\|_{L^q(X, \mu_X)}$, the function $q \mapsto \cM_q^{(m)}(\mu_X, \mathbf{p}_0)^{1/q}$ is increasing in $q$ and has limit $\cM_\infty^{(m)}(\mu_X, \mathbf{p}_0) = \left\| f_{m,\mathbf{p}_0} \right\|_{L^\infty(X, \mu_X)}$. Moreover,
	\begin{equation*}
		\sup_{q \geq 1} \cM_{q}^{(m)}(\mu_X)^{1/q}
		\leq \inf_{\mathbf{p}_0 \in \R^d} \sup_{q \geq 1} \cM_q^{(m)}(\mu_X, \mathbf{p}_0)^{1/q}
		= \inf_{\mathbf{p}_0 \in \R^d} \sup_{\mathbf{x} \in X} \|\mathbf{x} - \mathbf{p}_0\|_m
		= R_m(X).
	\end{equation*}
	If $\supp(\mu_X)$ is unbounded, every term above is infinite, so we only prove the reverse inequality when $\supp(\mu_X)$ is bounded.\\
	Choose a closed ball $\overline{B} \subset \R^d$ that is large enough to contain $\supp(\mu_X)$ and every point $\mathbf{p}_0$ that realizes the infimum in $\cM_q^{(m)}(\mu_X)$ for $q \geq 1$. By \Cref{lemma:moments_lipschitz}, the functions $\mathbf{p}_0 \mapsto \cM_q^{(m)}(\mu_X, \mathbf{p}_0)^{1/q}$ and $\mathbf{p}_0 \mapsto \cM_\infty^{(m)}(\mu_X, \mathbf{p}_0)$ are $1$-Lipschitz and, hence, continuous for all $q \geq 1$. After replacing $q \geq 1$ with an integer $k \geq 1$ and restricting $\mathbf{p}_0$ to the compact set $\overline{B}$, we get a sequence $\left( \cM_k^{(m)}(\mu_X, \mathbf{p}_0)^{1/k} \right)_k$ of non-decreasing continuous functions defined on a compact set with a continuous limit, so Dini's theorem implies that $\cM_k^{(m)}(\mu_X, \mathbf{p}_0)^{1/k}$ converges uniformly to $\cM_\infty^{(m)}(\mu_X, \mathbf{p}_0)$. Then for every $\epsilon > 0$, there exists $k_0 > 0$ such that
	\begin{equation*}
		\cM_\infty^{(m)}(\mu_X, \mathbf{p}_0) - \epsilon < \cM_k^{(m)}(\mu_X, \mathbf{p}_0)^{1/k}
	\end{equation*}
	for every $k > k_0$ and $\mathbf{p}_0 \in \overline{B}$. Then
	\begin{equation*}
		R_m(X) - \epsilon
		= \inf_{\mathbf{p}_0 \in \overline{B}} \cM_\infty^{(m)}(\mu_X, \mathbf{p}_0) - \epsilon
		\leq \inf_{\mathbf{p}_0 \in \overline{B}} \cM_k^{(m)}(\mu_X, \mathbf{p}_0)^{1/k}
		= \cM_k^{(m)}(\mu_X)^{1/k}
		\leq \cM_\infty^{(m)}(\mu_X).
	\end{equation*}
	Letting $\epsilon \to 0$ yields $R_m(X) \leq \cM_\infty^{(m)}(\mu_X)$ and, with that, the result.
\end{proof}

As mentioned above, the estimates of \cite{convergence_empirical_covering_dim_2014, WeedBach2019, convergence_empirical_covering_dim_2026} may replace the dimension of the ambient Euclidean space by an intrinsic dimension of $\supp(\mu)$, which can be considerably smaller.
These sources do not all use the same notion of dimension.
The ones we need are the standard Hausdorff and Minkowski dimensions $d_H(X)$ and $d_M(X)$ of a metric space, and the Wasserstein dimension $d_p^*(\mu_X)$ of \cite[Def.~4]{WeedBach2019}, which depends on a parameter $1 \leq p < \infty$.
\cite{WeedBach2019} also use the Minkowski dimension $d_M(\mu)$ of a measure \cite[Def.~5]{WeedBach2019}, but on a Polish space it equals $d_M(\supp \mu)$ because the Minkowski dimension is monotone and invariant under closures \cite[Lem.~3.3(i),(ii)]{Dimensions-Embeddings-Attractors}.
Rather than reproducing the precise definitions, we impose a standard regularity condition on our mm-spaces, under which all three agree with a single parameter~$s$.

\begin{defn}[Ahlfors regularity, e.g.\ {\cite[Ch.~8]{heinonen}}]
	\label{def:ahlfors-regular}
	Let $(X, d_X)$ be a metric space and $\mu_X$ a Borel measure on $X$.
	The triple $(X, d_X, \mu_X)$ is said to be \emph{Ahlfors $s$-regular} if there exists a constant $C_X$ such that
	\begin{equation}
		\label{eq:ahlfors-regular}
		\frac{1}{C_X} r^{s} \leq \mu_X(B_X(x_0, r)) \leq C_X r^{s}
	\end{equation}
	for every $x_0 \in X$ and $0 < r < \diam(X)$.
\end{defn}
\begin{lemma}
	\label{lemma:ahlfors-compact}
	Let $(X, d_X, \mu_X)$ be an Ahlfors $s$-regular mm-space.
	Then $X$ is compact and $\mu_X$ is fully supported.
\end{lemma}
\begin{proof}
	By \eqref{eq:ahlfors-regular}, $\mu_X(B_X(x,\epsilon/2)) \geq (\epsilon/2)^s/C_X$ for every $x \in X$ and $0 < \epsilon < \diam(X)$.
	If $x_1, \dots, x_k \in X$ are $\epsilon$-separated, the balls $B_X(x_i,\epsilon/2)$ are pairwise disjoint, and since $\mu_X$ is a probability measure, $k \leq C_X (2/\epsilon)^s$.
	The $\epsilon$-balls around an $\epsilon$-separated set of maximal cardinality cover $X$, so $X$ is totally bounded.
	Since $d_X$ is complete, $X$ is compact.
	Finally, \eqref{eq:ahlfors-regular} holds at every $x_0 \in X$, so every ball has positive measure and $\mu_X$ is fully supported.
\end{proof}

The following lemma shows that for an Ahlfors $s$-regular mm-space, the three relevant notions of dimension coincide with $s$.
In \Cref{sec:geometry-Kn}, we use it to bound the dimensions of $\Kn_n(X)$, which then enter the convergence rates of \cite{WeedBach2019} in the proof of \Cref{thm:main-sampling-intrinsic} in \Cref{sec:convergence-rates}.

\begin{lemma}
	\label{lemma:ahlfors-dimensions}
	Let $(X, d_X, \mu_X)$ be an Ahlfors $s$-regular mm-space and $\mathcal{H}_X^s$ the $s$-dimensional Hausdorff measure on $(X, d_X)$. Then:
	\begin{enumerate}
		\item\label{it:dimensions_abs_continuous} $\mu_X$ is absolutely continuous with respect to $\mathcal{H}_X^s$, written $\mu_X \ll \mathcal{H}_X^s$.
		\item\label{it:dimensions_hauss_mink} $d_H(X) = d_M(X) = s$.
		\item\label{it:dimensions_wasserstein} $d_p^*(\mu_X) =s$ for $1 \leq p \leq s/2$.
	\end{enumerate}
\end{lemma}
\begin{proof}
	Assume $(X, d_X, \mu_X)$ is Ahlfors $s$-regular.
	By \cite[Ex.~8.11]{heinonen}, there exists $C'>0$ such that
	\begin{equation}
		\label{eq:hausdorff_ahlfors_bilipschitz}
		\frac{1}{C'} \mathcal{H}_X^s(E) \leq \mu_X(E) \leq C' \mathcal{H}_X^s(E)
	\end{equation}
	for every Borel $E \subset X$.
	Then:
	\begin{enumerate}
		\item Every null set $E$ of $\mathcal{H}_X^s$ satisfies $\mu_X(E) \leq C' \mathcal{H}_X^s(E) = 0$, so $\mu_X \ll \mathcal{H}_X^s$ (and, analogously, $\mathcal{H}_X^s \ll \mu_X$).
		\item Combining equations \eqref{eq:ahlfors-regular} and \eqref{eq:hausdorff_ahlfors_bilipschitz} yields
		\begin{equation}
			\label{eq:hausdorff_measure_ahlfors}
			\frac{1}{C_XC'} \cdot r^{s} \leq \mathcal{H}_X^s(B_X(x_0,r)) \leq (C_XC') \cdot r^{s},
		\end{equation}
		so by \cite[Sec.~8.7]{heinonen}, $d_H(X) = s$. Regarding $d_M(X)$, the argument of \cite[Thm.~5.7]{mattila} can be extended from $\R^n$ to any metric space, so $d_M(X) = s$.
		\item By \Cref{lemma:ahlfors-compact} and \eqref{eq:hausdorff_measure_ahlfors}, $\supp \mu_X = X$ is a regular set of dimension $s$ in the sense of \cite[Def.~6]{WeedBach2019}.
		By \ref{it:dimensions_abs_continuous}, \cite[Prop.~8]{WeedBach2019} then yields $d_p^*(\mu_X) = s$ for $1 \leq p \leq s/2$.
	\end{enumerate}
\end{proof}

\section{Metric \texorpdfstring{$n$}{n}-point invariants}
\label{sec:n-point-invariants}

Metric $n$-point invariants formalize quantities that depend only on the pairwise distances among $n$ points.
In the simplest cases, $n$-point invariants take values in $\R$ or $\R^k$, but more general spaces also arise naturally, for example in the case of persistence diagrams.
From now on, we fix a Polish space $(Z,d_Z)$ equipped with its Borel $\sigma$-algebra.

\begin{defn}
	\label{def:n_invariant}
    A \emph{metric $n$-point invariant} valued in $Z$, or \emph{$n$-invariant} for short, is any measurable function $F \colon \Kn_n \to Z$.
	We say that an $n$-invariant is \emph{$L$-Lipschitz} if
   \begin{equation*}
        d_Z(F(M), F(M')) \leq L \cdot \|M - M' \|_\infty
   \end{equation*}
   for any $M, M' \in \Kn_n$, and \emph{Lipschitz} if it is $L$-Lipschitz for some $L \geq 0$, in which case we write $\operatorname{Lip}(F)$ for the smallest such $L$.
\end{defn}
Any $n$-invariant induces a map on $X^n$ via $F \circ \Psi_n \colon X^n \to Z$.
Conversely, any map $F_X \colon X^n \to Z$ that satisfies $F_X(\mathbf{x}) = F_X(\mathbf{x}')$ whenever $\Psi_n(\mathbf{x}) = \Psi_n(\mathbf{x}')$ defines a map $F \colon \Kn_n(X) \to Z$.
We often abuse notation and write $F(\mathbf{x})$ instead of $F(\Psi_n(\mathbf{x}))$ or $F_X(\mathbf{x})$.

\begin{defn}
	Let $F \colon \Kn_n \to Z$ be an $n$-invariant.
	Given an mm-space $(X, d_X, \mu_X)$, the \emph{distribution of $F$ over $X$} is the pushforward
	\begin{align*}
		\nu_{F, (X, d_X, \mu_X)} &\coloneq F_\# (\mu_{n,X}) = (F \circ \Psi_n)_\#(\mu_X^{\otimes n}).
	\end{align*}
	Throughout, we suppress parameters that can be inferred from context and write, for instance, $\nu_{F,X}$ and $\nu_{F,\mu}$.
\end{defn}
\begin{remark}
	\label{rmk:invariant-random-variable}
	Note that an $n$-invariant $F$ is a $Z$-valued random variable on the probability space $(\Kn_n,\mu_{n,X})$ for every choice of mm-space $X$, and $\nu_{F,X}$ is its distribution.
	Equivalently, $\nu_{F,X}$ is the distribution of $F(\mathbf{x})$ for a random $n$-tuple $\mathbf{x}\sim\mu_X^{\otimes n}$ of independent $\mu_X$-distributed points.
\end{remark}
The collection of all metric $n$-point invariants is closed under measurable postcomposition, and their induced distributions behave naturally under this operation.
\begin{prop}
	\label{prop:invariant-postcomposition}
	Let $F \colon \Kn_n \to Z$ be an $n$-invariant and let $f \colon Z \to Z'$ be measurable. Then $f\circ F$ is a $Z'$-valued $n$-invariant and $\nu_{f\circ F,X}=f_\#\nu_{F,X}.$
\end{prop}
\begin{proof}
	This follows from measurability of compositions and functoriality of pushforwards.
\end{proof}

The induced measure $\nu_{F,X}$ should be viewed as the distribution of the geometric quantity measured by $F$ across $X$.
As with any probability distribution, we can describe $\nu_{F,X}$ through statistics such as its mean and higher moments, each of which then summarizes an aspect of the geometry of $X$.
For general $Z$, this is obtained from the associated Fr\'echet function $\Phi_{\nu_{F,X}}(z_0) \coloneq \int_Z d_Z(z,z_0)^2 \, d\nu_{F,X}(z)$, whose minimizers, when they exist, are the Fr\'echet means of $\nu_{F,X}$.
When $Z = \R$ or $Z = \R^k$, the Fr\'echet mean reduces to the ordinary mean, and one recovers the familiar Euclidean summaries of a distribution.
More generally, we can compute $\bbE_{\nu_{F,X}}[h] \coloneq \int_Z h(z)\, d\nu_{F,X}(z)$ for a measurable $h \colon Z \to \R$ and recover, when $(Z, d_Z)$ is Euclidean space for instance, the moments, variances, and quantiles of $\nu_{F,X}$. More importantly for our purposes, we can use the Wasserstein distance to compare the distributions of an $n$-invariant $F$ on different mm-spaces, and thus quantify how different their geometry is in terms of the properties extracted by $F$.

There are many examples of $n$-invariants in the literature.
In the following, we discuss the ones we will use in our applications to Hamming geometry and the analysis of viral sequence data in \Cref{sec:hamming}, along with a few closely related ones.
Several of them are, or are closely related to, quantities used to study the geometry of evolutionary processes.

\begin{example}[Pairwise distance]
	\label{inv:distance}
	The simplest $n$-invariant is the canonical embedding $\Kn_2 \hookrightarrow \R$, which we call the \emph{distance map} and denote as $\dist$. This name reflects the fact that the induced map $\dist_X \colon X^2 \to \R$ is
	\begin{align*}
		\dist_X(\mathbf{x}) = d_X(x_1, x_2)
	\end{align*}
	for any $(X, d_X)$. Accordingly, the induced measure $\nu_{\dist,X}$ is called the \emph{distance distribution of $X$}; cf.\ \cite[Def.~5.4]{Memoli2011}. Note that any $Z$-valued $2$-invariant has the form $h \circ \dist$ for some measurable $h \colon \R \to Z$.
	For sequence data, $\nu_{\dist,X}$ is known as the \emph{mismatch distribution} in population genetics.
	Its mean is the nucleotide diversity \cite{NeiLi1979} scaled by the sequence length, and its shape is used to infer demographic history \cite{RogersHarpending1992}.
\end{example}

\begin{example}[Four-point hyperbolicity deficit]
	\label{inv:hyperbolicity}
	An important example with $n=4$ is the \emph{four-point hyperbolicity deficit}.
	Let $(X, d_X)$ be a metric space and $\mathbf{x} \in X^4$.
	For a partition $\{i,j\}\cup\{k,\ell\}=\{1,2,3,4\}$, denote the associated sum of pairwise distances by
	\begin{align*}
		s_{ij,k\ell}(\mathbf{x}) \coloneq d_X(x_i,x_j)+d_X(x_k,x_\ell).
	\end{align*}
	Collect the three possible sums in $S(\mathbf{x}) \coloneq\bigl(s_{13,24}(\mathbf{x}),s_{12,34}(\mathbf{x}),s_{14,23}(\mathbf{x})\bigr)$, and let $S_{(1)}(\mathbf{x})\geq S_{(2)}(\mathbf{x})\geq S_{(3)}(\mathbf{x})$ denote its components in descending order.
	Then the hyperbolicity deficit of $\mathbf{x}$ is the upper gap
	\begin{align*}
		\hyp(\mathbf{x})
		&\coloneq \frac{1}{2}\left(S_{(1)}(\mathbf{x})-S_{(2)}(\mathbf{x})\right).
	\end{align*}
	If $S_{(1)}(\mathbf{x})$ and $S_{(2)}(\mathbf{x})$ coincide, then $\hyp(\mathbf{x})=0$.
	This is exactly the case when the distances of the quadruple $\mathbf{x}$ behave like those on a tree, in the sense that they can be realized by a tree with four leaves corresponding to $x_1,x_2,x_3,x_4$; see \Cref{fig:hyperbolicity}.
	In particular, $\hyp(\mathbf{x}) = 0$ whenever two of the points coincide: if, say, $x_1 = x_2$, then $s_{13,24}(\mathbf{x}) = s_{14,23}(\mathbf{x}) \geq d_X(x_3,x_4) = s_{12,34}(\mathbf{x})$ by the triangle inequality.
	For arbitrary $\mathbf{x}$, the triangle inequality gives $s_{13,24}(\mathbf{x}) - s_{12,34}(\mathbf{x}) \leq 2d_X(x_2,x_3)$, and likewise for any other pair of sums, so $\hyp(\mathbf{x}) \leq \diam(X)$.

	\begin{figure}
		\centering
		\begin{tikzpicture}[scale=1.8]
    \begin{scope}
        \node at (0.65,-0.55) {\small $\hyp(\mathbf{x})>0$};
        \coordinate (x1) at (-0.10,-0.04);
        \coordinate (x2) at (1.48,0.15);
        \coordinate (x3) at (1.10,1.08);
        \coordinate (x4) at (-0.22,0.78);
        \draw[thick] (x3) to[out=188,in=8] (x1);
        \draw[preaction={draw=white,line width=4pt}, thick] (x4) to[out=-8,in=188] (x2);
        \draw[dashed] (x4) to[out=15,in=165] (x3);
        \draw[dashed] (x1) to[out=-15,in=-165] (x2);
        \foreach \pt/\lab/\pos in {x1/$x_1$/below,x2/$x_2$/below,x3/$x_3$/above,x4/$x_4$/above}
            \node[fill=black,circle,inner sep=1.8pt,label=\pos:{\lab}] at (\pt) {};
    \end{scope}

    \begin{scope}[xshift=4.0cm]
        \node[align=center] at (0.9,-0.5) {\small $\hyp(\mathbf{x})=0$};
        \coordinate (y1) at (0.00,-0.08);
        \coordinate (y4) at (0.00,0.88);
        \coordinate (y2) at (1.75,-0.08);
        \coordinate (y3) at (1.75,0.88);
        \coordinate (uL) at (0.55,0.35);
        \coordinate (uU) at (0.55,0.41);
        \coordinate (vL) at (1.20,0.35);
        \coordinate (vU) at (1.20,0.41);
        \coordinate (mL1) at (0.82,0.35);
        \coordinate (mU1) at (0.82,0.41);
        \coordinate (m) at (0.875,0.38);
        \coordinate (mL2) at (0.93,0.35);
        \coordinate (mU2) at (0.93,0.41);
        \coordinate (b1) at (0.00,-0.13);
        \coordinate (bu) at (0.55,0.31);
        \coordinate (bv) at (1.20,0.31);
        \coordinate (b2) at (1.75,-0.13);
        \coordinate (c4) at (0.00,0.93);
        \coordinate (cu) at (0.55,0.45);
        \coordinate (cv) at (1.20,0.45);
        \coordinate (c3) at (1.75,0.93);
        \draw[black, thick] (y1) -- (uL) -- (mL1);
        \draw[black, thick] (mL1) -- (m) -- (mU2);
        \draw[black, thick] (mU2) -- (vU) -- (y3);
        \draw[black, thick] (y4) -- (uU) -- (mU1);
        \draw[preaction={draw=white,line width=4pt}, black, thick] (mU1) -- (m) -- (mL2);
        \draw[black, thick] (mL2) -- (vL) -- (y2);
        \draw[dashed, rounded corners=2pt] (b1) -- (bu) -- (bv) -- (b2);
        \draw[dashed, rounded corners=2pt] (c4) -- (cu) -- (cv) -- (c3);
        \foreach \pt/\lab/\pos in {y1/$x_1$/below,y2/$x_2$/below,y3/$x_3$/above,y4/$x_4$/above}
            \node[fill=black,circle,inner sep=1.8pt,label=\pos:{\lab}] at (\pt) {};
    \end{scope}

    \draw[thick] (2.25,0.27) -- (2.55,0.27);
    \node at (2.75,0.27) {$S_{(1)}$};
    \draw[thick] (2.95,0.27) -- (3.25,0.27);
    \draw[dashed,thick] (2.25,-0.13) -- (2.55,-0.13);
    \node at (2.75,-0.13) {$S_{(2)}$};
    \draw[dashed,thick] (2.95,-0.13) -- (3.25,-0.13);
\end{tikzpicture}
		\caption{
			Illustration of the four-point hyperbolicity deficit.
			The hyperbolicity deficit $\hyp(\mathbf{x}) = \frac{1}{2}(S_{(1)}(\mathbf{x})-S_{(2)}(\mathbf{x}))$ measures the failure of a quadruple to behave like leaves on a tree.
			In a generic metric space, shown on the left, $S_{(1)}$ and $S_{(2)}$ do not reuse the same distances, so $S_{(1)}>S_{(2)}$ and $\hyp(\mathbf{x})>0$.
			In contrast, when the distances arise from a tree, shown on the right, $S_{(1)}$ and $S_{(2)}$ arise from the same edges with equal multiplicity and therefore coincide, giving $\hyp(\mathbf{x})=0$.
		}
		\label{fig:hyperbolicity}
	\end{figure}
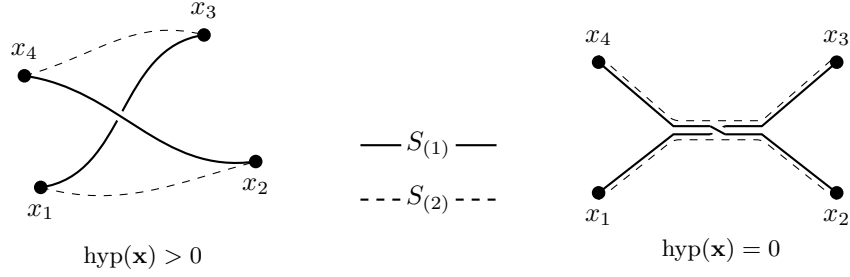

	For a metric space $X$, the supremum $\hyp(X) \coloneq \sup_{\mathbf{x}\in X^4}\hyp(\mathbf{x})$ is known as the \emph{Gromov hyperbolicity} of $X$ \cite[Sec.~1.1, p.~89]{Gromov1987}.
	It is a basic notion in metric geometry and geometric group theory that measures worst-case deviation from tree-likeness.
	For geodesic spaces, bounds on $\hyp(X)$ are equivalent to a uniform slim-triangle condition.
	In particular, $\hyp(X)=0$ if and only if $X$ embeds isometrically into an $\R$-tree (e.g.\ \cite[Thm.~8]{dress-tight-span}).
	In phylogenetics, the condition $\hyp(\mathbf{x})=0$ for all quadruples is known as the four-point condition, which characterizes the metrics that can be realized by trees \cite{Buneman1974}.

	For an mm-space $(X,d_X,\mu_X)$, the distribution $\nu_{\hyp,X}$ complements this worst-case measure.
	If $X$ is an $\R$-tree, every quadruple satisfies the four-point condition and $\nu_{\hyp,X}=\delta_0$.
	Since the only coupling of any probability measure with $\delta_0$ is the product coupling, the moments of $\nu_{\hyp,X}$ are its Wasserstein distances to the hyperbolicity distribution of a tree,
	\begin{align*}
		\W_p(\nu_{\hyp,X},\delta_0)^p = \bbE_{\mathbf{x}\sim\mu_X^{\otimes4}}\bigl[\hyp(\mathbf{x})^p\bigr]
		\qquad\text{for all } 1\leq p<\infty.
	\end{align*}
	In particular, up to a factor of $\tfrac13$, the $1$-Wasserstein distance $\W_1(\nu_{\hyp,X},\delta_0)$ is the average hyperbolicity of Chatterjee and Sloman with the base point averaged over $\mu_X$ (\Cref{sec:related-work}), while in the limit $p\to\infty$ we recover the Gromov hyperbolicity $\hyp(X)$ as long as $\mu_X$ has full support.

	Two classical measures of tree-likeness in phylogenetics (\Cref{sec:related-work}) are closely related to $\nu_{\hyp,X}$.
	In our notation, the $\delta$-score of Holland et al.\ \cite{HollandEtAl2002} is
	\begin{align*}
		\delta(\mathbf{x}) \coloneq \frac{S_{(1)}(\mathbf{x})-S_{(2)}(\mathbf{x})}{S_{(1)}(\mathbf{x})-S_{(3)}(\mathbf{x})} = \frac{2\hyp(\mathbf{x})}{S_{(1)}(\mathbf{x})-S_{(3)}(\mathbf{x})}
	\end{align*}
	when $S_{(1)}(\mathbf{x})>S_{(3)}(\mathbf{x})$, and $0$ otherwise, and the $\delta$-plot of a finite metric space is the histogram of these scores over all four-element subsets.
	The normalization makes $\delta$ invariant under rescaling of the metric, but also discontinuous at quadruples with $S_{(1)}(\mathbf{x})=S_{(2)}(\mathbf{x})=S_{(3)}(\mathbf{x})$, near which it takes the values $0$ and $1$.
	
	The Q-residual of Gray et al.\ \cite{GrayBryantGreenhill2010} first rescales the distances so that their average over pairs of distinct points is $1$, and then averages $(S_{(1)}(\mathbf{x})-S_{(2)}(\mathbf{x}))^2=4\hyp(\mathbf{x})^2$ over all four-element subsets.
	Since $\hyp$ scales linearly with the metric and vanishes on quadruples with a repeated point, the Q-residual of a metric space $X$ with $N\geq4$ points and the uniform measure $\mu_X$ equals
	\begin{align*}
		\frac{N(N-1)}{(N-2)(N-3)}\cdot\frac{4\,\W_2(\nu_{\hyp,X},\delta_0)^2}{\bigl(\int_\R t\,d\nu_{\dist,X}(t)\bigr)^2}.
	\end{align*}
	Up to a factor that tends to $1$ as $N\to\infty$, it is thus the squared $2$-Wasserstein distance between $\nu_{\hyp,X}$ and the hyperbolicity distribution of a tree, measured in units of the mean distance.

\end{example}

\begin{example}[Higher hyperbolicity \cite{higher-hyperbolicity}]
	\label{inv:higher_hyperbolicity}
	Let $(X,d_X)$ be a metric space, $m\geq1$, and let $I$ be an index set with $|I|=2(m+1)$.
	For $\mathbf{x}=(x_i)_{i\in I}\in X^{2(m+1)}$ and a permutation $\alpha$ of $I$, write
	\begin{align*}
		p_\alpha(\mathbf{x})
		&\coloneq\sum_{i\in I}d_X(x_i,x_{\alpha(i)}).
	\end{align*}
	Set $P(\mathbf{x}) \coloneq\bigl(p_\alpha(\mathbf{x})\bigr)_{\alpha \in \operatorname{Sym}(I)}$, with $\alpha$ ranging over all permutations of $I$, and let $P_{(1)}(\mathbf{x})\geq P_{(2)}(\mathbf{x})\geq\dots$ denote its components in descending order.
	The \emph{$m$-th hyperbolicity deficit} of $\mathbf{x}$ is the upper gap
	\begin{align*}
		\hyp_m(\mathbf{x})
		&\coloneq\frac{1}{2}\left(P_{(1)}(\mathbf{x})-P_{(2)}(\mathbf{x})\right).
	\end{align*}
	For $\delta\geq0$, the space $X$ is called \emph{$(m,\delta)$-hyperbolic} if $\hyp_m(\mathbf{x})\leq\delta$ for every $\mathbf{x}\in X^{2(m+1)}$.
	The \emph{$m$-th hyperbolicity} of $X$ is the optimal such bound,
	\begin{align*}
		\hyp_m(X)
		&\coloneq\sup_{\mathbf{x}\in X^{2(m+1)}}\hyp_m(\mathbf{x}).
	\end{align*}
	This generalizes Gromov's four-point hyperbolicity in the sense that every space with $\hyp_m(X)\leq\delta$ satisfies a uniform slim $(m+1)$-simplex condition \cite[Thm.~1.3]{higher-hyperbolicity}, while $\hyp_m(X)=0$ characterizes spaces of combinatorial dimension at most $m$ \cite[Cor.~4.4]{higher-hyperbolicity}.
	Indeed, when $m=1$, the identities $p_{(ij)(k\ell)}(\mathbf{x})=2s_{ij,k\ell}(\mathbf{x})$ and $p_{(ijk\ell)}(\mathbf{x})=s_{ij,k\ell}(\mathbf{x})+s_{i\ell,jk}(\mathbf{x})$, together with the triangle inequality, give $P_{(1)}(\mathbf{x})-P_{(2)}(\mathbf{x})= 2S_{(1)}(\mathbf{x}) - (S_{(1)}(\mathbf{x}) + S_{(2)}(\mathbf{x})) = S_{(1)}(\mathbf{x})-S_{(2)}(\mathbf{x})$.
	Hence $\hyp_1(\mathbf{x})=\hyp(\mathbf{x})$ for every $\mathbf{x}\in X^4$, and consequently $\hyp_1(X)=\hyp(X)$ (cf.\ \cite[Prop.~2.2]{higher-hyperbolicity}).

\begin{remark*}
	The provided definition of $(m,\delta)$-hyperbolicity is equivalent to the one in \cite[Def.~1.1]{higher-hyperbolicity}, which states that $X$ is $(m,\delta)$-hyperbolic if, for every $\mathbf{x}=(x_i)_{i\in I}$, there exists a permutation $\alpha\neq {-\operatorname{Id}}$ of $I=\{\pm 1, \pm2, \dots,\pm(m+1)\}$ such that
	\begin{align*}
		\sum_{i\in I}d_X(x_i,x_{-i})
		&\leq \sum_{i\in I}d_X(x_i,x_{\alpha(i)})+2\delta.
	\end{align*}
	In the notation above, this expression is simply $p_{-\operatorname{Id}}(\mathbf{x})\leq p_\alpha(\mathbf{x})+2\delta$.
	Suppose first that $\hyp_m(\mathbf{x})\leq\delta$ for every $\mathbf{x}$ and fix $\mathbf{x}$.
	We may assume $p_{-\operatorname{Id}}(\mathbf{x})=P_{(1)}(\mathbf{x})$, since otherwise $p_{-\operatorname{Id}}(\mathbf{x})\leq P_{(1)}(\mathbf{x})$ providing an $\alpha\neq -\operatorname{Id}$ that satisfies the desired inequality with $\delta=0$.
	Then there is some $\alpha\neq{-\operatorname{Id}}$ such that $P_{(2)}(\mathbf{x})=p_\alpha(\mathbf{x})$, which gives $p_{-\operatorname{Id}}(\mathbf{x})-p_\alpha(\mathbf{x})=2\hyp_m(\mathbf{x})\leq2\delta$.
	Conversely, suppose that $X$ satisfies \cite[Def.~1.1]{higher-hyperbolicity} and fix $\mathbf{x}$.
	We may assume $\hyp_m(\mathbf{x})>0$, since otherwise $\hyp_m(\mathbf{x})=0\leq\delta$ and there is nothing to prove.
	Then $P_{(1)}(\mathbf{x})>P_{(2)}(\mathbf{x})$, so $P_{(1)}(\mathbf{x})=p_\beta(\mathbf{x})$ for some unique permutation $\beta$.
	The identity $p_\beta(\mathbf{x})=p_{\beta^{-1}}(\mathbf{x})$ forces $\beta=\beta^{-1}$.
	By \cite[Lem.~2.1]{higher-hyperbolicity}, $\beta$ has no fixed points, hence is a product of $m+1$ disjoint transpositions and therefore conjugate to $-\operatorname{Id}$.
	As the hypothesis is assumed for every $\mathbf{x}$, we may relabel and take $\beta={-\operatorname{Id}}$.
	Now \cite[Def.~1.1]{higher-hyperbolicity} yields some $\alpha\neq{-\operatorname{Id}}$ with $P_{(1)}(\mathbf{x})\leq p_\alpha(\mathbf{x})+2\delta\leq P_{(2)}(\mathbf{x})+2\delta$, that is, $\hyp_m(\mathbf{x})\leq\delta$.
\end{remark*}

\end{example}

\begin{example}[$(n,k)$-persistence diagram \cite{GomezMemoli2024}]
	\label{inv:persistence-set}
	For a metric space $(X,d_X)$ and $r \geq 0$, the Vietoris--Rips complex is
	\begin{align*}
		\vr_r(X)
		&\coloneq \{\sigma \subset X \setcolon |\sigma| < \infty \text{ and } \diam(\sigma) \leq r\},
	\end{align*}
	and we write $\dgmVR_k(X)$ for its degree-$k$ persistence diagram.
	The construction uses only distances, so we obtain an $n$-invariant by extending it to distance matrices
	\begin{align*}
		\dgmVR_{n,k} \colon \Kn_n \to \overline{\cD},
	\end{align*}
	where $\overline{\cD}$ is a suitable space of persistence diagrams.
	This is the factorization of $\dgmVR_{n,k}$ through $\Psi_n$ observed in \cite{GomezMemoli2024}, which also shows that it is continuous and therefore measurable.
	Following \cite[Defs.~3.10 and~3.18]{GomezMemoli2024}, we write $\Dvr{n,k}(X) \coloneq \dgmVR_{n,k}(\Kn_n(X))$ for the \emph{$(n,k)$-persistence set} of $X$, the collection of all degree-$k$ diagrams realized by at most $n$ points of $X$, and $\Uvr{n,k}(X) \coloneq \nu_{\dgmVR_{n,k},X}$ for the \emph{$(n,k)$-persistence measure} of $X$, which records how often each of them occurs.

	In \Cref{def:n_invariant} we require the target to be Polish.
	The unrestricted space of countable diagrams with the bottleneck distance is not separable \cite[Ex.~5.3]{CheEtAl2024}, but \cite{Blumberg2014} proved that the sets $\cD_j$ of diagrams with at most $j$ points and the completion $\overline{\cD}$ of $\bigcup_{j \geq 0} \cD_j$ with the bottleneck distance are both Polish \cite[Thms.~3.2 and~3.5]{Blumberg2014}.
	Since a space with $n$ points has only finitely many simplices at every scale, its diagram has finitely many points and thus lies in $\overline{\cD}$.
	In fact, $(\overline{\cD}, d_B)$ is the space of choice of \cite{Blumberg2014} for studying empirical distributions of persistence diagrams.
	Another possible choice is the set of diagrams with finite total persistence equipped with the $p$-Wasserstein distance; see \cite[Def.~4, Thms.~6 and~12]{Mileyko2011}.

	In phylogenetics, the $(4,1)$-persistence invariant $\dgmVR_{4,1}$ has a simple geometric interpretation.
	It is nonempty precisely when the four points form a chordless cycle in the Vietoris--Rips complex at some scale; its unique class is born when the last cycle edge appears and dies when the first diagonal appears.
	For sequence data, the cycle condition means that every adjacent pair agrees at more variable sites than either opposite pair.
	Thus, $(4,1)$-persistence provides a scale-sensitive summary of competing repeated-state patterns within quadruples of sequences, which are associated with recurrent and reticulate processes such as homoplasy and recombination.
\end{example}

\begin{prop}
	\label{prop:lipschitz-constants}
	The $n$-invariants in \Cref{inv:distance,inv:hyperbolicity,inv:higher_hyperbolicity,inv:persistence-set} satisfy the following Lipschitz bounds.
	\begin{enumerate}
		\item The distance map $\dist \colon \Kn_2 \to \R$ is $1$-Lipschitz.
		\item The four-point hyperbolicity deficit $\hyp \colon \Kn_4 \to \R$ is $2$-Lipschitz.
		Moreover, the function $X \mapsto \hyp(X)$ on metric spaces is $4$-Lipschitz with respect to the Gromov--Hausdorff distance.
		\item For every $m \geq 1$, the higher-hyperbolicity invariant $\hyp_m \colon \Kn_{2(m+1)} \to \R$ is $2(m+1)$-Lipschitz.
		Moreover, the function $X \mapsto \hyp_m(X)$ on metric spaces is $4(m+1)$-Lipschitz with respect to the Gromov--Hausdorff distance.
		\item For every $n \geq 2$ and $k \geq 0$, the persistence-diagram invariant $\dgmVR_{n,k} \colon \Kn_n \to \overline{\cD}$ is $1$-Lipschitz with respect to the supremum norm on $\Kn_n$ and the bottleneck distance on $\overline{\cD}$.
	\end{enumerate}
\end{prop}

\begin{proof}
	\begin{enumerate}
		\item For $M,M' \in \Kn_2\simeq \R_{\geq 0}$, we have $|\dist(M)-\dist(M')|=\|M-M'\|_\infty$.
		\item For $M,M'\in\Kn_4$, let $S(M)$ and $S(M')$ be the vectors of the three distance sums in \Cref{inv:hyperbolicity}.
		If $D\coloneq\|M-M'\|_\infty$, then $\|S(M)-S(M')\|_\infty\leq 2D$.
		Taking either the largest or the second-largest component is $1$-Lipschitz with respect to the supremum norm, so
		\begin{align*}
			|\hyp(M)-\hyp(M')|
			&\leq \frac{1}{2}\left(2D+2D\right)
			=2D.
		\end{align*}
		The $4$-Lipschitz bound for $X\mapsto\hyp(X)$ is \cite[Prop.~A.8]{graph-approximations-geodesic-spaces}.

		\item For $M,M'\in\Kn_{2(m+1)}$, let $P(M)$ and $P(M')$ be the vectors of the distance sums in \Cref{inv:higher_hyperbolicity}.
		If $D\coloneq\|M-M'\|_\infty$, then $\|P(M)-P(M')\|_\infty\leq2(m+1)D$.
		Taking either the largest or the second-largest component is $1$-Lipschitz with respect to the supremum norm, so
		\begin{align*}
			|\hyp_m(M)-\hyp_m(M')|
			&\leq \frac{1}{2}\left(2(m+1)D+2(m+1)D\right)
			=2(m+1)D.
		\end{align*}

		Now let $R$ be a correspondence between metric spaces $X$ and $Y$.
		For any $\mathbf{x}\in X^{2(m+1)}$, choose $\mathbf{y}\in Y^{2(m+1)}$ componentwise so that $(x_i,y_i)\in R$ for every $i$.
		The resulting pair of distance matrices differs by at most the distortion $\dis(R)$ in the supremum norm, and therefore
		\begin{align*}
			\hyp_m(\mathbf{x})
			&\leq \hyp_m(\mathbf{y})+2(m+1)\dis(R) \\
			&\leq \hyp_m(Y)+2(m+1)\dis(R).
		\end{align*}
		Taking the supremum over $\mathbf{x}$, exchanging $X$ and $Y$, and then taking the infimum over all correspondences gives
		\begin{align*}
			|\hyp_m(X)-\hyp_m(Y)|
			&\leq 4(m+1)\GH(X,Y).
		\end{align*}

		\item Let $M,M'\in\Kn_n$ and let $X_M,X_{M'}$ be their finite labelled realizations.
		By stability of persistent homology \cite[Thm.~3.1]{ccsg09}\footnote{The cited reference has no factor~$2$ in front of $\GH$ because it uses the Vietoris--Rips filtration with parameter~$2r$.} and \Cref{lemma:GH_vs_Kn},
		\begin{align*}
			d_B\bigl(\dgmVR_{n,k}(M),\dgmVR_{n,k}(M')\bigr)
			&\leq 2\GH(X_M,X_{M'}) 
			\leq \|M-M'\|_\infty.
		\end{align*}
	\end{enumerate}
\end{proof}

\begin{remark}
	\label{rmk:lipschitz-hyperbolicity-constants}
	When $m=1$, the bound in part \textit{(iii)} is not sharp: since $\hyp_1=\hyp$, the sharper bound in part \textit{(ii)} applies.
	The discrepancy arises because, when $m=1$, the higher-hyperbolicity condition is a somewhat inefficient reformulation of the four-point condition due to the additional permutations and repeated distance terms; see \cite[Def.~1.1]{higher-hyperbolicity}.
\end{remark}

\section{Geometry of \texorpdfstring{$\Kn_n(X)$}{K\_n(X)}}
\label{sec:geometry-Kn}

This section establishes the geometric properties of $\Kn_n(X)$ required for the proofs of \Cref{thm:main-sampling,thm:main-sampling-intrinsic} in \Cref{sec:convergence-rates}.
We begin by computing the minimum enclosing radius of $\Kn_n(X)$ and the moments of $\mu_n$, and then bound the dimensions of $\Kn_n(X)$ and $\mu_n$.
Both rest on elementary Lipschitz and diameter estimates for $\Psi_n$, which we establish first.

\begin{lemma}
	\label{lemma:metrics}
	Let $(X, d_X)$ be a metric space. Then for every $x_1, x_2, x_1', x_2' \in X$,
	\begin{enumerate}
		\item\label{it:metric_1} $\lvert d_X(x_1,x_2) - d_X(x_1',x_2') \rvert \leq d_X(x_1,x_1') + d_X(x_2,x_2')$.
		\item\label{it:metric_2} $\lvert d_X(x_1, x_2) - d_X(x_1', x_2') \rvert \leq \diam(X)$.
	\end{enumerate}
\end{lemma}
\begin{proof}
	By the triangle inequality,
	\begin{equation*}
		d_X(x_1,x_2) - d_X(x_1',x_2') \leq d_X(x_1,x_1') + d_X(x_1',x_2') + d_X(x_2',x_2) - d_X(x_1',x_2') = d_X(x_1,x_1') + d_X(x_2,x_2').
	\end{equation*}
	The symmetric argument yields~\ref{it:metric_1}.
	On the other hand, $d_X(x_1,x_2) - d_X(x_1',x_2') \leq \diam(X) - 0 = \diam(X)$, and~\ref{it:metric_2} follows by symmetry once again.
\end{proof}

Since we embed $\Kn_n$ in either $\R^{n \times n}$ or $\R^{\binom{n}{2}}$, we first record the relationship between the metrics on these two spaces.
\begin{lemma}
	\label{lemma:metrics_Kn}
	Fix $1 \leq p \leq \infty$ and $n \geq 2$, and let $d \coloneq \binom{n}{2}$. Let $\|\cdot\|_{\ell^p(\R^m)}$ be the $\ell^p$ norm on $\R^m$. Then for every $M, M' \in \Kn_n$,
	\begin{equation*}
		\|M - M'\|_{\ell^p(\R^d)} = 2^{-1/p} \|M - M'\|_{\ell^p(\R^{n \times n})}.
	\end{equation*}
\end{lemma}
\begin{proof}
	If $p < \infty$,
	\begin{align*}
		\|M - M'\|_{\ell^p(\R^d)}
		&= \left( \sum_{1 \leq i < j \leq n} |d_X(x_i, x_j) - d_X(x_i', x_j')|^p \right)^{1/p}
		= \left( \frac{1}{2} \sum_{1 \leq i \neq j \leq n} |d_X(x_i, x_j) - d_X(x_i', x_j')|^p \right)^{1/p} \\
		&= 2^{-1/p} \|M - M'\|_{\ell^p(\R^{n \times n})}.
	\end{align*}
	The result for $p=\infty$ follows by letting $p \to \infty$.
\end{proof}

The next Proposition gives control over the curvature map $\Psi_n$ and the diameter of $\Kn_n(X)$.

\begin{prop}
	\label{prop:props_Kn}
	Let $(X, d_X)$ be a metric space.
	Fix $n \geq 2$ and $1 \leq p, q \leq \infty$.
	Let $d \coloneq \binom{n}{2}$.
	Equip $X^n$ with the $\ell^p$ product metric and $\Kn_n(X)$ with the $\ell^q$ metric coming from $\R^{d}$.
	Then:
	\begin{enumerate}
			\item\label{it:prop_Kn_Psi} $\Psi_n \colon X^n \to \Kn_n(X) \subset \R^d$ is Lipschitz continuous with constant $2^{1-1/q}(n-1)^{1/q}$ if $p \leq q$ and $2^{1-1/q}(n-1)^{1/q} n^{1/q-1/p}$ if $p > q$.
			In particular, the Lipschitz constant is 2 when $p = q = \infty$.
		\item\label{it:prop_Kn_diam} $\diam_q(\Kn_n(X)) \leq d^{1/q} \diam(X)$.
			In particular, $\diam_\infty(\Kn_n(X)) = \diam(X)$.
			Additionally, $\Kn_n(X)$ is unbounded if $X$ is.
		\item\label{it:prop_Kn_compact} If $X$ is compact, so is $\Kn_n(X)$.
	\end{enumerate}
\end{prop}

\begin{proof}
	Let $\mathbf{x}, \mathbf{x}' \in X^n$ and $M = \Psi_n(\mathbf{x}), M' = \Psi_n(\mathbf{x}') \in \Kn_n(X)$.

	\begin{enumerate}
		\item When $q<\infty$, \Cref{lemma:metrics}\ref{it:metric_1} and the Minkowski inequality yield
		\begin{align*}
			\|M - M'\|_{\ell^q(\R^{n \times n})}
			&= \left( \sum_{1 \leq i \neq j \leq n} |d_X(x_i, x_j) - d_X(x_i', x_j')|^q \right)^{1/q}\\
			&\leq \left( \sum_{1 \leq i \neq j \leq n} \left( d_X(x_i, x_i') + d_X(x_j', x_j) \right)^q \right)^{1/q} \\
			&\leq \left( \sum_{1 \leq i \neq j \leq n} d_X(x_i, x_i')^q \right)^{1/q} + \left( \sum_{1 \leq i \neq j \leq n} d_X(x_j, x_j')^q \right)^{1/q} \\
			&= \left( (n-1) \sum_{1 \leq i \leq n} d_X(x_i, x_i')^q \right)^{1/q} + \left( (n-1) \sum_{1 \leq j \leq n} d_X(x_j, x_j')^q \right)^{1/q} \\
			&= 2(n-1)^{1/q} d_{X,q}(\mathbf{x}, \mathbf{x}').
		\end{align*}
		Together with \Cref{lemma:metrics_Kn}, we obtain $\|M - M'\|_{\ell^q(\R^d)} = 2^{-1/q} \|M - M'\|_{\ell^q(\R^{n \times n})} \leq 2^{1-1/q} (n-1)^{1/q} d_{X,q}(\mathbf{x}, \mathbf{x}')$.
		The result follows from the basic inequalities $d_{X,q} \leq d_{X,p}$ if $p \leq q$ and $d_{X,q} \leq n^{1/q-1/p} \cdot d_{X,p}$ if $p > q$. Once again, the $q=\infty$ case follows by taking a limit.

		\item When $q<\infty$, \Cref{lemma:metrics}\ref{it:metric_2} gives
		\begin{equation*}
				\begin{multlined}[0.7\displaywidth]
				\|M - M'\|_{\ell^q(\R^d)}
				= \left( \sum_{1 \leq i < j \leq n} |d_X(x_i, x_j) - d_X(x_i', x_j')|^q \right)^{1/q} \\
				\leq \left( \sum_{1 \leq i < j \leq n} \diam(X)^q \right)^{1/q}
				= d^{1/q} \diam(X).
			\end{multlined}
		\end{equation*}
		Thus, $\diam_q(\Kn_n(X)) \leq d^{1/q} \diam(X)$.
		On the other hand, choose a sequence $\{(x_k, x_k')\}_{k=0}^\infty \subset X^2$ with $x_0 = x_0'$ such that $d_X(x_k, x_k') \to \diam(X)$ as $k \to \infty$.
		Let $M_k \coloneq \Psi_n(x_k', x_k, \dots, x_k) \in \Kn_n(X)$.
		Then
		\begin{equation*}
			\|M_k - M_0\|_{\ell^q(\R^d)} = (n-1)^{1/q} d_X(x_k, x_k') \to (n-1)^{1/q} \diam(X)
		\end{equation*}
		as $k \to \infty$. Hence, $\diam(\Kn_n(X)) = \infty$ if $\diam(X) = \infty$. If $q=\infty$, letting $q \to \infty$ in the above gives $\diam_\infty(\Kn_n(X)) \leq \diam(X)$, $\|M_k-M_0\|_{\ell^\infty(\R^d)} \to \diam(X)$ and, thus, $\diam_\infty(\Kn_n(X)) = \diam(X)$.

		\item Regardless of the values of $p$ and $q$, if $X$ is compact, then $\Kn_n(X) = \Psi_n(X^n)$ is the continuous image of a compact set, hence compact.
	\end{enumerate}

\end{proof}

We now apply \Cref{prop:moment_radius} to the curvature set $\Kn_n(X)$ and the curvature measure $\mu_n$.
The results of~\cite{Fournier2023} leave room to choose the centre $\mathbf{p}_0$ in $\cM_q^{(m)}(\mu_n, \mathbf{p}_0)$, so we will use the symmetry of $\Kn_n(X)$ and \Cref{prop:props_Kn} to minimize $\cM_q^{(m)}(\mu_n, \mathbf{p}_0)$.

\begin{prop}
	\label{prop:props_Kn_2}
	Let $(X, d_X, \mu_X)$ be an mm-space with fully supported $\mu_X$. Fix $n \geq 2$ and let $d \coloneq \binom{n}{2}$. Let $1 \leq q < \infty$ and $1 \leq m \leq \infty$. Let $\mu_n$ be the associated curvature measure.
	\begin{enumerate}
		\item\label{it:measure_radius} If $m < \infty$, the minimum enclosing radius of $\Kn_n(X)$ in $(\R^d, \ell^m)$ is at most $d^{1/m} \diam(X)/2$. If $m = \infty$, $R_\infty(\Kn_n(X)) = \diam(X)/2$.
		\item\label{it:measure_moment} $\cM_q^{(m)}(\mu_n)^{1/q} \leq d^{1/m} \diam(X)/2$ and $\sup_{q \geq 1} \cM_q^{(\infty)}(\mu_n)^{1/q} = \diam(X)/2$. In particular, if $\diam(X) < \infty$, then $\mu_n$ has finite moments of all orders.
	\end{enumerate}
\end{prop}
\begin{proof}
	If $\diam(X) = \infty$, the Proposition holds immediately because $\diam_p(\Kn_n(X)) = \infty$ by \Cref{prop:props_Kn}\ref{it:prop_Kn_diam} and, thus, $R_p(\Kn_n(X)) = \infty$ for all $p$. Hence, assume $\diam_\infty(\Kn_n(X)) = \diam(X) < \infty$, and let $D_0 \in \R^d$ be the vector with all components equal to $\diam(X)/2$.
	\begin{enumerate}
		\item For all $\mathbf{x} \in X^n$,
		\begin{equation*}
			\label{eq:MER}
			\| \Psi_n(\mathbf{x}) - D_0\|_{\ell^m(\R^d)}^m
			= \sum_{1 \leq i < j \leq n} |d_X(x_i, x_j) - \diam(X)/2|^m
			\leq \sum_{1 \leq i < j \leq n} (\diam(X)/2)^m
			= d (\diam(X)/2)^m.
		\end{equation*}
		Equality holds if $x_i = x_j$ for all $1 \leq i < j \leq n$. Hence, $R_m(\Kn_n(X)) \leq d^{1/m} \diam(X)/2$.

		If $m = \infty$, $R_\infty(\Kn_n(X)) \leq \diam(X)/2$ follows analogously.
		On the other hand, let $M_0 \in \R^d$ such that $\|M - M_0\|_{\infty} \leq R_\infty(\Kn_n(X))$ for all $M \in \Kn_n(X)$. Given $\epsilon > 0$, choose $M, M' \in \Kn_n(X)$ such that $\diam(X) - \epsilon \leq \|M-M'\|_\infty$.
		By the triangle inequality,
		\begin{equation*}
			\diam(X) - \epsilon \leq \|M - M'\|_\infty \leq \|M - M_0\|_\infty + \|M_0 - M'\|_\infty \leq 2 R_\infty(\Kn_n(X)).
		\end{equation*}
		Letting $\epsilon \to 0$ yields $\diam(X)/2 \leq R_\infty(\Kn_n(X))$.
		This proves~\ref{it:measure_radius}.

		\item This is a direct consequence of \Cref{prop:moment_radius} and \Cref{it:measure_radius}.
	\end{enumerate}
\end{proof}

Now we deal with the intrinsic dimension of $\Kn_n(X)$. First, we prove that the product of Ahlfors regular spaces is also Ahlfors regular.
\begin{lemma}
	\label{lemma:ahlfors-product}
	Let $(X, d_X, \mu_X)$ and $(Y, d_Y, \mu_Y)$ be Ahlfors regular of dimensions $s$ and $t$, respectively. Let $d_\infty$ be the metric $d_\infty((x,y), (x',y')) = \max(d_X(x,x'), d_Y(y,y'))$. Then $(X \times Y, d_\infty, \mu_X \otimes \mu_Y)$ is Ahlfors $(s+t)$-regular.
\end{lemma}
\begin{proof}
	Suppose that for every $x_0 \in X$, $y_0 \in Y$, and $r>0$,
	\begin{align*}
		\frac{1}{C_X}r^{s} \leq \mu_X(B_X(x_0,r)) \leq C_X r^{s}\\
		\frac{1}{C_Y}r^{t} \leq \mu_Y(B_Y(y_0,r)) \leq C_Y r^{t}.
	\end{align*}
	Since $B_{X \times Y}((x_0, y_0), r) = B_X(x_0, r) \times B_Y(y_0,r)$, we get
	\begin{equation*}
		\frac{1}{C_X C_Y} \cdot r^{s+t} \leq (\mu_X \otimes \mu_Y)\big( B_{X \times Y}((x_0, y_0), r) \big) \leq (C_X C_Y) \cdot r^{s+t}.
	\end{equation*}
\end{proof}

Unfortunately, the Ahlfors condition does not descend immediately from $X^n$ to $\Kn_n(X)$. Still, we can use comparisons from \cite{WeedBach2019} to provide an upper bound for the dimensions of $\Kn_n(X)$.
\begin{prop}
	\label{prop:ahlfors-regularity}
	Suppose that $(X, d_X, \mu_X)$ is an Ahlfors $s$-regular mm-space and fix $1 \leq m \leq \infty$. Equip $X^n$ and $\Kn_n(X)$ with the metrics $d_{X,m}$ and $\ell^m$, respectively. Then
	\begin{enumerate}
		\item\label{it:ahlfors_X^n} $d_H(X^n) = d_M(X^n) = ns$, and $d_p^*(\mu_X^{\otimes n}) = ns$ when $1 \leq p \leq ns/2$.
		\item\label{it:ahlfors_Kn} $d_p^*(\mu_n) \leq d_M(\Kn_n(X)) \leq ns$ when $1 \leq p \leq \frac{1}{2}d_M(\Kn_n(X))$.
	\end{enumerate}
\end{prop}
\begin{proof}
	\hfill
	\begin{enumerate}
		\item By \Cref{lemma:ahlfors-product}, $(X^n, d_{X,\infty}, \mu_X^{\otimes n})$ is Ahlfors $(ns)$-regular. Since $d_{X,m}$ and $d_{X,\infty}$ are bi-Lipschitz equivalent, $(X^n, d_{X,m}, \mu_X^{\otimes n})$ is also Ahlfors $(ns)$-regular. The conclusion follows by \Cref{lemma:ahlfors-dimensions}.
		\item Since $\mu_X$ is fully supported by \Cref{lemma:ahlfors-compact}, so is $\mu_X^{\otimes n}$, and by continuity of $\Psi_n$ so is $\mu_n$ on $\Kn_n(X)$.
		Hence $d_M(\mu_n) = d_M(\Kn_n(X))$, and \cite[Prop.~2]{WeedBach2019} yields $d_p^*(\mu_n) \leq d_M(\Kn_n(X))$ when $p \leq \frac{1}{2} d_M(\Kn_n(X))$.
		Moreover, by \Cref{prop:props_Kn}\ref{it:prop_Kn_Psi}, $\Psi_n:(X^n, d_{X,m}) \to (\Kn_n(X), \ell^m)$ is Lipschitz continuous, so by \cite[Lem.~3.3(iv)]{Dimensions-Embeddings-Attractors} and \ref{it:ahlfors_X^n}, $d_M(\Kn_n(X)) \leq d_M(X^n) = ns$.
	\end{enumerate}
\end{proof}

\begin{remark}
	\label{rmk:ahlfors_dim_Kn_S1}
	Note that $d_M(\Kn_n(X)) \leq ns$ in \Cref{prop:ahlfors-regularity} is not an equality in general. For example, \cite{curvature-sets-S1} proved that $\Kn_n(\Sp^1)$ is the geometric realization of a simplicial complex whose maximal simplices all have dimension $n-1$. Then $d_H(\Kn_n(\Sp^1)) = d_M(\Kn_n(\Sp^1)) = n-1$ because $\Kn_n(\Sp^1)$ is a finite union of simplices $\Delta_{n-1} \subset \R^d$ (with either $d=n \times n$ or $d = \binom{n}{2}$), all of whom have $d_H(\Delta_{n-1}) = d_M(\Delta_{n-1}) = n-1$. On the other hand, the uniform measure $\mu_{\Sp^1}$ satisfies
	\begin{equation*}
		\mu_{\Sp^1}(B_{\Sp^1}(x_0, r)) = r/\pi
	\end{equation*}
	for $0 < r \leq \pi$, so $\mu_{\Sp^1}$ is Ahlfors 1-regular. Hence, $d_M((\Sp^1)^n) = n > d_M(\Kn_n(\Sp^1))$.
\end{remark}

\section{Wasserstein convergence of empirical curvature measures}
\label{sec:convergence-rates}

In this section we prove our main results \labelcref{thm:main-sampling}--\labelcref{thm:main-sampling-intrinsic}. The main tool for \Cref{thm:main-sampling} is the Wasserstein convergence of empirical measures on Euclidean space, established by Fournier--Guillin \cite{convergence_empirical_fournier}, in the form of \cite{Fournier2023}, which provides an explicit constant.
We apply their results to $\mu_n$, bounding its moment term by the diameter of $X$ with the results of \Cref{sec:geometry-Kn}.
Lipschitz $n$-invariants inherit this bound (\Cref{cor:main-sampling-invariant}), and \Cref{lemma:empirical_sample_size} turns it into sample sizes (\Cref{cor:main-sampling-invariant-bounds}).
We improve the convergence rate for Ahlfors regular spaces in \Cref{thm:main-sampling-intrinsic}, by combining the estimates of \cite{WeedBach2019} with the dimension bounds of \Cref{sec:geometry-Kn}.

The following are \Cref{thm:main-sampling} and \Cref{cor:main-sampling-invariant}.
\begin{theorem}
	\label{thm:empirical_Kn}
	Fix a real number $1 \leq p < \infty$ and integers $N \geq 1$, $n \geq 2$.
	Let $d \coloneq \binom{n}{2}$ and assume $p\neq d/2$.
	Let $(X, d_X, \mu_X)$ be a bounded mm-space with fully supported probability measure $\mu_X$.
	Embed $\Kn_n(X)$ in $(\R^d, \| \cdot \|_\infty)$.
	Then
	\begin{equation*}
		\bbE[\W_p(\empirical{\mu}{n}{N}, \mu_n)^p] \leq C(d,p) \cdot \diam(X)^p \cdot N^{-e_{d,p}},
		\qquad
		e_{d,p} \coloneq
		\begin{cases}
			1/2, & d < 2p\\
			p/d, & d > 2p,
		\end{cases}
	\end{equation*}
	where $C(d,p) \coloneq \dfrac{2^{d/2-1}}{1 - 2^{d/2-p}}$ if $d < 2p$ and $C(d,p) \coloneq \dfrac{2^{p-2p/d}(1 - 2^{-d/2})^{1-2p/d}}{1-2^{p-d/2}}$ if $d > 2p$.
\end{theorem}
\begin{proof}
	Write $D_X \coloneq \diam(X)$.
	Since the metric on $\Kn_n(X)$ is $\| \cdot \|_\infty$, we apply \cite[Thm.~2.1]{Fournier2023} with $m = \infty$. Then
	\begin{equation*}
		\bbE[\W_p(\empirical{\mu}{n}{N}, \mu_n)^p] \leq 2^p \kappa_{d,p}^{(\infty)} \left( \cM_{q}^{(\infty)}(\mu_n) \right)^{p/q} \theta_{d,p,q}^{(\infty)} N^{-e_{d,p}},
	\end{equation*}
	with $e_{d,p}$ as in the statement.
	By \Cref{prop:moment_radius,prop:props_Kn_2}\ref{it:measure_moment}, $2^p \left( \cM_{q}^{(\infty)}(\mu_n) \right)^{p/q} \leq 2^p \left( \cM_{\infty}^{(\infty)}(\mu_n) \right)^{p} = D_X^p$. Hence,
	\begin{equation*}
		\bbE[\W_p(\empirical{\mu}{n}{N}, \mu_n)^p] \leq \kappa_{d,p}^{(\infty)} D_X^p \theta_{d,p,q}^{(\infty)} N^{-e_{d,p}}.
	\end{equation*}
		By the comments in \cite[Sec.~2.4]{Fournier2023}, the function $q \mapsto \theta_{d,p,q}^{(m)}$ is decreasing and converges to $1$ as $q \to \infty$.
		Taking the infimum over all $q$ and setting $C(d,p) \coloneq \kappa_{d,p}^{(\infty)}$ yields the result.
\end{proof}

\begin{corollary}
	\label{cor:empirical_Kn_invariant}
	Under the assumptions of \Cref{thm:empirical_Kn}, let $F \colon \Kn_n \to Z$ be a Lipschitz $n$-invariant. Then
	\begin{equation*}
		\bbE[\W_p(\empirical{\nu}{F}{N}, \nu_F)^p] \leq C(d,p) \cdot \bigl(\operatorname{Lip}(F) \diam(X)\bigr)^p \cdot N^{-e_{d,p}}.
	\end{equation*}
\end{corollary}
\begin{proof}
	Since $\empirical{\nu}{F}{N} = F_\#(\empirical{\mu}{n}{N})$ and $\nu_F = F_\#(\mu_n)$, \Cref{lemma:wasserstein_lipschitz} yields
	\begin{equation*}
		\W_p(\empirical{\nu}{F}{N}, \nu_F) \leq \operatorname{Lip}(F) \, \W_p(\empirical{\mu}{n}{N}, \mu_n)
	\end{equation*}
	for every realization of the sample.
	Taking $p$-th powers and expectations and applying \Cref{thm:empirical_Kn} gives the claim.
\end{proof}

In other words, $\empirical{\nu}{F}{N}$ converges at the same rate as $\empirical{\mu}{n}{N}$, up to an additional multiplicative constant that has the same effect as scaling $X$ by $\operatorname{Lip}(F)$.
By \Cref{lemma:empirical_sample_size}, these bounds immediately tell us how many $n$-tuples suffice for the empirical measures to approximate $\mu_n$ and $\nu_F$ up to a given error with a given confidence, which is the content of \Cref{cor:main-sampling-invariant-bounds}.
\begin{corollary}
	\label{cor:empirical_Kn_samples}
	Under the assumptions of \Cref{thm:empirical_Kn}, choose an error margin $\epsilon > 0$ and a confidence level $0 < \alpha < 1$, and let $N_\ast$ be the positive solution of
	\begin{equation*}
		C(d,p) N_\ast^{-e_{d,p}} + \sqrt{-\tfrac{1}{2} \ln(1-\alpha)} N_\ast^{-1/2} = \epsilon^p.
	\end{equation*}
	Then for all $N > N_\ast$, with probability at least $\alpha$, we have $\W_p(\empirical{\mu}{n}{N}, \mu_n) \leq \epsilon \diam(X)$ and, uniformly over all Lipschitz $n$-invariants $F$, $\W_p(\empirical{\nu}{F}{N}, \nu_F) \leq \epsilon \operatorname{Lip}(F) \diam(X)$.
\end{corollary}
\begin{proof}
	Write $D_X \coloneq \diam(X)$.
	We may assume $D_X > 0$, since otherwise both sides vanish.
	The support of $\mu_n$ is contained in the closure of $\Kn_n(X)$ in $\R^d$, which is compact of diameter $D_X$ by \Cref{prop:props_Kn}\ref{it:prop_Kn_diam}.
	By \Cref{thm:empirical_Kn}, \Cref{lemma:empirical_sample_size} applies with $D \coloneq D_X$ and $\beta(N) \coloneq C(d,p) D_X^p N^{-e_{d,p}}$, and its condition on $N_\ast$ is the displayed equation.
	This gives $\W_p(\empirical{\mu}{n}{N}, \mu_n) < \epsilon D_X$ with probability at least $\alpha$.
	In this event, \Cref{lemma:wasserstein_lipschitz} gives $\W_p(\empirical{\nu}{F}{N}, \nu_F) \leq \operatorname{Lip}(F) \, \W_p(\empirical{\mu}{n}{N}, \mu_n) \leq \epsilon \operatorname{Lip}(F) D_X$ for all Lipschitz $F$.
\end{proof}

\begin{remark}
	\label{rmk:uniform-sample-reuse}
	Since \Cref{cor:empirical_Kn_samples} is uniform in $F$, the same i.i.d.\ sample of $n$-tuples may be used to obtain the empirical distributions of any number of $n$-invariants.
	While this makes the observed empirical distributions dependent, the Wasserstein bound holds for all of them simultaneously with the same confidence level $\alpha$.
	By contrast, if $R$ separate sampling runs are used and all $R$ guarantees are to hold with overall confidence $\alpha$, the union bound tells us that the per-run confidence has to be set to $1-(1-\alpha)/R$.
\end{remark}

Since $F$ is Lipschitz, its range satisfies $\diam F(\Kn_n(X)) \leq \operatorname{Lip}(F) \diam(X)$.
When $F$ satisfies the more stringent bound $\diam F(\Kn_n(X)) \leq \lambda \diam(X)$ for some $\lambda \leq \operatorname{Lip}(F)$, it pays to apply \Cref{lemma:empirical_sample_size} to $\nu_F$ directly, instead of moving through $\mu_n$, which gives the following refinement.
\begin{corollary}
	\label{cor:empirical_Kn_invariant_range}
	Under the assumptions of \Cref{cor:empirical_Kn_samples}, assume additionally that $\diam F(\Kn_n(X)) \leq \lambda \diam(X)$ for some $\lambda > 0$.
	Choose an error margin $\epsilon > 0$ and a confidence level $0 < \alpha < 1$, and let $N_\lambda$ be the positive solution of
	\begin{equation*}
		\operatorname{Lip}(F)^p C(d,p) N_\lambda^{-e_{d,p}} + \lambda^p \sqrt{-\tfrac{1}{2} \ln(1-\alpha)} N_\lambda^{-1/2} = \epsilon^p.
	\end{equation*}
	Then for all $N > N_\lambda$, $\W_p(\empirical{\nu}{F}{N}, \nu_F) \leq \epsilon \diam(X)$ with probability at least $\alpha$.
\end{corollary}
\begin{proof}
	Write $D_X \coloneq \diam(X)$.
	We may again assume $D_X > 0$.
	Since $\Kn_n(X)$ is bounded (\Cref{prop:props_Kn}\ref{it:prop_Kn_diam}) and $F$ is Lipschitz, $F(\Kn_n(X))$ is totally bounded, so its closure in the Polish space $Z$ is a compact set of diameter $\lambda D_X$ containing $\supp \nu_F$.
	By \Cref{cor:empirical_Kn_invariant}, \Cref{lemma:empirical_sample_size} applies with $D \coloneq \lambda D_X$, $\beta(N) \coloneq \operatorname{Lip}(F)^p C(d,p) D_X^p N^{-e_{d,p}}$, and error margin $\epsilon/\lambda$, and its condition on $N_\lambda$ is the displayed equation divided by $\lambda^p$.
	This gives $\W_p(\empirical{\nu}{F}{N}, \nu_F) < \epsilon D_X$ with probability at least $\alpha$.
\end{proof}

\begin{figure}
	\centering
	\includegraphics[width=\textwidth]{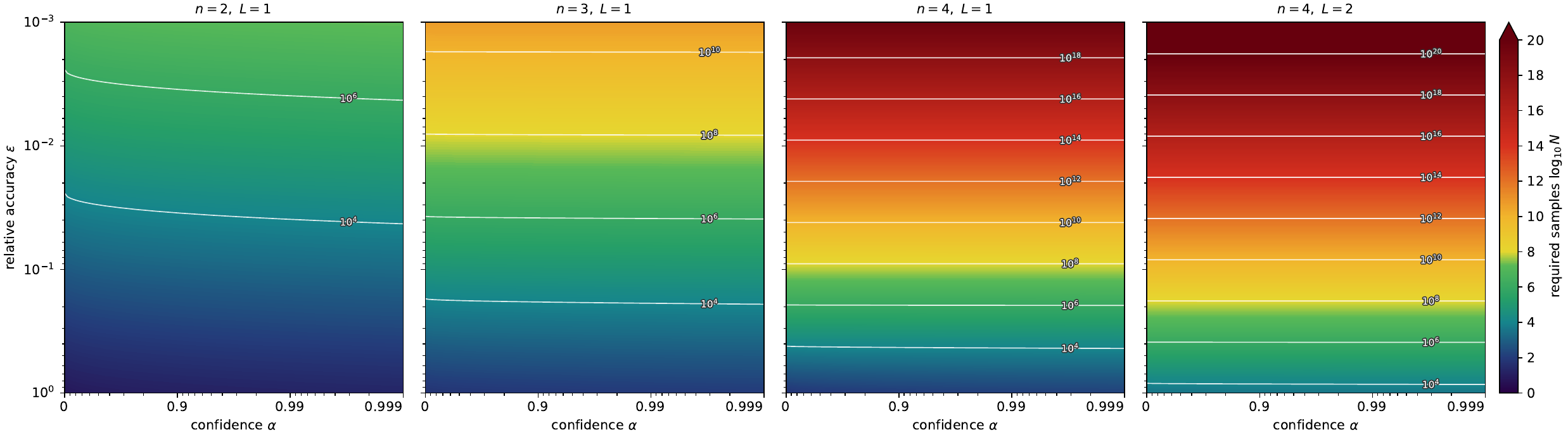}
	\caption{Required sample count $N_{\lambda=1}$ from \Cref{cor:empirical_Kn_invariant_range} for estimating $\nu_F$ up to $\W_1$ error $\epsilon \diam(X)$, plotted against confidence $\alpha$ and accuracy $\epsilon$.
	White contours mark powers of ten in $N$.
	The colours indicate sampling cost from feasible (blue/green) through expensive (yellow) to prohibitive (red) as $N$ increases.}
	\label{fig:relative-sample-size-nomogram}
\end{figure}

To reach a $\W_1$ error of $\epsilon \diam(X)$ from \Cref{cor:empirical_Kn_samples}, one has to ask for an accuracy of $\epsilon/\operatorname{Lip}(F)$.
Since $\lambda$ enters only the second term of the equation in \Cref{cor:empirical_Kn_invariant_range}, the two sample sizes satisfy $N_\lambda(\epsilon) \leq N_{\operatorname{Lip}(F)}(\epsilon) = N_\ast(\epsilon/\operatorname{Lip}(F))$.
The improvement is modest, though: for $d > 2p$ the sample size is governed by the first term $\operatorname{Lip}(F)^p C(d,p) N^{-e_{d,p}}$, which involves neither $\lambda$ nor $\alpha$, so the target accuracy $\epsilon$ drives the count while the range of $F$ and the confidence level only move the smaller term (cf.\ \Cref{fig:relative-sample-size-nomogram}).

Both corollaries pass through the curvature measure $\mu_n$.
This is what makes \Cref{cor:empirical_Kn_samples} uniform over all Lipschitz $n$-invariants, and it is also what ties the exponent to the dimension $d=\binom{n}{2}$ of the space of distance matrices.
If only a given real-valued $n$-invariant is of interest, one can instead apply the argument of \Cref{thm:empirical_Kn} to $\nu_F$ directly, viewed as a measure on $\R$.
\begin{prop}
	\label{prop:empirical_real_invariant}
	Fix a real number $1 \leq p < \infty$ and integers $N \geq 1$, $n \geq 2$.
	Let $(X, d_X, \mu_X)$ be a bounded mm-space, and let $F \colon \Kn_n \to \R$ be a real-valued $n$-invariant with $\diam F(\Kn_n(X)) \leq \lambda \diam(X)$ for some $\lambda > 0$.
	Then
	\begin{equation*}
		\bbE[\W_p(\empirical{\nu}{F}{N}, \nu_F)^p] \leq C(1,p) \cdot \bigl(\lambda \diam(X)\bigr)^p \cdot N^{-1/2},
		\qquad
		C(1,p) = \frac{1}{2^{1/2} - 2^{1-p}}.
	\end{equation*}
	Moreover, choose an error margin $\epsilon > 0$ and a confidence level $0 < \alpha < 1$, and let
	\begin{equation*}
		N_F \coloneq \left( \Bigl( C(1,p) + \sqrt{-\tfrac{1}{2} \ln(1-\alpha)} \Bigr) \frac{\lambda^p}{\epsilon^p} \right)^2.
	\end{equation*}
	Then for all $N > N_F$, $\W_p(\empirical{\nu}{F}{N}, \nu_F) \leq \epsilon \diam(X)$ with probability at least $\alpha$.
\end{prop}
\begin{proof}
	Write $D_X \coloneq \diam(X)$ and $D \coloneq \lambda D_X$.
	We may assume $D_X > 0$, since otherwise $\nu_F$ is a Dirac measure and both sides vanish.
	The closure $K$ of $F(\Kn_n(X))$ in $\R$ is a compact set of diameter at most $D$ containing $\supp \nu_F$, and $\empirical{\nu}{F}{N} = F_\#(\empirical{\mu}{n}{N})$ is the empirical measure of the $N$ independent samples $F(\Psi_n(\mathbf{x}^{(i)})) \sim \nu_F$.
	We repeat the proof of \Cref{thm:empirical_Kn} for $\nu_F$ on $\R$ in place of $\mu_n$ on $\R^d$, i.e.\ with $d = 1$, which satisfies $d < 2p$ for every $p \geq 1$, so the exponent is $e_{1,p} = 1/2$.
	All norms on $\R$ agree, and \cite[Thm.~2.1]{Fournier2023} with $m = \infty$ gives
	\begin{equation*}
		\bbE[\W_p(\empirical{\nu}{F}{N}, \nu_F)^p] \leq 2^p \kappa_{1,p}^{(\infty)} \left( \cM_{q}^{(\infty)}(\nu_F) \right)^{p/q} \theta_{1,p,q}^{(\infty)} N^{-1/2}.
	\end{equation*}
	By \Cref{prop:moment_radius}, $\cM_{q}^{(\infty)}(\nu_F)^{1/q} \leq R_\infty(K) \leq D/2$, so $2^p \left( \cM_{q}^{(\infty)}(\nu_F) \right)^{p/q} \leq D^p$.
	Taking the infimum over $q$ as in the proof of \Cref{thm:empirical_Kn} gives the first claim with $C(1,p) = \kappa_{1,p}^{(\infty)}$.
	For the second claim, \Cref{lemma:empirical_sample_size} applies to $\nu_F$ with $D$, $\beta(N) \coloneq C(1,p) D^p N^{-1/2}$, and error margin $\epsilon/\lambda$.
	Its condition on $N_F$ reads $\bigl( C(1,p) + \sqrt{-\frac{1}{2} \ln(1-\alpha)} \bigr) N_F^{-1/2} \leq \epsilon^p/\lambda^p$, which holds with equality by definition of $N_F$.
	This gives $\W_p(\empirical{\nu}{F}{N}, \nu_F) < (\epsilon/\lambda) D = \epsilon D_X$ with probability at least $\alpha$.
\end{proof}
\begin{remark}
Nothing in the proof goes beyond the fact that $\nu_F$ lives on $\R$, and such one-dimensional bounds are classical; see \cite{BobkovLedoux2019} for a comprehensive treatment.
\end{remark}

The main qualitative content of \Cref{thm:empirical_Kn} lies in the exponents $N^{-1/2}$ and $N^{-p/d}$.
The examples following Thm.~1 in \cite{convergence_empirical_fournier} show that the exponents $e_{d,p}=1/2$ and $e_{d,p}=p/d$ of \cite[Thm.~2.1]{Fournier2023} are sharp.
We adapt them to give an example of a curvature set where the growth rate is sharp.
\begin{example}
	\label{ex:tight_example_two_points}
	Let $(X, d_X, \mu_X)$ be an mm-space with two points at distance $D$ and uniform measure $\mu_X$. Then $\Kn_2(X)$, viewed as a subset of $\R$, is the set $\{0, D\}$ and $\mu_2$ is the uniform measure. Since $d/2 = 1/2 < 1$, \Cref{thm:empirical_Kn} yields
	\begin{equation*}
		\bbE[\W_p(\empirical{\mu}{2}{N}, \mu_2)^p] \leq C(1,p) D^p N^{-1/2} = \frac{D^p}{2^{1/2}-2^{1-p}} N^{-1/2}
	\end{equation*}
	with $\displaystyle C(1,p) = \frac{1}{2^{1/2} - 2^{1-p}}$ for any $p \geq 1$. On the other hand, $\empirical{\mu}{2}{N} = \frac{K}{N}\delta_0 + (1-\frac{K}{N})\delta_D$, where $K \sim \operatorname{Bin}(N,1/2)$ is the number of sampled distances equal to $0$, so $\W_p(\empirical{\mu}{2}{N}, \mu_2)^p = D^p \lvert K/N - 1/2 \rvert$ as in Ex.~(a) following Thm.~1 in \cite{convergence_empirical_fournier}.
	By symmetry and the identity $\frac{N-2k}{N}\binom{N}{k} = \binom{N-1}{k} - \binom{N-1}{k-1}$, the resulting sum telescopes, and Stirling's formula yields
	\begin{equation*}
		\bbE[\W_p(\empirical{\mu}{2}{N}, \mu_2)^p]
		= \frac{D^p}{2^N} \sum_{k=0}^{N} \Bigl\lvert \frac{k}{N} - \frac{1}{2} \Bigr\rvert \binom{N}{k}
		= \frac{D^p}{2^N} \binom{N-1}{\lfloor N/2 \rfloor}
		\sim \frac{D^p}{\sqrt{2\pi}} N^{-1/2}.
	\end{equation*}
\end{example}

\begin{remark}
	\label{rmk:tight_example_two_points}
	When $p=1$, the ratio between the constants $C(1,p)$ and $\displaystyle \frac{1}{\sqrt{2\pi}}$ is $\displaystyle \frac{\sqrt{2\pi}}{\sqrt{2}-1} \approx 6.0515$, so the above example shows that the constant $C(d,p)$ is not always optimal.
	Nevertheless, \cite{Fournier2023} put a lot of work into optimizing $C(d,p)$, and what matters for applications is that it is explicit, which is what allows us to compute concrete sample sizes such as those in \Cref{fig:relative-sample-size-nomogram}.
\end{remark}

Unlike in the last example, the exponents in \Cref{thm:empirical_Kn} are unlikely to be sharp for $n \geq 3$, because $\Kn_n(X)$ is often of much lower dimension than $\binom{n}{2}$ (\Cref{rmk:ahlfors_dim_Kn_S1}).
Combining the estimates of \cite{WeedBach2019} with \Cref{prop:ahlfors-regularity}, we can replace $\binom{n}{2}$ by $ns$ for Ahlfors $s$-regular $X$, which is an improvement whenever $s < (n-1)/2$.
This is the content of \Cref{thm:main-sampling-intrinsic}, which is the following result.
\begin{theorem}
	\label{thm:empirical_Kn_intrinsic_WB}
	Fix an integer $n \geq 2$.
	Let $d \coloneq n^2$ or $d \coloneq \binom{n}{2}$.
	Let $(X, d_X, \mu_X)$ be an Ahlfors $s$-regular mm-space.
	Embed $\Kn_n(X)$ in $(\R^d, \| \cdot \|_\infty)$.
	Fix real numbers $1 \leq p \leq \frac{1}{2}d_M(\Kn_n(X))$ and $t > s$.
	Then there exists $0 < \epsilon' \leq 1$, depending on $(X, d_X, \mu_X)$, $n$, $p$, and $t$ but not on $N$, such that for all $N \geq 1$
	\begin{equation*}
		\bbE[\W_p(\empirical{\mu}{n}{N}, \mu_n)^p] \leq \diam(X)^p \left( C_1 \cdot N^{-p/nt} + C_2 \cdot N^{-1/2} \right),
	\end{equation*}
	where $C_1 \coloneq 3^{\frac{3ntp}{nt-2p}+1} \left( \frac{1}{3^{nt/2-p}-1}+3 \right)$ and $C_2 \coloneq (27/\epsilon')^{nt/2}$.
\end{theorem}
\begin{proof}
	Write $D_X \coloneq \diam(X)$.
	We may assume $D_X > 0$, since otherwise both sides vanish.
	As \cite{WeedBach2019} normalize to diameter at most $1$, we first prove the bound for the rescaled mm-space $X' \coloneq (X, d_X/D_X, \mu_X)$, which has diameter $1$ and is again Ahlfors $s$-regular, with constant $C_X \max(D_X^s, D_X^{-s})$.
	Its curvature set $\Kn_n(X') = D_X^{-1} \Kn_n(X)$ is compact by \Cref{lemma:ahlfors-compact} and \Cref{prop:props_Kn}\ref{it:prop_Kn_compact}, has diameter $1$ by \Cref{prop:props_Kn}\ref{it:prop_Kn_diam}, and has the same Minkowski dimension as $\Kn_n(X)$.
	Hence $1 \leq p \leq \frac{1}{2} d_M(\Kn_n(X'))$, and \Cref{prop:ahlfors-regularity}\ref{it:ahlfors_Kn} applied to $X'$ yields $d_p^*(\mu_{n,X'}) \leq ns < nt$ for its curvature measure $\mu_{n,X'}$.
	Following the proof of \cite[Cor.~1]{WeedBach2019}, this implies there exists $0 < \epsilon' \leq 1$, depending only on $\mu_{n,X'}$, $p$, and $t$, that satisfies the hypothesis of \cite[Prop.~5]{WeedBach2019} for $\mu_{n,X'}$ with $nt$ in place of their $s$, so for all $N \geq 1$ we have $\bbE[\W_p(\empirical{\mu}{n}{N}', \mu_{n,X'})^p] \leq C_1 \cdot N^{-p/nt} + C_2 \cdot N^{-1/2}$, where the image $\empirical{\mu}{n}{N}'$ of $\empirical{\mu}{n}{N}$ under scaling by $1/D_X$ is an empirical measure of $\mu_{n,X'}$.
	Scaling by $D_X$ maps $\mu_{n,X'}$ to $\mu_{n,X}$ and $\empirical{\mu}{n}{N}'$ to $\empirical{\mu}{n}{N}$ and multiplies $\W_p$ by $D_X$, which gives the claim.
\end{proof}

\begin{remark}
	\label{rmk:concentration-inequalities}
	Combining \Cref{lemma:empirical_sample_size} with \Cref{thm:empirical_Kn_intrinsic_WB} gives a version of \Cref{cor:empirical_Kn_samples} with exponent $p/nt$ in place of $e_{d,p}$, which is an improvement exactly when $nt < \binom{n}{2}$.
	It is of no practical use, however, because $C_2$ depends on the mm-space through the non-explicit parameter $\epsilon'$, so the required sample size cannot be computed.
\end{remark}

\section{Block decomposition of four-point hyperbolicity distributions}
\label{sec:blockwise-sampling}

In this section we prove \Cref{thm:main-block-decomposition}, which states that the hyperbolicity distribution of a finite mm-space is the sum of those of the blocks of its tight span.
Intuitively, the four-point hyperbolicity deficit only registers how geodesics branch, not how far out they extend after branching.
A finite metric space has no geodesics to speak of, so we first embed $X$ isometrically into its tight span, a geodesic metric space that contains $X$ and in which this branching becomes visible.
For finite $X$ the tight span decomposes into blocks, and every quadruple with a nonzero deficit can be localized to a single block, so that non-tree-likeness is traced to identifiable pieces of $X$.
As a consequence, the hyperbolicity distribution can also be estimated block by block, which reduces the sampling cost when the blocks that are not intervals are small compared with $X$.

We begin by recalling the definition of the tight span, introduced independently by Isbell \cite{isbell-six-theorems} and Dress \cite{dress-tight-span}; we follow the presentation of Lang \cite[Sec.~3]{lang-injective-hulls}.
The idea is to adjoin to $X$ every point that could lie on a geodesic between points of $X$.
Each such point $p$ is described by its distances $f(x) = d(p,x)$ to the points of $X$, and the triangle inequality requires $f(x)\geq d_X(x,x') - f(x')$ for all $x, x' \in X$, with equality exactly when $p$ lies on a geodesic from $x$ to $x'$.
The tight span is the set of all such functions $f$ that are minimal, meaning that no value $f(x)$ can be lowered.

\begin{defn}[Tight span {\cite{isbell-six-theorems, dress-tight-span}}]
	Let $(X,d_X)$ be a metric space. The \emph{tight span} $T(X,d_X)$ is the set of functions $f \colon X \to \R$ such that
	\begin{equation*}
		f(x) = \sup_{x' \in X} \big[d_X(x,x')-f(x')\big] \quad \text{for all } x \in X,
	\end{equation*}
	equipped with the $L^\infty$ metric $d_\infty(f,g) \coloneq \sup_{x \in X} |f(x) - g(x)|$.
\end{defn}

A metric space $E$ is \emph{injective} if every $1$-Lipschitz map into $E$ defined on a subset $A$ of a metric space $B$ extends to a $1$-Lipschitz map $B \to E$ \cite[Sec.~1]{isbell-six-theorems}.
The tight span is also known as the \emph{injective hull} of $X$, since it is the smallest injective metric space containing $X$ \cite[2.10]{isbell-six-theorems}.

We will need the following properties.
\begin{prop}
	\label{prop:tight_span}
	Let $(X, d_X)$ be a metric space. Then:
	\begin{enumerate}
		\item The map $h_\bullet \colon X \to T(X,d_X)$ given by $x \mapsto \big( h_x \colon x' \mapsto d_X(x,x') \big)$ is an isometric embedding $(X, d_X) \hookrightarrow T(X, d_X)$.
		\item $T(X, d_X)$ is geodesic.
		\item If $X$ is finite, $T(X ,d_X)$ is a finite polytopal complex, i.e.\ a polyhedral complex whose cells are convex compact polytopes.
		\item If $A \subset X$, there is an isometric embedding $\iota \colon T(A, d_X|_A) \hookrightarrow T(X, d_X)$ with $\iota(f)|_A = f$ for all $f \in T(A, d_X|_A)$. It is compatible with the embeddings of the first item, that is, $\iota(h_a) = h_a$ for all $a \in A$.
	\end{enumerate}
\end{prop}
\begin{proof}
	\textit{(i)} See the proof of \cite[Thm.~2.1]{isbell-six-theorems} (routine calculation using the triangle inequality).
	\textit{(ii)} $T(X,d_X)$ is injective \cite[2.9]{isbell-six-theorems}, and injective metric spaces are geodesic \cite[Sec.~2]{lang-injective-hulls}.
	\textit{(iii)} See the discussion following \cite[Prop.~3.5]{lang-injective-hulls}.
	\textit{(iv)} The embedding is \cite[Prop.~3.5(1)]{lang-injective-hulls}. Since $d_\infty(f, h_a) = f(a)$ for all $f \in T(X,d_X)$ \cite[Eq.~(3.2)]{lang-injective-hulls}, we get $d_\infty(\iota(h_a), h_a) = h_a(a) = 0$.
\end{proof}

\begin{example}
	\label{ex:tight-span-4pts}
	Perhaps the simplest example of an injective space is a metric tree with its geodesic distance. In fact, by \cite[Thm.~8]{dress-tight-span}, $T(X, d_X)$ is the smallest $\R$-tree that contains $X$ when $\hyp(X)=0$. Another simple example is when $|X|=4$: $T(X, d_X)$ is isometric to a rectangle $[0, L_1] \times [0, L_2]$ equipped with the $\ell^1$ metric and with 4 edges glued to its vertices \cite[Ex.~(1.16)]{dress-tight-span}; see \Cref{fig:tight_span_4_pts}. Since $X$ embeds isometrically in $T(X, d_X)$, distances in $X$ can be read off from this picture, e.g.\ $d_X(x_1, x_2) = a_1 + L_1 + a_2$.
\end{example}
\begin{figure}
	\centering
	\begin{tikzpicture}[scale=1.25]
	\tikzmath{
		\pt=0.025;
		\a = 1;
		\b = 2;
		\d1 = 1;
		\d2 = 0.75;
		\d3 = 0.7;
		\d4 = 0.9;
		\t1 = 45;
		\t2 = 45;
		\t3 = 45;
		\t4 = 45;
		\pax1 =  0;	\pay1 =  0;
		\pax2 =-\b;	\pay2 =  0;
		\pax3 =-\b;	\pay3 =-\a;
		\pax4 =  0;	\pay4 =-\a;
		\pbx1 =  \pax1+cos(\t1)*\d1;		\pby1 =  \pay1+sin(\t1)*\d1;
		\pbx2 =  \pax2-cos(\t2)*\d2;		\pby2 =  \pay2+sin(\t2)*\d2;
		\pbx3 =  \pax3-cos(\t3)*\d3;		\pby3 =  \pay3-sin(\t3)*\d3;
		\pbx4 =  \pax4+cos(\t4)*\d4;		\pby4 =  \pay4-sin(\t4)*\d4;
	}

	\fill[pattern={Lines[angle=45, distance=3pt, line width=0.3pt]}, pattern color=black!50] (\pax3,\pay3) rectangle (\pax1,\pay1);

	\draw[fill] (\pax1,\pay1) circle [radius=\pt] node[below left=2pt, fill=white, inner sep=1pt]{$c_1$};
	\draw[fill] (\pax2,\pay2) circle [radius=\pt] node[below right=2pt, fill=white, inner sep=1pt]{$c_2$};
	\draw[fill] (\pax3,\pay3) circle [radius=\pt] node[above right=2pt, fill=white, inner sep=1pt]{$c_3$};
	\draw[fill] (\pax4,\pay4) circle [radius=\pt] node[above left=2pt, fill=white, inner sep=1pt]{$c_4$};
	
	\draw[fill] (\pbx1,\pby1) circle [radius=\pt] node[above right]{$x_1$};
	\draw[fill] (\pbx2,\pby2) circle [radius=\pt] node[above left ]{$x_2$};
	\draw[fill] (\pbx3,\pby3) circle [radius=\pt] node[below left ]{$x_3$};
	\draw[fill] (\pbx4,\pby4) circle [radius=\pt] node[below right]{$x_4$};
	
	\draw (\pax1,\pay1) -- (\pax2,\pay2) node[midway, above]{$L_1$};
	\draw (\pax2,\pay2) -- (\pax3,\pay3);
	\draw (\pax3,\pay3) -- (\pax4,\pay4);
	
	\draw (\pax4,\pay4) -- (\pax1,\pay1) node[midway, right]{$L_2$};

	
	

	\draw (\pax1,\pay1) -- (\pbx1,\pby1) node[midway, below=6pt, right=2pt]{$a_1$};
	\draw (\pax2,\pay2) -- (\pbx2,\pby2) node[midway, below=5pt, left=1pt]{$a_2$};
	\draw (\pax3,\pay3) -- (\pbx3,\pby3) node[midway, above left]{$a_3$};
	\draw (\pax4,\pay4) -- (\pbx4,\pby4) node[midway, above=3pt, right=1pt]{$a_4$};
\end{tikzpicture}
	\caption{The tight span of any four-point metric space $X=\{x_1, x_2, x_3, x_4\}$ is isometric to a rectangle $[0, L_1] \times [0, L_2]$ with the $\ell^1$ metric and 4 edges of lengths $a_1, a_2, a_3, a_4$ glued to its vertices.}
	\label{fig:tight_span_4_pts}
\end{figure}
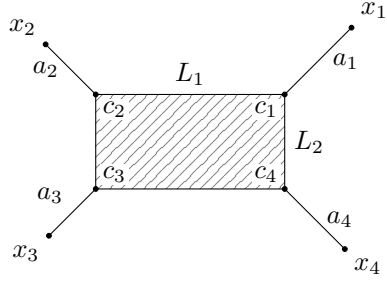

\begin{remark}
	\label{rmk:polytopal_structure}
	$T(X, d_X)$ can be a polytopal complex even if $X$ is not finite, but we restrict to the finite case for simplicity.
\end{remark}

We exploit the polytopal structure of the tight span via the following decomposition.
\begin{defn}[Blocks and cut-vertices, cf.~\cite{bc-tree}]
	\label{def:blocks_cuts}
	Let $T$ be a connected polytopal complex. A vertex $v \in T$ is a \emph{cut-vertex} if $T \setminus \{v\}$ is disconnected. We say that $v$ \emph{separates} two subsets $A, B \subset T$ if $A \setminus \{v\}$ and $B \setminus \{v\}$ are contained in different connected components of $T \setminus \{v\}$. A maximal subcomplex $B \subset T$ that does not have cut-vertices is called a \emph{block}. We denote the set of cut-vertices of $T$ as $\cut(T)$ and the set of blocks as $\blocks(T)$.
\end{defn}

\begin{defn}[Block-cut tree, cf.~\cite{bc-tree}]
	The \emph{block-cut tree} of a connected polytopal complex $T$ is the graph $\bctree(T)$ with nodes $\blocks(T) \cup \cut(T)$ and edge set
	\begin{equation*}
		E \coloneq \{ \{B,c\} \mid B \in \blocks(T), c \in \cut(T) \text{ and } c \in B \}.
	\end{equation*}
\end{defn}

It is straightforward to show that $\bctree(T)$ is a tree for any connected polytopal complex by generalizing \cite[Thm.~1]{bc-tree}.
This tree is the combinatorial device that keeps track of how the blocks of the tight span are attached to one another.
Given a finite metric space $(X, d_X)$, we denote the set of blocks and cut-vertices of $T(X, d_X)$ by $\blocks(X)$ and $\cut(X)$, respectively.
Similarly, we write $\bctree(X)$ to denote the block-cut tree of $T(X, d_X)$.

\begin{example}
	\label{ex:block_cut_tree}
	In \Cref{fig:tight_span_4_pts}, let $B$ be the central rectangle $[0, L_1] \times [0, L_2]$ and $E_i$ the edge attached to $B$ at $c_i$. Then $\cut(X) = \{c_1, c_2, c_3, c_4\}$ and $\blocks(X) = \{B, E_1, E_2, E_3, E_4\}$, while $\bctree(X)$ is the graph on $\blocks(X) \cup \cut(X)$ with edges $\{B, c_i\}$ and $\{c_i, E_i\}$ for $1 \leq i \leq 4$.
\end{example}

With this language in place, we can project points and tuples to a chosen block by collapsing everything outside the block onto the unique separating cut-vertex.

\begin{defn}
	Let $(X, d_X)$ be a finite metric space and $B \in \blocks(X)$. We define $P_B \colon X \to B$ by
	\begin{equation*}
		P_B(x) \coloneq
		\begin{cases}
			x & \text{ if } x \in B\\
			c & \text{ if } x \notin B
		\end{cases}
	\end{equation*}
	where $c \in \cut(X) \cap B$ is the unique cut-vertex that separates $x$ from $B$. Given $\mathbf{x} = (x_1, \dots, x_n) \in X^n$, we extend the definition of $P_B$ to $X^n$ by setting $P_B(\mathbf{x}) \coloneq (P_B(x_1), \dots, P_B(x_n))$.
\end{defn}
That is, $P_B$ fixes the points of the tuple already contained in $B$ and collapses everything else onto the point where it attaches to $B$; equivalently, $P_B(x)$ is the point of $B$ nearest to $x$, since $P_B(x)$ separates $x$ from $B$ and every geodesic from $x$ to a point of $B$ therefore passes through $P_B(x)$.
For example, in \Cref{fig:tight_span_4_pts}, $P_B((x_1, x_2, c_3, c_4)) = (c_1, c_2, c_3, c_4)$.

The next lemmas show that at most one block retains a full four-point configuration under projection, that such a block exists whenever the hyperbolicity deficit is nonzero, and that projecting onto it preserves the deficit.
We first observe that the lengths $a_i$ of the edges attached to the rectangle in \Cref{fig:tight_span_4_pts} do not affect $\hyp(\mathbf{x})$.
\begin{lemma}
	\label{lemma:tight_span_4_pts}
	Let $\mathbf{x} = (x_1, x_2, x_3, x_4)$ and $\mathbf{c} = (c_1, c_2, c_3, c_4)$ as in \Cref{fig:tight_span_4_pts}. Then $\hyp(\mathbf{x}) = \hyp(\mathbf{c}) = \min(L_1, L_2)$.
\end{lemma}
\begin{proof}
	Denote the metric on $T(X, d_X)$ as $d_T$. Note that $d_X(x_i, x_j) = a_i + d_T(c_i, c_j) + a_j$, so for any permutation $\sigma \in \operatorname{Sym}(4)$,
	\begin{equation*}
		d_X(x_{\sigma(1)}, x_{\sigma(2)}) + d_X(x_{\sigma(3)}, x_{\sigma(4)})
		= d_T(c_{\sigma(1)}, c_{\sigma(2)}) + d_T(c_{\sigma(3)}, c_{\sigma(4)}) + (a_1 + a_2 + a_3 + a_4).
	\end{equation*}
	Thus, $S(\mathbf{x})=S(\mathbf{c})+(a_1+a_2+a_3+a_4)(1,1,1)$, and
	\begin{align*}
		\hyp(\mathbf{x})
		= \frac{1}{2}\left(S_{(1)}(\mathbf{x})-S_{(2)}(\mathbf{x})\right)
		= \frac{1}{2}\left(S_{(1)}(\mathbf{c})-S_{(2)}(\mathbf{c})\right)
		= \hyp(\mathbf{c}).
	\end{align*}
	Lastly, $s_{12,34}(\mathbf{c}) = 2L_1$, $s_{14,23}(\mathbf{c}) = 2L_2$, and $s_{13,24}(\mathbf{c}) = 2(L_1+L_2)$, so $S_{(2)}(\mathbf{c}) = 2\max(L_1, L_2)$ and $S_{(1)}(\mathbf{c}) = 2(L_1+L_2)$. Hence,
	$\hyp(\mathbf{c}) = \frac{1}{2}\left[S_{(1)}(\mathbf{c})-S_{(2)}(\mathbf{c})\right] = \min(L_1, L_2)$.
\end{proof}

The same splitting holds for the projection onto an arbitrary block of $T(X, d_X)$, which gives the following generalization.

\begin{lemma}
	\label{lemma:hyp_projection}
	Let $(X, d_X)$ be a finite metric space, $B \in \blocks(X)$, and $\mathbf{x} \in X^4$. If $|P_B(\mathbf{x})| = 4$, then $\hyp(\mathbf{x}) = \hyp(P_B(\mathbf{x}))$. Otherwise, $\hyp(P_B(\mathbf{x})) = 0$.
\end{lemma}
\begin{proof}
	Denote the metric on $T(X, d_X)$ by $d_T$; it restricts to $d_X$ on $X$ by \Cref{prop:tight_span}\,\textit{(i)}.
	Suppose first that $|P_B(\mathbf{x})| < 4$. Then two entries of $P_B(\mathbf{x})$ coincide, so $\hyp(P_B(\mathbf{x})) = 0$ (see \Cref{inv:hyperbolicity}).

	Now suppose $|P_B(\mathbf{x})| = 4$, and write $c_i \coloneq P_B(x_i)$ and $a_i \coloneq d_T(x_i, c_i)$. We show that
	\begin{equation*}
		d_X(x_i, x_j) = a_i + d_T(c_i, c_j) + a_j \quad \text{for all } i \neq j.
	\end{equation*}
	By \Cref{prop:tight_span}\,\textit{(ii)}, any two points of $T(X, d_X)$ are joined by a geodesic, and such a geodesic passes through every point that separates its endpoints.
	If $x_i \notin B$, then $c_i$ separates $x_i$ from $B$. It also separates $x_i$ from $x_j$: otherwise $x_j$ would lie in the same connected component of $T(X, d_X) \setminus \{c_i\}$ as $x_i$, so $x_j \notin B$ and $c_i$ would separate $x_j$ from $B$, and uniqueness in the definition of $P_B$ would give $c_j = c_i$.
	Hence $d_X(x_i, x_j) = a_i + d_T(c_i, x_j)$, which holds trivially if $x_i \in B$, since then $c_i = x_i$.
	Similarly, if $x_j \notin B$, then $c_j$ separates $x_j$ from $c_i \in B \setminus \{c_j\}$, so $d_T(c_i, x_j) = d_T(c_i, c_j) + a_j$, which again holds trivially if $x_j \in B$.
	Consequently, each of the three pair sums of $\mathbf{x}$ exceeds the corresponding pair sum of $P_B(\mathbf{x})$ by $a_1 + a_2 + a_3 + a_4$, and $\hyp(\mathbf{x}) = \hyp(P_B(\mathbf{x}))$ follows as in the proof of \Cref{lemma:tight_span_4_pts}.
\end{proof}

Thus projecting onto a block preserves the hyperbolicity deficit whenever four distinct points remain.
The next lemma shows that this happens for at most one block.

\begin{lemma}
	\label{lemma:block_projection}
	Let $(X, d_X)$ be a finite metric space and fix $n \geq 3$. For any $\mathbf{x} \in X^n$, there exists at most one $B \in \blocks(X)$ such that $|P_B(\mathbf{x})| = n$.
\end{lemma}
\begin{proof}
	Suppose there exists $B_0 \in \blocks(X)$ such that $|P_{B_0}(\mathbf{x})| = n$. Choose any $B \in \blocks(X) \setminus \{B_0\}$, and let $c_0 \in \cut(X) \cap B_0$ and $c \in \cut(X) \cap B$ be the cut-vertices that separate $B_0$ and $B$. Note that the unique path in $\bctree(X)$ from $B_0$ to $B$ contains $c$. Each point $x_i$ corresponds to a node of $\bctree(X)$, namely to $x_i$ itself if it is a cut-vertex and to the unique block containing it otherwise. If $P_{B_0}(x_i) \neq c_0$, then the path in $\bctree(X)$ from the node corresponding to $x_i$ to $B_0$ does not contain $c_0$. However, concatenating it with the path from $B_0$ to $B$ yields a path from the node corresponding to $x_i$ to $B$ that contains $c$. Hence, $P_B(x_i) = c$. Since $|P_{B_0}(\mathbf{x})| = n$, there are at least $n-1 \geq 2$ points $x_i$ such that $P_{B_0}(x_i) \neq c_0$, so by the previous argument, there are at least two $x_i$ such that $P_B(x_i) = c$. Hence, $|P_B(\mathbf{x})| < n$.
\end{proof}

The next lemma shows that such a block exists whenever the hyperbolicity deficit is nonzero.

\begin{lemma}
	\label{lemma:hyp_not_0}
	Let $(X, d_X)$ be a finite metric space and $\mathbf{x} \in X^4$. If $\hyp(\mathbf{x}) \neq 0$, then there exists a block $B \in \blocks(X)$ such that $|P_B(\mathbf{x})| = 4$.
\end{lemma}
\begin{proof}
	Since $\hyp(\mathbf{x}) \neq 0$, the points $x_1, \dots, x_4$ are pairwise distinct (see \Cref{inv:hyperbolicity}), and their tight span $T(\mathbf{x})$ is as in \Cref{fig:tight_span_4_pts}, with central rectangle $B_0 = [0, L_1] \times [0, L_2]$ and edges $E_i$ from $x_i$ to the corners $c_i$ of $B_0$.
	By \Cref{lemma:tight_span_4_pts}, $\min(L_1, L_2) = \hyp(\mathbf{x}) > 0$, so the corners $c_i$ are pairwise distinct and the edges $E_i$ are pairwise disjoint.
	By \Cref{prop:tight_span}\,\textit{(iv)}, applied to $A = \{x_1, x_2, x_3, x_4\}$, there is an isometric embedding $T(\mathbf{x}) \hookrightarrow T(X, d_X)$ that is compatible with the embeddings of $\mathbf{x}$ and $X$, and we identify $T(\mathbf{x})$ with its image.
	Removing a single point from the rectangle $B_0$ leaves it connected, so no cut-vertex of $T(X, d_X)$ separates two points of $B_0$. Hence $B_0$ is contained in a single block $B \in \blocks(X)$.

	We claim that $P_B(x_i) \in E_i$ for every $i$. If $x_i \in B$, then $P_B(x_i) = x_i \in E_i$. Otherwise, $P_B(x_i)$ separates $x_i$ from $B$. If $P_B(x_i)$ were not contained in $E_i$, then $E_i$ would be a path in $T(X, d_X) \setminus \{P_B(x_i)\}$ from $x_i$ to $c_i \in B$, a contradiction.
	Since the edges $E_i$ are pairwise disjoint, $|P_B(\mathbf{x})| = 4$.
\end{proof}

It follows that every $\mathbf{x} \in X^4$ with $\hyp(\mathbf{x}) \neq 0$ has a unique block $B_{\mathbf{x}} \in \blocks(X)$ with $|P_{B_{\mathbf{x}}}(\mathbf{x})| = 4$.
We equip each block $B \in \blocks(X)$ with the metric inherited from $T(X, d_X)$ and with the measure $\mu_B \coloneq (P_B)_\# \mu_X$, which moves all mass outside $B$ onto the cut-vertex where it is attached, and write $\nu_{\hyp,B} \coloneq \hyp_\#(\mu_B^{\otimes 4})$ for the distribution of hyperbolicity deficits of the resulting mm-space.
With this, we can state and prove~\Cref{thm:main-block-decomposition}.
\begin{theorem}
	\label{thm:block_decomposition_hyp}
	Let $(X, d_X, \mu_X)$ be a finite mm-space. Then, for every Borel set $D \subset \R$ with $0 \notin D$,
	\begin{equation*}
		\nu_{\hyp,X}(D) = \sum_{B \in \blocks(X)} \nu_{\hyp,B}(D).
	\end{equation*}
	Moreover, $\nu_{\hyp,X}(\{0\}) = 1 - \sum_{B \in \blocks(X)} \nu_{\hyp,B}(\R \setminus \{0\})$.
\end{theorem}
\begin{proof}
	Let $D \subset \R$ be a Borel set with $0 \notin D$. Then
	\begin{align*}
		\nu_{\hyp,X}(D)
		&= \mu_X^{\otimes 4}(\{ \mathbf{x} \in X^4 \mid \hyp(\mathbf{x}) \in D \}) \\
		&= \sum_{B \in \blocks(X)} \mu_X^{\otimes 4}(\{ \mathbf{x} \in X^4 \mid \hyp(\mathbf{x}) \in D \text{ and } B_{\mathbf{x}} = B \}) \\
		&= \sum_{B \in \blocks(X)} \mu_X^{\otimes 4}(\{ \mathbf{x} \in X^4 \mid \hyp(P_B(\mathbf{x})) \in D \}) \\
		&= \sum_{B \in \blocks(X)} \nu_{\hyp,B}(D),
	\end{align*}
	where the second equality groups the quadruples $\mathbf{x}$ by their unique blocks $B_{\mathbf{x}}$, provided by \Cref{lemma:block_projection,lemma:hyp_not_0} since $0 \notin D$, and the third follows from \Cref{lemma:hyp_projection} for the same reason.
	The formula for $\nu_{\hyp,X}(\{0\})$ then follows by normalization.
\end{proof}
\begin{corollary}
	\label{cor:block_decomposition_hyp}
	Let $(X, d_X, \mu_X)$ be a finite mm-space with $\blocks(X) = \{B_0\} \sqcup \mathcal{B}_{\mathcal{I}}$, where every block in $\mathcal{B}_{\mathcal{I}}$ is isometric to an interval.
	Then $\nu_{\hyp,X} = \nu_{\hyp,B_0}$.
\end{corollary}
\begin{proof}
	A block $B$ that is isometric to an interval is an $\R$-tree, so every quadruple in $B$ has hyperbolicity deficit $0$ (see \Cref{inv:hyperbolicity}) and $\nu_{\hyp,B}=\delta_0$.
	In particular, the blocks in $\mathcal{B}_{\mathcal{I}}$ contribute nothing to the sum in \Cref{thm:block_decomposition_hyp}, and $\nu_{\hyp,X}$ and $\nu_{\hyp,B_0}$ therefore agree on every Borel set not containing $0$, and hence also on $\{0\}$, so the two measures coincide.
\end{proof}
\begin{remark}
If $\mu_X$ has full support, Dress's criterion \cite[Thm.~8]{dress-tight-span} amounts to the statement that $\nu_{\hyp,X}=\delta_0$ if and only if all blocks of $T(X,d_X)$ are intervals, each of which satisfies $\nu_{\hyp,B}=\delta_0$.
The block decomposition in \Cref{thm:block_decomposition_hyp} captures how $\nu_{\hyp,X}$ arises from the hyperbolicity distributions of the blocks more generally.
\end{remark}

As a consequence of the block decomposition, we also find that $\nu_{\hyp,X}$ can be estimated block by block: blocks isometric to intervals have $\nu_{\hyp,B} = \delta_0$ and can be skipped, and for every other block one estimates $\nu_{\hyp,B}$ by its empirical distribution as in \Cref{cor:empirical_Kn_invariant_range} (with $\lambda = 1$, see \Cref{inv:hyperbolicity}), adds these away from $0$, and obtains the atom at $0$ by normalization.\footnote{Due to sampling error, the estimated masses of the blocks away from $0$ may add up to slightly more than $1$, which can produce a negative probability at $0$.}
The $\W_1$ errors of the blocks add up, so to reach a total $\W_1$ error $\epsilon\diam(X)$, we may assign to each sampled block $B$ the share $\epsilon\diam(X) \cdot \diam(B)/\sum_{B'} \diam(B')$ of this error, where the sum runs over the blocks that are not intervals.
Measured in units of the diameter of the respective block, this gives every block the same error $\epsilon' \coloneq \epsilon\diam(X)/\sum_{B'} \diam(B')$, so with $p = 1$ each block requires $N_{\lambda=1}(\epsilon')$ samples instead of the $N_{\lambda=1}(\epsilon)$ samples needed for $X$.
Since $N_{\lambda=1}$ grows like $\epsilon^{-6}$, this pays off when few blocks are not intervals and their diameters add up to much less than $\diam(X)$.
In the situation of \Cref{cor:block_decomposition_hyp}, the number of samples drops by roughly a factor $(\diam(X)/\diam(B_0))^6$.
The cut-vertices of $T(X,d_X)$ that are attached to blocks with nonzero hyperbolicity distribution can be computed from the distance matrix in $O(|X|^3)$ time \cite{DressEtAl2010}.

\section{Applications in Hamming geometry and molecular evolution}
\label{sec:hamming}

Molecular evolution studies how the genomes of organisms change over time.
Abstractly, genome sequences are words of finite length over the four-character nucleotide alphabet $\{A,C,G,T\}$.
After alignment, distances between them are measured by counting the number of sites at which two sequences differ, commonly called the \emph{Hamming distance} \cite{Hamming1950}.
We refer to the metric space of all possible sequences of a fixed length together with this metric as the \emph{Hamming cube}.
This space provides a natural ambient geometry for the evolution of genomes.
The genomes of a population determines a probability measure on the Hamming cube.
The shape of this distribution then reflects information about mutation frequency, selection, and interactions between sites.
Accordingly, the distributions of $n$-invariants obtained from the observed population can be used to study the evolutionary dynamics of a population.

In \Cref{sec:hamming-sitewise}, we show how the distributions of $n$-invariants on the Hamming cube are determined by the site-wise nucleotide frequencies when sites are independent.
We then apply this to the uniform measure (\Cref{subsec:hamming-uniform}) and to single-nucleotide substitution models (\Cref{sec:simulated-evolution}), which serve as references for the SARS\nobreakdash-CoV\nobreakdash-2 genome data in \Cref{subsec:sars-cov2}.

\subsection{The site-wise structure of Hamming geometry}
\label{sec:hamming-sitewise}

Let $\Sigma$ be a finite alphabet of cardinality $k\ge 2$, equipped with the discrete metric ${d_\Sigma(\sigma,\tau)\coloneq\mathbf{1}_{\sigma\neq \tau}}$.
\begin{defn}
	For $L\ge 1$, the \emph{Hamming cube of dimension $L$ over $\Sigma$} is the $L$-fold product space $H(\Sigma,L)\coloneq\Sigma^L$ equipped with the $\ell^1$-product metric, called the \emph{Hamming metric}, given for $x,x'\in H(\Sigma,L)$ by
	\begin{align}
		d_H(x,x') \coloneq \sum_{m=1}^L d_\Sigma(x_m,x_m') = \#\{m\in\{1,\ldots,L\}\mid x_m\neq x_m'\}.
	\end{align}
\end{defn}

We often write $H$ for $H(\Sigma,L)$ when alphabet and length are clear from context.
Elements ${x=(x_m)_{m=1}^{L} \in H}$ are called \emph{words} or \emph{sequences}, and we denote the projection to the $m$-th character or \emph{site} of a word by $\pi_m \colon H\to\Sigma$, $\pi_m(x)\coloneq x_m$.
For an $n$-tuple of words $\mathbf{x}=(x_1,\dots,x_n)\in H^n$, we apply this projection componentwise and write $\pi_m(\mathbf{x})\coloneq(\pi_m(x_1),\dots,\pi_m(x_n))\in\Sigma^n$.
Characters of $\Sigma$ are also called \emph{alleles}, and we say a site $m$ is \emph{monoallelic}, \emph{biallelic}, or \emph{triallelic} if the number of distinct alleles occurring in $\pi_m(\mathbf{x})$ is exactly $1$, $2$, or $3$, respectively.

Consider a product probability measure $\mu$ on $H$, and write $\rho_m\coloneq(\pi_m)_\#\mu$ for its single-site marginals; the product assumption means that $\mu$ is the product of these marginals, i.e.\ $\mu=\bigotimes_{m=1}^L\rho_m$.
Under the canonical identification $H^n=(\Sigma^L)^n\cong(\Sigma^n)^L$, we have $\mu^{\otimes n}= \bigotimes_{m=1}^L\rho_m^{\otimes n}$.
In other words, drawing $n$ independent words from $\mu$ is equivalent to drawing, independently at each site $m$, an $n$-tuple of characters with distribution $\rho_m^{\otimes n}$.
For a probability measure $\rho$ on $\Sigma$ and $j\geq 2$, define its $j$-th \emph{collision moment} by $c_j(\rho)\coloneq\sum_{\sigma\in\Sigma}\rho(\sigma)^j$.
Equivalently, $c_j(\rho)$ is the probability that $j$ independent draws from $\rho$ all agree.

A particularly well-behaved class of $n$-invariants consists of those that are given by a sum of single-site contributions.
\begin{defn}
	Assume $Z$ is a commutative monoid with measurable addition.
	We say a $Z$-valued metric $n$-invariant $F$ is \emph{site-wise additive on the Hamming cube $H$} if for every $\mathbf{x}\in H^n$ one has
	\begin{align}
		F(\mathbf{x})=\sum_{m=1}^L F(\pi_m(\mathbf{x})).
	\end{align}
\end{defn}

For a site-wise additive $n$-invariant $F$, the pushforward measure $\nu_{F,\mu}$ is fully determined by the single-site measures $\nu_{F,\rho_m}$.
To make this precise, for a commutative monoid $Z$ with measurable addition and $r\geq 1$ let
$\summ_r \colon Z^r \to Z$ be $r$-fold summation, i.e.\ the map given by $\summ_r(z_1,\dots,z_r)\coloneq \sum_{m=1}^r z_m$.
For probability measures $\nu,\lambda$ on $Z$, denote their additive convolution by
$\nu\ast\lambda \coloneq (\summ_2)_\#(\nu\otimes\lambda)$, or equivalently $(\nu\ast\lambda)(\{z\})=\sum_{z_1+z_2=z}\nu(\{z_1\})\lambda(\{z_2\})$.
We write $\nu^{\ast r}\coloneq (\summ_r)_\#(\nu^{\otimes r})$ for the $r$-th power under additive convolution.

\begin{lemma}
\label{lem:hamming-sitewise-pgf}
Let $Z$ be a commutative monoid with measurable addition and let $F$ be a $Z$-valued metric $n$-invariant that is site-wise additive on $H$.
Then
\begin{align}
	\nu_{F,\mu}=\nu_{F,\rho_1}\ast\cdots\ast\nu_{F,\rho_L}.
\end{align}
In particular, if $\rho_m=\rho$ for every $m$, then $\nu_{F,\mu}=(\nu_{F,\rho})^{\ast L}$.
\end{lemma}
\begin{proof}
	Under $H^n\cong(\Sigma^n)^L$, the single-site contributions $F\circ\pi_m$ are independent and have distributions $\nu_{F,\rho_m}$.
	Since $F$ evaluates on $H^n$ as the sum of single-site contributions, its distribution is the stated convolution.
\end{proof}

The distance map is site-wise additive by definition, so \Cref{lem:hamming-sitewise-pgf} gives the following result.

\begin{prop}
	\label{prop:product-distance-distribution}
	Let $\mu=\bigotimes_{m=1}^L\rho_m$ be a product measure on $H$, and set $p_m\coloneq 1-c_2(\rho_m)$.
	Then the distance distribution is Poisson--binomial with probability vector $(p_1,\ldots,p_L)$:
	\begin{align}
		\nu_{\dist,\mu}
		=\operatorname{PBin}(p_1, \ldots, p_L).
	\end{align}
	In particular, if $\rho_m=\rho$ for every $m$, then $p_m=p$ and $\nu_{\dist,\mu}=\operatorname{Pois}(p)$.
\end{prop}
\begin{proof}
	The probability that two characters disagree is the complement of the second collision moment, i.e.\ given by $p_m=1-c_2(\rho_m)$, such that $\nu_{\dist,\rho_m}=\operatorname{Bern}(p_m)$.
	By \Cref{lem:hamming-sitewise-pgf}, we have $\nu_{\dist,\mu}=\nu_{\dist,\rho_1}\ast\cdots\ast\nu_{\dist,\rho_L} = \operatorname{Bern}(p_1) \ast \cdots \ast \operatorname{Bern}(p_L)$, which is exactly the Poisson--binomial distribution $\operatorname{PBin}(p_1, \ldots, p_L)$.
\end{proof}

For the four-point hyperbolicity deficit, recall from \Cref{inv:hyperbolicity} that $\hyp=\gap_+\circ S$, where $\gap_+(a,b,c)\coloneq \tfrac12\bigl(\max\{a,b,c\}-\med\{a,b,c\}\bigr)$ extracts the upper gap from triple of pairwise-distance sums $S$:
\begin{align}
	S(\mathbf{x}) \coloneq \bigl(s_{13,24}(\mathbf{x}),s_{12,34}(\mathbf{x}),s_{14,23}(\mathbf{x})\bigr),
	\qquad \mathbf{x}\in H^4.
\end{align}
According to \Cref{prop:invariant-postcomposition}, the hyperbolicity distribution is given by $\nu_{\hyp,X}=(\gap_+)_\#\nu_{S,X}$.
Each component $s_{ij,k\ell}$ of $S$ is site-wise additive and hence so is $S$.
Using \Cref{lem:hamming-sitewise-pgf}, the distribution of $S$ on $H^4$ is thus the convolution of its single-site distributions.
\Cref{fig:hamming-single-site-pairsums} shows the nine values $S$ can take on a single site.
The five displayed patterns admit explicit probabilities in terms of the collision moments, leading to the following result.

\begin{figure}[t]
	\centering
	\begin{tikzpicture}
	\newcommand{\site}[5]{
		\coordinate (x1) at (0,0);
		\coordinate (x2) at (1.15,0);
		\coordinate (x3) at (1.15,0.82);
		\coordinate (x4) at (0,0.82);
		\draw[thick] (x1) -- (x3);
		\draw[thick] (x2) -- (x4);
		\draw[dashed,thick] (x1) -- (x2);
		\draw[dashed,thick] (x4) -- (x3);
		\draw[dotted,thick] (x1) -- (x4);
		\draw[dotted,thick] (x2) -- (x3);
		\path (x1) node[below left=-2pt] {#1};
		\path (x2) node[below right=-2pt] {#2};
		\path (x3) node[above right=-2pt] {#3};
		\path (x4) node[above left=-2pt] {#4};
		\foreach \p in {x1,x2,x3,x4}
			\fill (\p) circle (1.25pt);
		\node at (0.575,-0.52) {#5};
	}


	\begin{scope}[shift={(0,0.65)}]
		\site{$\sigma$}{$\sigma$}{$\sigma$}{$\sigma$}{\scriptsize $S=(0,0,0)$}
	\end{scope}

	\begin{scope}[shift={(2.2,0.65)}]
		\site{$\sigma$}{$\sigma$}{$\sigma$}{$\tau$}{\scriptsize $S=(1,1,1)$}
	\end{scope}

	\begin{scope}[shift={(4.4,2.55)}]
		\site{$\sigma$}{$\sigma$}{$\tau$}{$\tau$}{\scriptsize $S=(2,0,2)$}
	\end{scope}

	\begin{scope}[shift={(4.4,0.65)}]
		\site{$\sigma$}{$\tau$}{$\sigma$}{$\tau$}{\scriptsize $S=(0,2,2)$}
	\end{scope}

	\begin{scope}[shift={(4.4,-1.25)}]
		\site{$\sigma$}{$\tau$}{$\tau$}{$\sigma$}{\scriptsize $S=(2,2,0)$}
	\end{scope}

	\begin{scope}[shift={(6.6,2.55)}]
		\site{$\sigma$}{$\sigma$}{$\tau$}{$\nu$}{\scriptsize $S=(2,1,2)$}
	\end{scope}

	\begin{scope}[shift={(6.6,0.65)}]
		\site{$\sigma$}{$\tau$}{$\sigma$}{$\nu$}{\scriptsize $S=(1,2,2)$}
	\end{scope}

	\begin{scope}[shift={(6.6,-1.25)}]
		\site{$\sigma$}{$\tau$}{$\nu$}{$\sigma$}{\scriptsize $S=(2,2,1)$}
	\end{scope}

	\begin{scope}[shift={(8.8,0.65)}]
		\site{$\sigma$}{$\tau$}{$\nu$}{$\xi$}{\scriptsize $S=(2,2,2)$}
	\end{scope}

	\begin{scope}[overlay, shift={(1.65,1.35)}, scale=0.75, transform shape]
		\draw[gray!65, rounded corners=7pt, fill=gray!4]
			(11.05,1.25) rectangle (15.0,3.8);

		\coordinate (y1) at (11.55,1.75);
		\coordinate (y2) at (14.5,1.75);
		\coordinate (y3) at (14.5,3.3);
		\coordinate (y4) at (11.55,3.3);
		\draw[thick] (y1) -- (y3);
		\draw[thick] (y2) -- (y4);
		\draw[dashed,thick] (y1) -- (y2);
		\draw[dashed,thick] (y4) -- (y3);
		\draw[dotted,thick] (y1) -- (y4);
		\draw[dotted,thick] (y2) -- (y3);
		\foreach \p in {y1,y2,y3,y4}
			\fill (\p) circle (1.5pt);
		\node[below left=-2pt] at (y1) {$\sigma_1$};
		\node[below right=-2pt] at (y2) {$\sigma_2$};
		\node[above right=-2pt] at (y3) {$\sigma_3$};
		\node[above left=-2pt] at (y4) {$\sigma_4$};
		\node at (13.025,3.53) {\scriptsize $s_{12,34}$};
		\node[rotate=90] at (11.32,2.525) {\scriptsize $s_{14,23}$};
		\node[rotate=-28] at (12.6,2.88) {\scriptsize $s_{13,24}$};
	\end{scope}
\end{tikzpicture}
	\caption{Single-site values of the pair-sum map $S$ on $(\Sigma, d_\Sigma)$. The components of $S$ correspond in order to the sums of the diagonal, horizontal, and vertical edge lengths, as indicated by the upper-right schematics.}
	\label{fig:hamming-single-site-pairsums}
\end{figure}
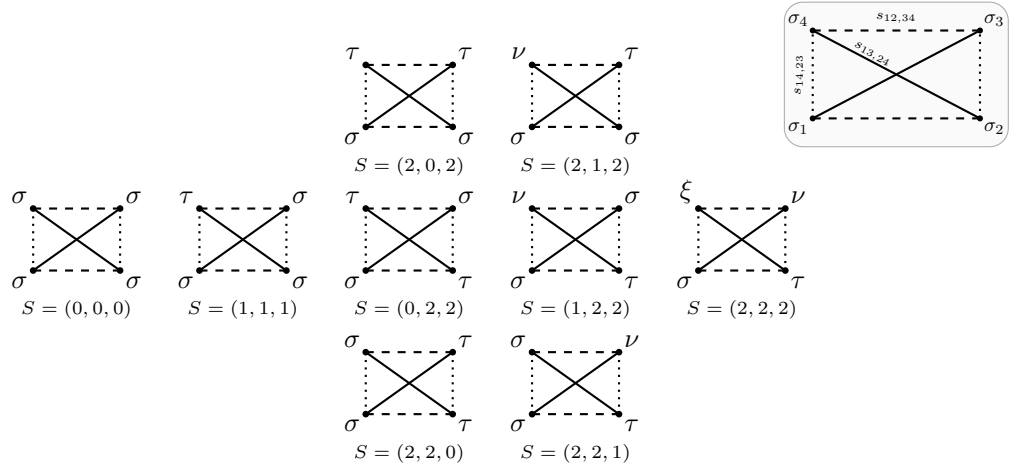

\begin{prop}
	\label{prop:single-site-pairsum-distribution}
	Let $\rho$ be a probability measure on $\Sigma$ and write $c_j=c_j(\rho)$.
	The induced single-site distribution of $S$ is
	\begin{align}
		\nu_{S,\rho}
		=
		\begin{multlined}[t]
		c_4\,\delta_{(0,0,0)}
		+
		4(c_3-c_4)\,\delta_{(1,1,1)}
		+
		(c_2^2-c_4)\bigl(
			\delta_{(0,2,2)}+\delta_{(2,0,2)}+\delta_{(2,2,0)}
		\bigr)
		\\
		+
		2(c_2-2c_3+2c_4-c_2^2)\bigl(
			\delta_{(1,2,2)}+\delta_{(2,1,2)}+\delta_{(2,2,1)}
		\bigr)
		+
		(1-6c_2+3c_2^2+8c_3-6c_4)\,\delta_{(2,2,2)}.
		\end{multlined}
	\end{align}
\end{prop}
\begin{proof}
	The value $S=(0,0,0)$ occurs when all four characters agree, which happens with probability $\sum_\sigma \rho(\sigma)^4=c_4$.
	$S=(1,1,1)$ occurs when one character appears three times and another appears once, which has probability
	\begin{align}
		4\sum_\sigma \rho(\sigma)^3\bigl(1-\rho(\sigma)\bigr)=4(c_3-c_4).
	\end{align}
	The three arrangements with two pairs of equal characters correspond to $S=(0,2,2)$, $S=(2,0,2)$, and $S=(2,2,0)$, each with probability
	\begin{align}
		\sum_{\sigma\neq\tau}\rho(\sigma)^2\rho(\tau)^2=c_2^2-c_4.
	\end{align}
	The three triallelic arrangements have $S=(1,2,2)$, $S=(2,1,2)$, and $S=(2,2,1)$.
	For each of these, there are two possible positions for the repeated pair, and hence the probability is
	\begin{align}
		2\sum_\sigma \rho(\sigma)^2
		\sum_{\substack{\tau,\nu\neq\sigma\\\tau\neq\nu}}\rho(\tau)\rho(\nu)
		&=
		2\sum_\sigma \rho(\sigma)^2\bigl(1-2\rho(\sigma)+2\rho(\sigma)^2-c_2\bigr)
		=2(c_2-2c_3+2c_4-c_2^2).
	\end{align}
	The remaining probability, corresponding to four distinct characters and hence to $S=(2,2,2)$, is obtained by subtracting the probabilities of the preceding cases from $1$, giving $1-6c_2+3c_2^2+8c_3-6c_4$.
\end{proof}

\begin{prop}
	\label{prop:product-hyp-distribution}
	Let $\mu=\bigotimes_{m=1}^L\rho_m$ be a product measure on $H$.
	Then the distribution of hyperbolicity deficits on $H^4$ is
	\begin{align}
		\nu_{\hyp,\mu}
		=
		(\gap_+)_\#\bigl(\nu_{S,\rho_1}\ast\cdots\ast\nu_{S,\rho_L}\bigr).
	\end{align}
\end{prop}
\begin{proof}
	By \Cref{lem:hamming-sitewise-pgf}, the distribution of $S$ on $H^4$ is the convolution of its single-site distributions.
	Since $\hyp=\gap_+\circ S$, the result follows from \Cref{prop:invariant-postcomposition}.
\end{proof}
While this is not an explicit analytic expression for $\nu_{\hyp,\mu}$, it provides a method to compute $\nu_{\hyp,\mu}$ from the single-site distributions $\nu_{S,\rho_m}$.
In practice one can for example use a fast Fourier transform for the convolution step, which makes exact computation feasible for moderately large $L$.

The site-wise decomposition also shows which sites can produce a positive deficit.
The single-site patterns $(0,0,0)$, $(1,1,1)$, and $(2,2,2)$ add the same amount to all three components of $S$ and therefore do not affect $\hyp$.
Each of the remaining patterns, two pairs of equal characters or a triallelic site, favours one of the three ways of grouping the four sequences into two pairs, and $\hyp(\mathbf{x})>0$ requires sites favouring at least two different groupings.
For biallelic sites, this is exactly incompatibility in the sense of the four-gamete test \cite{HudsonKaplan1985}.
In particular, if the four sequences are related by a tree on which every site mutates at most once, every such site favours the grouping given by the tree, and $\hyp(\mathbf{x})=0$.
A positive deficit thus requires homoplasy or recombination.

The pushforward along $\gap_+$ also reveals that triallelic quadruples are the only single-site patterns for which the three components of $S$ do not have the same parity, and hence the only possible source of half-integer hyperbolicity deficits.
This motivates the following definition.
\begin{defn}
	For a probability measure $\rho$ on $\Sigma$, define its \emph{triallelic accessibility} to be the probability of drawing exactly three distinct characters in four independent trials,
	\begin{align}
	\theta(\rho)
	\coloneq
	\rho^{\otimes 4}\bigl(\{(\tau_1,\ldots,\tau_4)\in\Sigma^4\setcolon|\{\tau_1,\ldots,\tau_4\}|=3\}\bigr).
	\end{align}
\end{defn}
By \Cref{prop:single-site-pairsum-distribution}, the common coefficient of the three triallelic atoms $(1,2,2)$, $(2,1,2)$, and $(2,2,1)$ in the single-site distribution of $S$ is $\theta(\rho)/3$.

\begin{prop}
	\label{prop:half-integer-suppression}
	Let $\mu=\bigotimes_{m=1}^L\rho_m$. Then
	\begin{align}
		\nu_{\hyp,\mu}\bigl(\tfrac12+\Z_{\geq 0}\bigr)
		\leq
		\sum_{m=1}^L \theta(\rho_m).
	\end{align}
	In particular, if every $\rho_m$ is supported on at most two characters, then $\nu_{\hyp,\mu}\bigl(\tfrac12+\Z_{\geq 0}\bigr)=0$.
\end{prop}
\begin{proof}
	Let $A\subset H^4$ be the set of quadruples that are triallelic at some site.
	By \Cref{prop:single-site-pairsum-distribution}, if $\mathbf{x}\notin A$, then every summand $S(\pi_m(\mathbf{x}))$ in the site-wise decomposition $S(\mathbf{x})=\sum_{m=1}^L S(\pi_m(\mathbf{x}))$ is a vector whose components have the same parity.
	Consequently, the components of $S(\mathbf{x})$ have the same parity and $\hyp(\mathbf{x})\in\Z_{\geq 0}$ for every $\mathbf{x}\in H^4\setminus A$, so $\hyp^{-1}(\tfrac12+\Z_{\geq 0})\subseteq A$.
	Thus, $\nu_{\hyp,\mu}(\tfrac12+\Z_{\geq 0})=\mu^{\otimes 4}(\hyp^{-1}(\tfrac12+\Z_{\geq 0}))\leq\mu^{\otimes 4}(A)$.
	Independence across sites then gives $\mu^{\otimes 4}(A)=1-\prod_{m=1}^L(1-\theta(\rho_m))\leq\sum_{m=1}^L\theta(\rho_m)$.
\end{proof}
Note that, as a consequence of \Cref{prop:half-integer-suppression}, $\nu_{\hyp,\rho}$ and $\nu_{\hyp,\mu}$ are supported on $\ell_k \Z_{\geq 0}$, with lattice spacing $\ell_k=1$ when $k=2$ and $\ell_k=\tfrac{1}{2}$ when $k\geq 3$.
For the uniform measure, all half-integer bins are accordingly empty for $k=2$ (\Cref{fig:uniform-hamming-hyperbolicity}, top row).

The $(4,1)$-persistence measure $\Uvr{4,1}$ (\Cref{inv:persistence-set}) can formally be treated in the same way as the hyperbolicity deficit by convolving the single-site distributions of the vector of all six pairwise distances and then applying the persistence map.
Unfortunately, persistence does not factor through an intermediate, lower-dimensional $n$-invariant, so this does not admit a more tractable expression.
For this reason, we obtain the $(4,1)$-persistence measures primarily through sampling.
It is, however, possible to describe the support of $\Uvr{4,1}$, i.e.\ to determine the $(4,1)$-persistence set $\Dvr{4,1}$ for the Hamming cube.

\begin{prop}
	\label{prop:D41_Hypercube}
	The $(4,1)$-persistence set of the Hamming cube is given by
	\begin{align}
		\Dvr{4,1}(H(\Sigma, L)) = \{ (t_b, t_d) \in \Z^2 \setcolon 0 < t_b < t_d \leq L \text{ and } t_d \leq 2t_b \}.
	\end{align}
\end{prop}
\begin{proof}~\\
	$(\subseteq)$\quad Since the metric $d_H$ is integer valued and has diameter $L$, $(t_b, t_d) \in \Z^2$ and $t_d \leq L$ for all $(t_b, t_d) \in \Dvr{4,1}(H(\Sigma, L))$. To show the bound $t_d \leq 2t_b$, let $X = \{x_1, x_2, x_3, x_4\} \subset H(\Sigma, L)$, and let $b$, $c$, $a_i$ (with $i=1,2,3,4$) be the isolation indices of $X$ (see \Cref{fig:tight_span_4_pts}). By \cite[Prop.~6.4, Item~3]{GomezMemoli2024}, $t_d(X) - t_b(X) \leq \min(b,c)$. Note that every $d_H(x_i, x_j)$ equals $a_i+a_j$ plus either $b$, $c$, or both, so $\min(b,c) \leq d_H(x_i, x_j)$ for all $1\leq i, j \leq 4$. In particular, $t_b(X)$ equals $d_H(x_i, x_j)$ for some $i$ and $j$, so $t_d(X) - t_b(X) \leq \min(b,c) \leq t_b(X)$. Thus, $t_d(X) \leq 2 t_b(X)$.

	$(\supseteq)$\quad Let $(t_b, t_d) \in \Z^2$ with $0 < t_b < t_d \leq L$ and $t_d \leq 2t_b$. Let $\sigma, \tau, \xi \in \Sigma$, and let $x_1 = \sigma^{t_d} \cdot \xi^{L-t_d}$, $x_2 = \tau^{t_b} \sigma^{t_d-t_b} \cdot \xi^{L-t_d}$, $x_3 = \tau^{t_d} \cdot \xi^{L-t_d}$, and $x_4 = \sigma^{t_b} \tau^{t_d-t_b} \cdot \xi^{L-t_d}$. Note that all $x_i$ share the last term $\xi^{L-t_d}$, so the distances $d_H(x_i, x_j)$ are determined by the $\sigma$ and $\tau$ components. In fact, $d_H(x_1, x_2) = d_H(x_3, x_4) = t_b$, $d_H(x_2, x_3) = d_H(x_4, x_1) = t_d - t_b$, and $d_H(x_1, x_3) = d_H(x_2, x_4) = t_d$. Since $t_b < t_d \leq 2t_b$, we have $t_d - t_b \leq t_b < t_d$, so $t_b(X) = d_H(x_1, x_2) = t_b$ and $t_d(X) = d_H(x_1, x_3) = t_d$.
\end{proof}

\subsection{The uniform measure}
\label{subsec:hamming-uniform}

We now describe the geometric properties of the Hamming cube itself, as seen by the distance distribution, the distribution of hyperbolicity deficits, and the $(4,1)$-persistence measure.
For this, equip $\Sigma$ with the uniform probability measure $\rho \coloneq \frac{1}{k}\sum_{a \in \Sigma} \delta_a$ and consider the uniform measure $\mu \coloneq \rho^{\otimes L}$ on $H(\Sigma, L)$.

The distance distribution on $H(\Sigma, L)$ is $\operatorname{Bin}(L, \frac{k-1}{k})$, as follows from \Cref{prop:product-distance-distribution}.

The distribution of hyperbolicity deficits can be computed from the single-site distribution of the pairwise sums $S$ per \Cref{prop:product-hyp-distribution}.
For the latter, observe that the collision moments of $\rho$ are given by $c_j(\rho)=k^{1-j}$ for $j=2,3,4$.
Substitution into \Cref{prop:single-site-pairsum-distribution} gives the explicit single-site distribution:
\begin{align}
	\nu_{S,\rho}
	=
	\begin{multlined}[t]
	\frac{1}{k^3}\,\delta_{(0,0,0)}
	+
	\frac{4(k-1)}{k^3}\,\delta_{(1,1,1)}
	+
	\frac{k-1}{k^3}\bigl(
		\delta_{(0,2,2)}+\delta_{(2,0,2)}+\delta_{(2,2,0)}
	\bigr)
	\\
	+
	\frac{2(k-1)(k-2)}{k^3}\bigl(
		\delta_{(1,2,2)}+\delta_{(2,1,2)}+\delta_{(2,2,1)}
	\bigr)
	+
	\frac{(k-1)(k-2)(k-3)}{k^3}\,\delta_{(2,2,2)}.
	\end{multlined}
\end{align}
From this, the distribution of hyperbolicity deficits $\nu_{\hyp,\mu}$ can be computed exactly for moderate $L$ by convolving $L$ copies of $\nu_{S,\rho}$ and pushing forward along $\gap_+$ (\Cref{prop:product-hyp-distribution}).
For the uniform measure, the upcoming proposition gives an asymptotic expression for the distribution of hyperbolicity deficits as $L\to\infty$.
See \Cref{fig:uniform-hamming-hyperbolicity} for a comparison of the exact and asymptotic distributions for various alphabet sizes and word lengths.

\begin{figure}
	\centering
	\includegraphics[width=\textwidth]{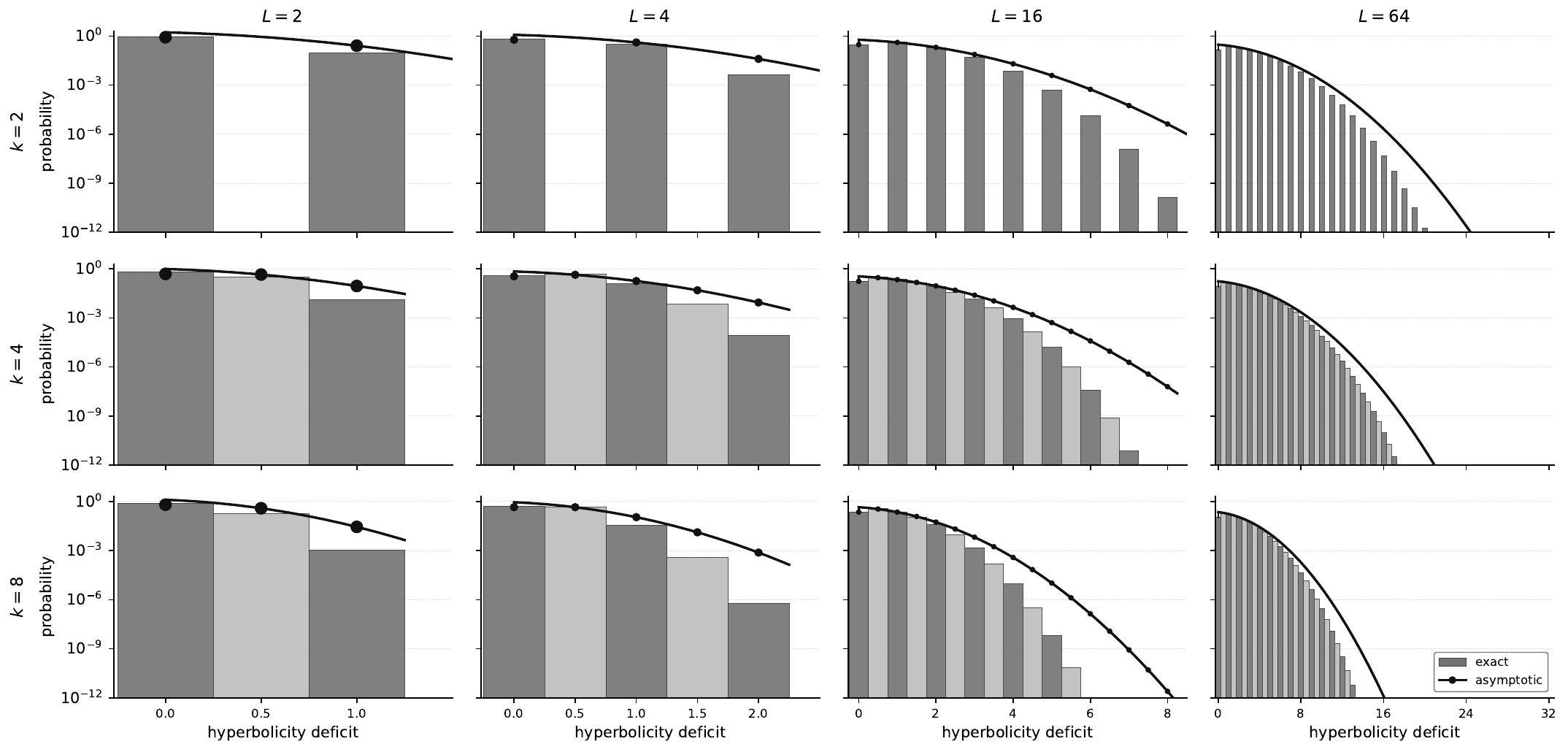}
	\caption{Exact and asymptotic distributions of hyperbolicity deficits $\nu_{\hyp,\mu}$ for the uniform measure $\mu$ on the Hamming cubes $H(\Sigma, L)$ for alphabet sizes $k=2,4,8$ (rows), and word lengths $L=2,4,16,64$ (columns).
	Bars give the exact probabilities computed by FFT-based convolution with alternating shades for integer and half-integer values.
	For $k=2$ all half-integer bins are empty.
	The black curve is the Gaussian approximation in \Cref{prop:hamming-hyp-clt} with nodes marking the value on the corresponding support $\ell_k\Z_{\geq 0}$.
	}
	\label{fig:uniform-hamming-hyperbolicity}
\end{figure}

\begin{prop}
	\label{prop:hamming-hyp-clt}
	Let $\mathcal{N}_{\mu,\sigma^2}$ denote the probability density of the normal distribution with mean $\mu$ and variance $\sigma^2$, write $\Phi_{\mu,\sigma^2}(r) \coloneq \int_{-\infty}^r \mathcal{N}_{\mu,\sigma^2}(t)\,dt$ for its cumulative distribution function, and let $\Theta$ be the Heaviside function with $\Theta(0)\coloneq\tfrac12$.
	Let $\sigma_k^2 \coloneq \frac{2(k-1)}{k^2}$ be the variance of the site-wise $\nu_{S,\rho}$.
	For $r\in\ell_k\Z$,
	\begin{align}
		\nu_{\hyp,\mu}(\{r\})
		&\sim
		6 \ell_k \Theta(r)\ \mathcal{N}_{0, \frac12 \sigma_k^2 L}(r)\ \Phi_{0,\frac32\sigma_k^2 L}(-r),
		\quad
		(L\to\infty).
	\end{align}
\end{prop}
\begin{proof}
	The preceding convolution and pushforward identities determine $\nu_{\hyp,\mu}$ from the distribution of $S$ on $\Sigma^4$.
	The mean of $\nu_{S,\rho}$ is $\mu_k=\frac{2(k-1)}{k}(1,1,1)$ and the covariance matrix has diagonal entries $\sigma_k^2=\frac{2(k-1)}{k^2}$ and is zero otherwise.
	Standardize the image of $S \colon H^4 \to \Z^3$ by centring it at $L\mu_k$ and rescaling by $\sqrt{L}\sigma_k$.
	Denote the corresponding affine map $A_{k,L} \colon \R^3 \to \R^3$ by $A_{k,L}(z)\coloneq\frac{z-L\mu_k}{\sqrt{L}\sigma_k}$.
	The centred and rescaled distribution of $S$ on $H$ is then $\widetilde{\nu}_{S,\mu}\coloneq(A_{k,L})_\# \nu_{S,\mu}$.

	Since $\nu_{S,\mu}=\nu_{S,\rho}^{\ast L}=(\summ_L)_\#(\nu_{S,\rho}^{\otimes L})$ by \Cref{lem:hamming-sitewise-pgf}, pushing forward along $A_{k,L}$ gives exactly the distribution of a centred and rescaled sum of $L$ i.i.d.\ random vectors.
	Consequently, the multivariate central limit theorem gives $\widetilde{\nu}_{S,\mu}\xrightarrow{w}\mathcal{N}_{\mathbf{0},\mathrm{I}_3}$ as $L\to\infty$ \cite[Thm.~29.5]{Billingsley1995}, where $\mathcal{N}_{\mathbf{0},\mathrm{I}_3}$ denotes the standard normal distribution on $\R^3$.
	Since $\mu_k$ lies on the diagonal and $\gap_+$ is invariant under diagonal translations and homogeneous under positive scalar multiplication, $\gap_+(A_{k,L}(z))=\frac{\gap_+(z)}{\sigma_k\sqrt{L}}$ for every $z\in\R^3$.
	The continuous mapping theorem~\cite[Thm.~29.2]{Billingsley1995} then gives
	\begin{align}
		(\gap_+)_\#\widetilde{\nu}_{S,\mu}\xrightarrow{w}(\gap_+)_\#\mathcal{N}_{\mathbf{0},\mathrm{I}_3}.
	\end{align}
	Write $q_{(1)}\geq q_{(2)}\geq q_{(3)}$ for the descending order statistics of the components $q_1,q_2,q_3$ under $\mathcal{N}_{\mathbf{0},\mathrm{I}_3}$, such that $2\gap_+(q)=q_{(1)}-q_{(2)}$.
	The joint density of two order statistics \cite[Sec.~2.2, Eq.~(2.2.2)]{DavidNagaraja2003} specializes here to
	\begin{align}
		f_{q_{(1)},q_{(2)}}(u,v)
		&=
		6\Phi_{0,1}(v)\,\mathcal{N}_{0,1}(v)\,\mathcal{N}_{0,1}(u)\,\mathbf{1}_{\{v\leq u\}}.
	\end{align}
	The three factors have the following interpretation: $\mathcal{N}_{0,1}(u)$ and $\mathcal{N}_{0,1}(v)$ are the Gaussian densities of the components that realize the maximum and median, while $\Phi_{0,1}(v)$ is the probability that the minimum lies below $v$, enforcing the ordering constraint $v= q_{(2)}\geq q_{(3)}$.
	Integrating $f_{q_{(1)},q_{(2)}}$ over $v$ with $u=v+t$ gives the probability density of the spacing $t=q_{(1)}-q_{(2)}=2\gap_+(q)$.
	Completing the square, we have $\mathcal{N}_{0,1}(v)\mathcal{N}_{0,1}(v+t)=\mathcal{N}_{0,2}(t)\mathcal{N}_{-\frac{t}{2},\frac{1}{2}}(v)$.
	If $Y$ is a random variable with distribution $\mathcal{N}_{-\frac{t}{2},\frac{1}{2}}$, then $\bbE[\Phi_{0,1}(Y)]=\Phi_{0,\frac{3}{2}}(-\tfrac{t}{2})$.
	Thus the probability density of $q_{(1)}-q_{(2)}$ is given by
	\begin{align}
		f_{q_{(1)}-q_{(2)}}(t)
		&=
		6\int_{-\infty}^{\infty}
			\Phi_{0,1}(v)\,\mathcal{N}_{0,1}(v)\,\mathcal{N}_{0,1}(v+t)
		\,dv
		=
		6\mathcal{N}_{0,2}(t)\Phi_{0,\frac{3}{2}}\!\left(-\frac{t}{2}\right),
		\qquad t\geq 0.
	\end{align}
	The scaled variable $\gap_+(S)/(\sigma_k\sqrt{L})$ takes values in the lattice $\frac{\ell_k}{\sigma_k\sqrt{L}}\Z_{\geq 0}$.
	The limiting variable is $\gap_+(q)$ with $q\sim\mathcal{N}_{\mathbf{0},\mathrm{I}_3}$, and since $q_{(1)}-q_{(2)}=2\gap_+(q)$ its density is given by $2 f_{q_{(1)}-q_{(2)}}(2w)$ for $w\geq 0$.
	The leading term in the multivariate lattice local expansion \cite[Thm.~22.1]{BhattacharyaRangaRao2010}, applied to the sum distribution of $S$ and summed over the lattice fibres of $\gap_+$, gives the corresponding local approximation for the pushforward: its mass at a lattice point is the limiting density at the corresponding scaled value times the scaled lattice spacing.
	Thus, for $r\in\ell_k\Z_{\geq 0}$,
	\begin{align}
		\nu_{\hyp,\mu}(\{r\})
		&\sim
		\Theta(r)\,
		\frac{\ell_k}{\sigma_k\sqrt{L}}\,
		2 f_{q_{(1)}-q_{(2)}}\!\left(\frac{2r}{\sigma_k\sqrt{L}}\right),
		\quad
		(L\to\infty).
	\end{align}
	The Heaviside factor $\Theta(0)=\frac12$ arises from the half cell at the boundary point $r=0$.
	Using the standard scaling identities $a^{-1}\mathcal{N}_{0,\sigma^2}(x/a)=\mathcal{N}_{0,a^2\sigma^2}(x)$ and $\Phi_{0,\sigma^2}(x/a)=\Phi_{0,a^2\sigma^2}(x)$, gives
	\begin{align}
		\frac{2\ell_k}{\sigma_k\sqrt{L}}\,
		f_{q_{(1)}-q_{(2)}}\!\left(\frac{2r}{\sigma_k\sqrt{L}}\right)
		&=
		6\ell_k\,
		\mathcal{N}_{0,\sigma_k^2\frac{L}{2}}(r)\,
		\Phi_{0,\frac{3}{2}\sigma_k^2 L}(-r),
	\end{align}
	which yields the stated formula.
\end{proof}

\subsection{Single-nucleotide substitution models}
\label{sec:simulated-evolution}

\begin{figure}
    \centering
    \includegraphics[width=\textwidth]{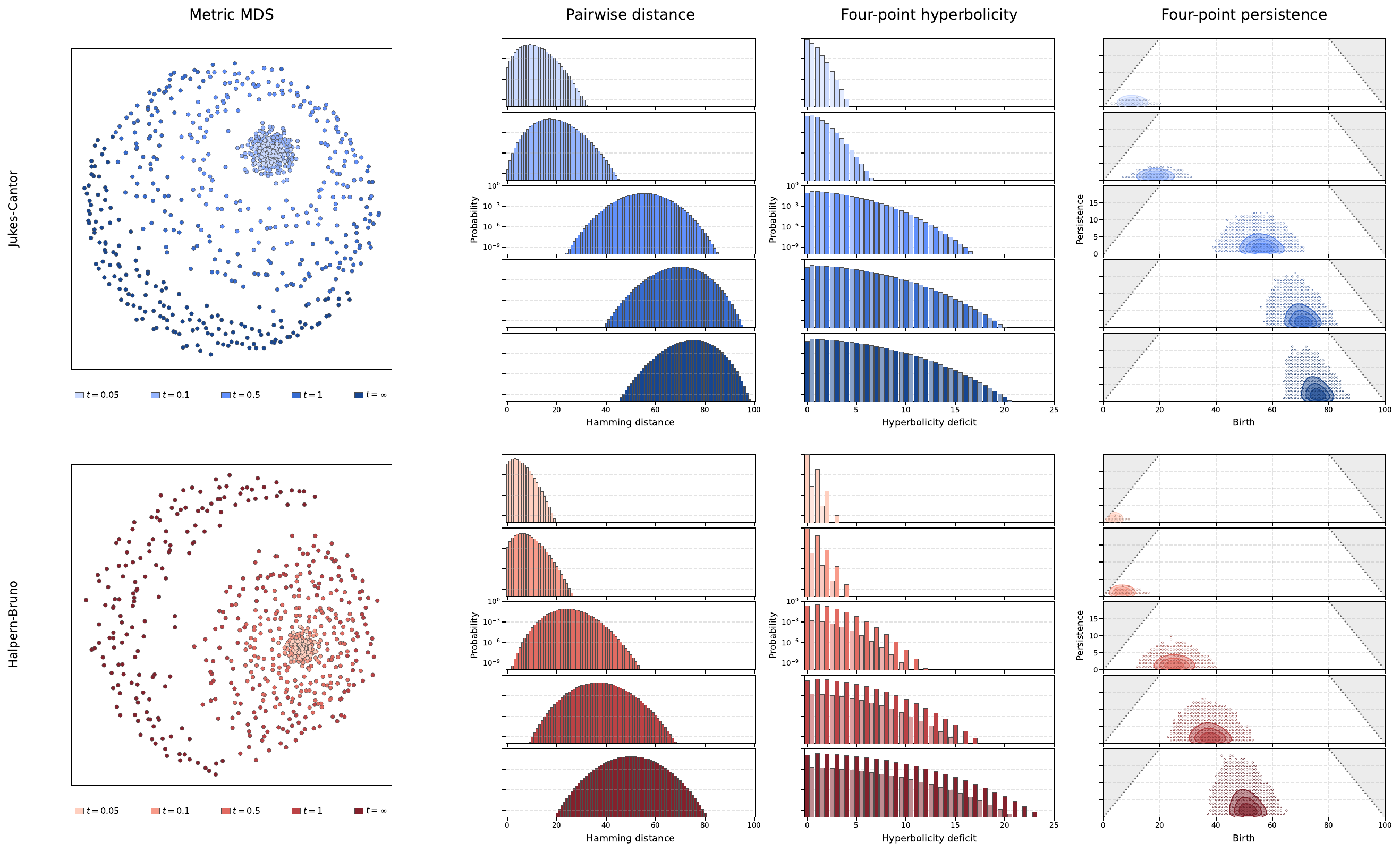}
    \caption{
        Metric $n$-point invariants of single-nucleotide substitution models with alphabet size $k=4$ and sequence length $L=100$.
        The top row shows the JC model and the bottom row shows a homogeneous HB model with JC baseline matrix and fitness profile $\phi=(0,0,-10,-10)$.
        Columns show a sampled metric-MDS visualization, the pairwise Hamming-distance distribution, the four-point hyperbolicity deficit distribution, and the $(4,1)$-persistence measure in a birth--persistence plot at times $t=0.05,0.1,0.5,1.0,\infty$.
		The distance distribution and the distribution of hyperbolicity deficits are exact, while the persistence measures are estimated from $N\approx 10^8$ independent quadruples, as determined by \Cref{cor:empirical_Kn_samples} for error $\epsilon=0.1$ and per-distribution confidence $\alpha=0.995$ (providing $0.95$ simultaneous confidence across the ten distributions).
		In the $(4,1)$-persistence diagrams, points mark observed nonzero birth--persistence classes and shaded regions correspond to Gaussian-smoothed highest-density regions containing $95\%$, $80\%$, and $50\%$ of the nonzero $(4,1)$-persistence mass.
        The grey region is excluded by the support bounds in Hamming geometry reported in \Cref{prop:D41_Hypercube}.
    }
    \label{fig:substitution-model-invariants}
\end{figure}

We now consider probability measures on $H$ derived from models for evolutionary dynamics.
Let $\mu_t$ denote the probability distribution of genomes for some organism at time $t$.
A \emph{substitution model} describes the evolution of $\mu_t$ as a continuous-time Markov chain on $H$: individual genomes undergo random \emph{nucleotide substitutions}, i.e.\ transitions $x \to x'$ based on substituting nucleotides to obtain $x'$ from $x$.
The distribution $\mu_t$ thus evolves according to the Kolmogorov forward equation
\begin{align}
	\frac{d}{dt}\mu_t = \mu_t Q,
\end{align}
with generator $Q$.
The off-diagonal entry $Q_{xx'} \geq 0$ is the rate at which a genome of type $x$ is replaced by one of type $x'$, while the diagonal entry $Q_{xx} = -\sum_{x' \neq x} Q_{xx'}$ is chosen such that each row sums to zero, which ensures that $\mu_t$ remains a probability measure for all $t \geq 0$.
In general, $Q$ may couple substitution rates at different sites, encoding epistatic interactions, where the rate of substitution at site $m$ depends on the presence of a certain allele at site $m'$.

We restrict ourselves to \emph{single-nucleotide substitution models}, where each substitution changes exactly one nucleotide, independently from all other sites.
In this case, the generator $Q$ is given by
\begin{align}
	Q = \sum_{m=1}^{L} I^{\otimes(m-1)} \otimes Q^{(m)} \otimes I^{\otimes(L-m)},
\end{align}
where $I$ denotes the $k\times k$ identity matrix and each $Q^{(m)}$ is a $k\times k$ single-site generator encoding nucleotide substitutions at site $m$.
In this case one can show that, if the initial distribution is a product measure $\mu_0 = \bigotimes_{m=1}^L \rho_{m,0}$, then the distribution remains a product measure $\mu_t = \bigotimes_{m=1}^L \rho_{m,t}$, where $\rho_{m,t} \coloneq \rho_{m,0} \exp\left(tQ^{(m)}\right)$.

One of the simplest examples for such models is the \emph{Jukes--Cantor (JC) model}, where all nucleotide substitutions occur at the same rate, i.e.\ the single-site generator is given by
\begin{align}
	Q^{(m)}_{\sigma\tau} = \frac{1}{k-1}, \quad \sigma\neq\tau,
\end{align}
and diagonal entries are determined by the condition that rows sum to zero, so that each state has total exit rate $1$ and, correspondingly, $t$ corresponds to the number of expected substitutions per site.

Another important example is the \emph{Halpern--Bruno (HB) model}, in which a site-independent baseline mutation process is modulated by site-specific fitness effects to obtain the substitution rates \cite{HalpernBruno1998}.
Let $B$ be a matrix of transition rates on $\Sigma$, and let $\phi^{(m)}\colon\Sigma\to\R$ be a function that associates a fitness value to each allele at site $m$.
The HB single-site generator at site $m$ has off-diagonal entries
\begin{align}
	Q^{(m)}_{\sigma\tau}
	&=
	B_{\sigma\tau}\,
	\frac{\phi^{(m)}(\tau)-\phi^{(m)}(\sigma)}{1-\exp\left(-(\phi^{(m)}(\tau)-\phi^{(m)}(\sigma))\right)},
	\qquad
	\sigma\neq\tau,
\end{align}
and diagonal entries are determined by the condition that rows sum to zero.
The matrix $B$ encodes the baseline substitution rates independent of evolutionary selection, reflecting for instance biochemical frequencies.
The fitness factor then makes substitutions towards higher-fitness alleles faster and substitutions towards lower-fitness alleles slower.

The JC and HB models differ in their long-time behaviour.
As $t\to\infty$ the JC model approaches the uniform distribution of \Cref{subsec:hamming-uniform}.
The homogeneous HB model with symmetric JC baseline matrix $B$ and finite fitness values instead approaches a stationary Boltzmann distribution $\rho_\phi(\sigma)\propto\exp(\phi(\sigma))$ with distance distribution $\operatorname{Bin}(L,1-c_2(\rho_\phi))$, recovering the uniform baseline only when $\phi$ is constant.

The distance distribution, the distribution of hyperbolicity deficits, and the $(4,1)$-persistence measure of the JC and HB models are determined as discussed in \Cref{prop:product-distance-distribution}, \Cref{prop:product-hyp-distribution}, and around \Cref{prop:D41_Hypercube}, respectively.
In particular, since the single-site generators of both models are irreducible, $\mu_t=\bigotimes_{m=1}^L\rho_{m,t}$ is fully supported for every $t>0$, so the sampling results of \Cref{sec:convergence-rates} apply.
See \Cref{fig:substitution-model-invariants} for a comparison of the JC model and the homogeneous HB model with JC baseline matrix and the biallelic-favouring fitness profile $\phi=(0,0,-10,-10)$ at every site (for more details, see \Cref{app:substitution-model-computation}).

The distance distribution depends on the fitness only through the collision moments $c_2(\rho_{m,t})$: it records how variable each site is, but does not distinguish two fit alleles from several fit alleles with the same collision probability.
Both models are homogeneous, so their distance distributions (\Cref{fig:substitution-model-invariants}, second column) are binomial at all times and differ only in the mismatch probability $1-c_2(\rho_t)$.
In the stationary limit $t=\infty$, the biallelic-favouring profile $\phi$ gives $c_2(\rho_\phi)\approx\frac12$, so the distance distribution of the HB model is centred near $L/2$ rather than near $\frac34L$ as for the JC model.
In contrast, the distribution of hyperbolicity deficits (\Cref{fig:substitution-model-invariants}, third column) depends on all three collision moments $c_2(\rho_{m,t})$, $c_3(\rho_{m,t})$, and $c_4(\rho_{m,t})$, and hence is sensitive to the fitness distributions at each site.
For example, by \Cref{prop:half-integer-suppression} its half-integer part is bounded by the total triallelic accessibility $\sum_{m=1}^L\theta(\rho_{m,t})$, which suppresses half-integer hyperbolicity deficits.
For the HB model, whose fitness profile favours two alleles, this produces a sawtooth pattern in which the half-integer bars lie well below the neighbouring integer bars, whereas for the JC model integer and half-integer values follow a single smooth profile.

\subsection{SARS-CoV-2 sequence data}
\label{subsec:sars-cov2}

\begin{figure}
	\centering
	\includegraphics[width=\textwidth]{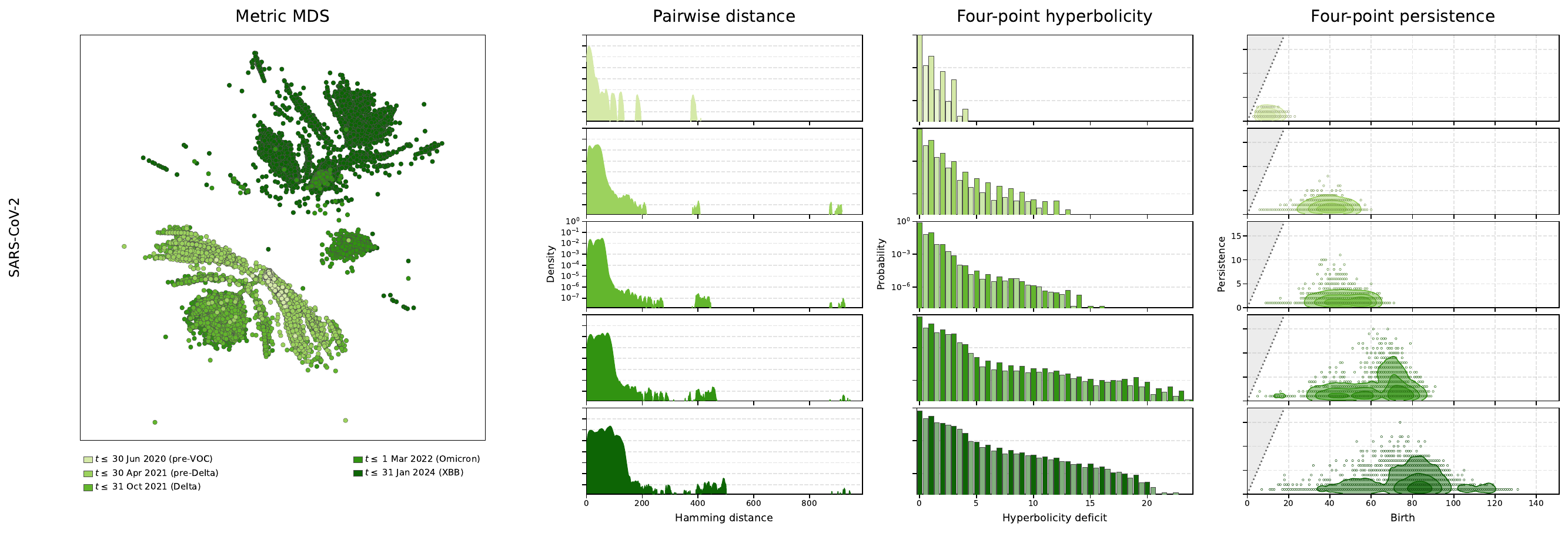}
	\caption{Geometry of the cumulative SARS-CoV-2 data, one row per collection-date cutoff, earliest at the top.
	The leftmost panel shows a $2$D metric-MDS embedding of a subsample of $50,000$ sequences from the full alignment.
	The remaining columns show the distributions of the pairwise distance $\nu_{\dist}$, the four-point hyperbolicity deficit $\nu_{\hyp}$, and the $(4,1)$-persistence measure $\Uvr{4,1}$.
	In the $(4,1)$-persistence panels, contours enclose $95\%$, $80\%$, and $50\%$ of the Gaussian-smoothed nonzero mass, while the shaded region is excluded by the $(4,1)$-persistence set of the ambient Hamming cube (\Cref{prop:D41_Hypercube}).
	The empirical curvature measures $\empirical{\mu}{4}{N}$ are obtained with the sample size $N$ required by \Cref{cor:empirical_Kn_samples} at accuracy $\epsilon=0.1$ and a combined confidence of $0.95$ across the five cutoffs, and the distributions of the three $n$-invariants are computed from them.
	}
	\label{fig:sars-cov2-invariants}
\end{figure}

We now turn from theoretical substitution models to real-world data of SARS\nobreakdash-CoV\nobreakdash-2 genome sequences collected during the COVID-19 pandemic between December 2019 and January 2024 and shared via GISAID, the global data science initiative \cite{Elbe2017} (see \Cref{app:sars-data-prep}).
The ambient metric space is the Hamming cube $H(\Sigma, L)$ over the alphabet $\Sigma=\{A,C,G,T\}$ ($k=4$) of length $L=28{,}207$.
The filtered whole-genome alignment consists of $5{,}000{,}450$ dated sequences, forming a fixed subset $X\subset H$ of $|X|=2{,}722{,}574$ distinct haplotypes.
Their full distance matrix would have about $3.7\times10^{12}$ entries and occupy more than $7$ terabytes even at two bytes per entry, so all distances are evaluated only within the sampled tuples.
Of the $L=28{,}207$ sites, $27{,}181$ are polymorphic across $X$, the remaining sites are constant and do not contribute to the Hamming distance between sequences in~$X$.

Using the collection dates of the sequences, we can construct a time series of measures over $X$.
Writing $c_{\leq t}(x)$ for the number of records of haplotype $x$ collected on or before a cutoff time $t$ and $R(t)\coloneq\sum_{x\in X}c_{\leq t}(x)$ for the total number of sequences collected by time $t$,
\begin{align}
	\mu^{\mathrm{SARS}}_{\leq t}
	&\coloneq
	\frac{1}{R(t)}\sum_{x\in X}c_{\leq t}(x)\,\delta_x.
\end{align}
We use five collection-date cutoffs, corresponding to five phases of the COVID-19 pandemic:
30~June~2020 (early outbreak, pre Variant of Concern),
30~April~2021 (Alpha-dominant with Beta and Gamma co-circulating, pre-Delta),
31~October~2021 (Delta-dominated, pre-Omicron),
1~March~2022 (the Omicron BA.1/BA.2 wave, pre-recombinant),
and 31~January~2024 (the Omicron-recombinant and XBB era, the full dataset).

The measure $\mu^{\mathrm{SARS}}_{\leq t}$ at each collection-date cutoff is a frequency-weighted cumulative sample of the observed SARS\nobreakdash-CoV\nobreakdash-2 sequences, rather than an instantaneous distribution at time $t$.
If $\mu_t^{\mathrm{SARS}}$ denotes an idealized instantaneous distribution and $\lambda(t)$ the sampling intensity at time $t$, then its population-level cumulative counterpart would correspond to the weighted average $\mu^{\mathrm{SARS}}_{\leq t} = \frac{1}{\Lambda(t)}\int_0^t \lambda(s)\mu_s^{\mathrm{SARS}}\,ds$, where $\Lambda(t)\coloneq\int_0^t\lambda(s)\,ds$ is the total number of samples collected by time $t$.

At each collection-date cutoff $t$, we estimate the curvature measure $\mu_4$ of $(X, d_H, \mu^{\mathrm{SARS}}_{\leq t})$ from a single i.i.d.\ sample of $N \approx 5.3\times10^{7}$ quadruples and compute all three distributions from it.
The hyperbolicity deficit and the $(4,1)$-persistence are evaluated on each quadruple directly, while for the distance we use that the restriction $\pi_{12} \colon \Kn_4 \to \Kn_2$ to the first two points is $1$-Lipschitz and pushes $\mu_4$ forward to $\mu_2$, so by \Cref{prop:invariant-postcomposition} the $4$-invariant $\dist \circ \pi_{12}$ has distribution $\nu_{\dist}$.
By \Cref{cor:empirical_Kn_samples} with $\epsilon = 0.1$ and confidence $1-0.05/15$, the sample size ensures $\W_1(\empirical{\mu}{4}{N}, \mu_4) \leq 0.1\diam(X)$ and hence, by \Cref{rmk:uniform-sample-reuse}, $\W_1(\empirical{\nu}{F}{N}, \nu_F) \leq 0.1\operatorname{Lip}(F)\diam(X)$ for every Lipschitz $4$-invariant $F$ evaluated on the same sample.
The $(4,1)$-persistence is $1$-Lipschitz (\Cref{prop:lipschitz-constants}), which gives accuracy $0.1\diam(X)$ for $\Uvr{4,1}$.
The hyperbolicity deficit and the distance are real-valued with range of diameter at most $\diam(X)$ (see \Cref{inv:hyperbolicity}), so at the same sample size and confidence, \Cref{prop:empirical_real_invariant} shows that $\nu_{\hyp}$ and $\nu_{\dist}$ are estimated much more accurately than $\mu_4$, to within $5.6\times10^{-4}\diam(X)$.
The confidence $1-0.05/15$ is chosen so that, using the standard union bound, we obtain an overall confidence of $0.95$ in these accuracies of the three distributions over the five cutoffs.
\Cref{fig:sars-cov2-invariants} shows the resulting distributions, with one row per cutoff from the earliest (top) to the full dataset (bottom).

\paragraph{Pairwise distance.}
The mean of the mismatch distribution $\nu_{\dist}$ (\Cref{fig:sars-cov2-invariants}, second column), and with it the nucleotide diversity of $\mu^{\mathrm{SARS}}_{\leq t}$, grows monotonically across the five cutoffs, from $10$ substitutions at 30~June~2020 to $62$ at the full dataset.
The variance grows faster than the mean, the variance-to-mean ratio rising from $2.7$ to $14$.
The distribution is multimodal, with additional prominent modes appearing in each successive cutoff.
A thin tail of large distances, visible as small isolated bumps on the right, persists at every cutoff (sampled maximum $989$), but its mass remains below $30$ parts per million (ppm) at all five cutoffs.

\paragraph{Four-point hyperbolicity deficit.}
The hyperbolicity deficits (\Cref{fig:sars-cov2-invariants}, third column) vanish on the majority of sampled quadruples throughout the pandemic, with probability $\bbP(\hyp=0)$ falling over time from $98\%$ to $55\%$.
The mean deficit grows from $0.015$ to $0.60$, and the sampled maximum from $4$ at the first cutoff to $24$ at the fourth, providing a lower bound of $\hyp(X)\geq 24$ for the Gromov hyperbolicity of the full dataset.
Half-integer hyperbolicity deficits remain suppressed throughout, visible as the lower bars of the alternating pattern, and grow from $0.2\%$ to $20\%$ of the total mass across the cutoff dates, still significantly below the uniform Hamming cube reference of $50\%$.

\paragraph{$(4,1)$-persistence.}
The $(4,1)$-persistence measure $\Uvr{4,1}$ (\Cref{fig:sars-cov2-invariants}, fourth column) is close to trivial: the nonzero mass ranges only from $68$ to $310$ ppm across the five cutoffs. 
Among the nonzero $(4,1)$-persistence classes, the proportion with lifetime $1$ ranges from $77\%$ to $94\%$ across the cutoffs, and the sampled maximum lifetime is $15$.
Their births concentrate near the modes of $\nu_{\dist}$, drifting upward across the cutoffs from around $9$ at the first to around $84$ at the last.

\paragraph{Discussion.}
The three distributions resolve increasingly specific features of the data: the distance distribution records how far apart genomes lie, the hyperbolicity deficit how often four genomes fail to fit on a tree, and $(4,1)$-persistence the rarer case in which such a failure produces a cycle.
A single expanding population would produce a smooth unimodal wave \cite{RogersHarpending1992}, whereas here the pairwise-distance distribution becomes increasingly multimodal.
Its upward drift and multimodality are consistent with the accumulation of genetically separated lineage clusters in the cumulative measures $\mu^{\mathrm{SARS}}_{\leq t}$, which also appear as separated groups in the metric-MDS embedding (\Cref{fig:sars-cov2-invariants}, first column).
Against this background, the increase of $\bbP(\hyp>0)$ from about $2\%$ at the beginning of the pandemic to $45\%$ at its end shows that non-tree-like four-point configurations become common as these clusters accumulate, whereas only $68$--$310$ ppm of quadruples have nontrivial $(4,1)$-persistence.
Thus, most violations of the four-point condition do not give rise to nontrivial $(4,1)$-persistence.
A positive deficit requires homoplasy or recombination (\Cref{sec:hamming-sitewise}), and a single mutation shared by one genome from each of two lineage clusters can already produce one.
The rise of $\bbP(\hyp>0)$ is thus in line with the many documented recurrences of mutations across lineages in SARS\nobreakdash-CoV\nobreakdash-2 \cite{vanDorpEtAl2020,MartinEtAl2021}.
The suppression of half-integer deficits provides additional sitewise information: by \Cref{prop:half-integer-suppression}, such a deficit can occur only if the sampled quartet is triallelic at some site.
Because their frequency remains below the uniform Hamming-cube reference, their suppression points to constrained triallelic accessibility, consistent with the strongly biased SARS\nobreakdash-CoV\nobreakdash-2 mutation spectrum \cite{Simmonds2020,DeMaioEtAl2021}.

The nontrivial $(4,1)$-persistence classes provide a more specific repeated-state signature.
By \Cref{inv:persistence-set}, such a class occurs when the four genomes can be cyclically ordered so that each adjacent pair agrees at more variable sites than either opposite pair; its birth is the longest cycle-edge distance and its death the shorter diagonal distance.
The concentration of birth times near the modes of the mismatch distribution therefore places these cycles at characteristic separations of the main lineage clusters, and their drift from around $9$ to around $84$ tracks the increasing separation of those clusters.
The simplest way to obtain such a cycle across two lineage clusters is a single site at which one genome from each cluster carries the same nucleotide and the other two do not, i.e.\ the same SNV arising independently in both lineages; any class produced this way has lifetime exactly $1$.
The predominance of lifetime-$1$ classes thus shows the between-lineage counterpart of the short-lived, within-lineage classes that were used to detect recurrent mutations in SARS\nobreakdash-CoV\nobreakdash-2 and shown to correlate with adaptation \cite{BleherEtAl2021}, in agreement with phylogenetic evidence of recurrent mutation \cite{vanDorpEtAl2020,MartinEtAl2021}.
Since the classes found at later cutoffs are born at larger scales, they involve the later, more separated lineage clusters, so such recurrences are not confined to the early pandemic.
The still rarer classes with longer lifetimes record a larger separation between cycle-edge and diagonal scales and reflect more complex repeated-state patterns, including multisite convergence and recombination; both within- and between-lineage recombination have been documented in SARS\nobreakdash-CoV\nobreakdash-2, albeit at a low recombination-to-mutation ratio \cite{TurakhiaEtAl2022}.

\section*{Acknowledgements}

We gratefully acknowledge all data contributors, i.e., the Authors and their Originating laboratories responsible for obtaining the specimens, and their Submitting laboratories for generating the genetic sequence and metadata and sharing via the GISAID Initiative, on which this research is based.
The authors acknowledge the use of de.NBI Cloud and the support by the High Performance and Cloud Computing Group at the Zentrum für Datenverarbeitung of the University of Tübingen and the German Federal Ministry of Education and Research (BMBF) through grant no 031 A535A.
M.B. is funded by the European Research Council (ERC) under the European Union's Horizon 2020 research and innovation programme (project PEPS, no.~101071786).
M.N. was supported by the Helmholtz Association under the joint research school “HIDSS4Health -- Helmholtz
Information and Data Science School for Health”.
M.N. and A.O. were supported by the Karlsruhe Institute of Technology and the Helmholtz Association
[POF4; 5207.0004.0012].
F.M. gratefully acknowledges support
from the National Science Foundation through grants CCF-2523653,
DMS-2524362 and RI-1901360.
This project was supported by the Vector Stiftung.

\section*{Code and data availability}

The code used to generate all figures and analyses in this article is available in the public GitHub repository \url{https://github.com/subthaumic/n-invariants} and archived on Zenodo under \url{https://doi.org/10.5281/zenodo.23210358} \cite{n-invariants-code}.
The computational procedures used for the theoretical substitution-model results are documented in \Cref{app:substitution-model-computation}.
The SARS\nobreakdash-CoV\nobreakdash-2 genome data analysed in this study were shared via GISAID, the global data science initiative \cite{Elbe2017}.
The findings of this study are based on $5{,}000{,}450$ records collected between 24~December~2019 and 31~January~2024 in 217 countries and territories, and accessible at \url{https://doi.org/10.55876/gis8.261006ze} (EPI\_SET\_261006ze, comprising $5{,}000{,}421$ of the records).
The download date and preparation of these data are documented in \Cref{app:sars-data-prep}.

\section*{Tool and computational resource disclosure}

During the preparation of this work, the authors used Codex (OpenAI; GPT-5 and GPT-6 series models) and Claude Code (Anthropic; Claude Opus~5 and~5.5) to improve the wording of parts of the text and to check statements against the cited literature.
These tools were also used to write scripts for generating figures and to refactor the research code into a clean Python package available in the public repository.
The authors reviewed and edited all output and take full responsibility for the content of this publication.

\bibliographystyle{alpha}
\bibliography{refs.bib}

\clearpage


\begin{appendix}

\section{Simulations of single-nucleotide substitution models}
\label[appendix]{app:substitution-model-computation}

For \Cref{fig:substitution-model-invariants}, we used the alphabet $\Sigma=\{0,1,2,3\}$, sequence length $L=100$, and the homogeneous product initial measure $\delta_0^{\otimes L}$.
The JC generator has total exit rate one, as in \Cref{sec:simulated-evolution}.
For the homogeneous HB model, every site uses the same symmetric JC baseline generator and fitness profile $\phi=(0,0,-10,-10)$.
At each finite time $t\in\{0.05,0.1,0.5,1\}$, we computed the one-site distribution $\rho_t=\delta_0\exp(tQ)$ by matrix exponentiation.
At $t=\infty$, we used the corresponding stationary distribution: the uniform distribution for JC and the distribution proportional to $\exp(\phi)$ for HB.

For each model, the metric-MDS panel uses $160$ independently sampled sequences from each of the five time layers, embedded jointly from their pairwise Hamming-distance matrix.
We used one random initialization and at most $500$ metric-MDS iterations.

The pairwise-distance distributions were evaluated exactly as $\operatorname{Bin}(L,1-c_2(\rho_t))$ using \Cref{prop:product-distance-distribution}.
For the distributions of hyperbolicity deficits, we enumerated the one-site distribution $\nu_{S,\rho_t}$, computed its $L$-fold convolution by fast Fourier transform, and pushed the result forward along $\gap_+$ as in \Cref{prop:product-hyp-distribution}.
The sample count $N=117{,}207{,}232$ was obtained from \Cref{cor:empirical_Kn_samples} with $n=4$, $p=2$, error $\epsilon=0.1$, Lipschitz constant $1$, and per-distribution confidence $\alpha=0.995$; by a union bound, the latter provides at least $0.95$ simultaneous confidence across the ten model--time distributions.
For each model and time, the $(4,1)$-persistence measure was then estimated from $N$ independent ordered quadruples whose sequences were sampled with replacement from $\rho_t^{\otimes L}$.
Rather than materializing the sampled sequences, we sampled multinomial counts of the one-site six-edge mismatch patterns in batches of $100{,}000$ quadruples and summed them to obtain the six Hamming distances.
For visualization, the nonzero birth--persistence mass was smoothed on its integer lattice with a Gaussian kernel of standard deviation one lattice unit, and the plotted contours enclose $95\%$, $80\%$, and $50\%$ of that smoothed mass.

All code implementing the steps described in this section is available at \url{https://github.com/subthaumic/n-invariants} \cite{n-invariants-code}.

\section{SARS-CoV-2 data availability and preparation}
\label[appendix]{app:sars-data-prep}

We used SARS\nobreakdash-CoV\nobreakdash-2 genome data shared via GISAID, the global data science initiative \cite{Elbe2017}.
We downloaded the alignment msa\_0222.fasta on 24~February~2024.
This alignment comprises 14,956,990 SARS\nobreakdash-CoV\nobreakdash-2 nucleotide sequences that were aligned to the reference sequence hCoV-19/Wuhan/WIV04/2019 (GISAID accession number EPI\_ISL\_402124) with MAFFT v7.497 \cite{Katoh2002}.

\begin{enumerate}[topsep=0.1ex, itemsep=0.1ex, parsep=0.1ex]
	\item \emph{Coverage and date filter.} Every sequence containing a character other than A, C, G, T or the gap symbol~`-', or lacking a complete YYYY-MM-DD collection date in its header, was removed. This retains exactly the completely covered, dated genomes.
	\item \emph{Gap removal.} Since the distances between indels and nucleotides are ambiguous, we removed all gap symbols~`-' from the alignment. We first discarded every alignment column in which more than $0.05\%$ of the sequences carried a gap. This removes the positions of common insertions and deletions from the alignment, leaving only the core genome. We then removed the residual sequences that still carried a gap in any of the retained columns. The result is a fixed-width alignment over $\{A,C,G,T\}$, containing only the positions on the genome that have not been subject to a widespread insertion or deletion event.
\end{enumerate}

\enlargethispage{2\baselineskip}

After these steps, $5{,}000{,}450$ complete records remain, which reduce to $U=2{,}722{,}574$ distinct haplotypes.
The findings of this study are based on these records, collected between 24~December~2019 and 31~January~2024 in 217 countries and territories, and accessible at \url{https://doi.org/10.55876/gis8.261006ze} (EPI\_SET\_261006ze, comprising $5{,}000{,}421$ of the records).
These unique haplotypes provide a subset $X \subset H(\Sigma, L)$ of the Hamming cube over the alphabet $\Sigma=\{A,C,G,T\}$ and length $L=28{,}207$, of which $27{,}181$ sites are polymorphic.
The time series $\mu^\text{SARS}_{\leq t}$ is obtained from the frequency of the haplotypes observed up to the cutoff date $t$.
At each cutoff, the $n$-tuples are drawn independently and with replacement from $\mu^\text{SARS}_{\leq t}$, so the sampling guarantees of \Cref{sec:convergence-rates} apply to $X$ with this measure.

All code implementing the steps described in this section is available at \url{https://github.com/subthaumic/n-invariants} \cite{n-invariants-code}.

\end{appendix}

\end{document}